\documentclass[12pt,article]{amsart}
\usepackage{geometry}
\usepackage{amsmath,latexsym,amssymb,amsthm,graphicx}
\usepackage{mathtools}
\usepackage{caption}
\usepackage{subcaption}
\usepackage{xfrac}
\usepackage{fix-cm}
\usepackage{accents}
\usepackage{multicol}
\usepackage{color,transparent}
\usepackage{verbatim}

\usepackage[bookmarks=true,%
    colorlinks=true,%
    linkcolor=blue,%
    citecolor=blue,%
    filecolor=blue,%
    menucolor=blue,%
    urlcolor=blue]{hyperref}
\usepackage{stmaryrd}
\usepackage{multirow}
\usepackage{float}
\usepackage{array}
\usepackage[shortlabels]{enumitem}
\usepackage{pdfsync}
\usepackage{tikz}
\usepackage{tikz-cd}
\usepackage{float}
\usepackage[T1]{fontenc}
\usepackage[utf8]{inputenc}
\usepackage[textsize=small]{todonotes}

\usepackage{hyphenat}
\usepackage{booktabs}

\graphicspath{{mpdraws/}{draws/}}

\setenumerate[1]{label=(\arabic*)}

\numberwithin{equation}{section}
\numberwithin{figure}{section}

\newtheorem{mainthm}{Theorem}

\newtheorem*{maindef}{Definition}
\newtheorem{theorem}[equation]{Theorem}
\newtheorem{lemma}[equation]{Lemma}
\newtheorem{claim}[equation]{Claim}
\newtheorem{corollary}[equation]{Corollary}
\newtheorem{proposition}[equation]{Proposition}

\newtheorem*{ex*}{Exercise}

\theoremstyle{definition}
\newtheorem{remark}[equation]{Remark}

\newtheorem{example}[equation]{Example}
\newtheorem{definition}[equation]{Definition}

\theoremstyle{plain}

\newcommand\N{\ensuremath{\mathbb{N}}}
\newcommand\dev{\ensuremath{\mathit{dev}}}
\newcommand\Z{\ensuremath{\mathbb{Z}}}

\newcommand\R{\ensuremath{\mathbb{R}}}

\newcommand\B{\ensuremath{\mathcal{B}}}
\newcommand*{\Bg}[1]{\ensuremath{\mathcal{B}^{#1}}}
\newcommand\G{\ensuremath{\mathcal{G}}}

\DeclareMathOperator{\Curr}{Curr}
\DeclareMathOperator{\Teich}{Teich}

\DeclareMathOperator{\sys}{sys}

\newcommand{\del}{\partial}

\newcommand*{\wt}[1]{\widetilde{#1}}

\renewcommand{\epsilon}{\varepsilon}

\DeclareMathOperator{\SL}{\mathit{SL}}
\DeclareMathOperator{\PSL}{\mathit{PSL}}

\newcommand{\capt}{\cap\kern-0.7em|\kern0.7em}

\newcommand{\Curves}{\mathit{Curves}}

\newcommand{\ML}{\mathcal{ML}}
\newcommand{\CAT}{\mathit{CAT}}

\DeclareMathOperator{\supp}{supp}

\usetikzlibrary{arrows}
\tikzset{
    labl/.style={anchor=south, rotate=90, inner sep=.5mm}
}
\tikzstyle{every picture}=[> = to]
\tikzset{cdlabel/.style={execute at begin node=$\scriptstyle,execute at end node=$}}
\tikzset{implication/.style={double equal sign distance, -implies}}
\tikzset{biimplication/.style={double equal sign distance, implies-implies}}

\begin{document}
\title{Dual spaces of geodesic currents}

\author[De Rosa]{Luca de Rosa}
\address{
Department of Mathematics\\
         ETH Zurich\\
R\"amistrasse 101, 8092 Zurich\\
Switzerland}
\email{luca.derosa@math.ethz.ch}

\author[Martínez-Granado]{Dídac Martínez-Granado}
\address{Department of Mathematics\\University of Luxembourg\\Av. de la Fonte 6, Esch-sur-Alzette, L-4364, Luxembourg}       
\email{didac.martinezgranado@uni.lu}

\begin{abstract}
Every geodesic current on a hyperbolic surface has an associated dual space. If the current is a lamination, this dual embeds isometrically into a real tree. We show that, in general, the dual space is a Gromov hyperbolic metric tree-graded space, and express its Gromov hyperbolicity constant in terms of the geodesic current.  In the case of geodesic currents with no atoms and full support, such as those coming from certain higher rank representations, we show the duals are homeomorphic to the surface.
We also analyze the completeness of the dual and the properties of the action of the fundamental group of the surface on the dual. 
Furthermore, we compare two natural topologies in the space of duals. 

\end{abstract}
\maketitle
\tableofcontents

\section{Introduction}
\label{sec:intro}
An $\mathbb{R}$-tree is a geodesic metric space where any two points are connected by a unique arc isometric to a closed interval in $\mathbb{R}$.
A measured lamination $\lambda$ on a surface $X$ defines a (dual) $\pi_1(X)$-action on an $\mathbb{R}$-tree as follows. Lift the lamination to the universal cover $\wt{X}$ and define the pseudo-distance between two points by considering the measure of the set of geodesics intersecting the geodesic segment connecting them.
This turns out to define a $0$-hyperbolic space $X_{\lambda}$, that embeds isometrically into an $\mathbb{R}$-tree, called the tree dual to the measured lamination.
There are several equivalent formulations of this construction: see ~\cite{MORGAN1991143,W98:Duals,Kap00:Kapovich2000HyperbolicMA}. We explore the equivalence of these with our construction in Subsection~\ref{subsec:dual_tree}.
It follows from the construction of the dual $X_{\lambda}$ of a measured lamination that the translation length of an element $g \in \pi_1(X)$ is equal to the intersection number of the measured lamination with $[g]$, the homotopy class represented by $g$.

First introduced by Bonahon in his seminal paper \cite{Bonahon86:EndsHyperbolicManifolds}, geodesic currents can be understood as an extension of measured laminations where the geodesics in the support of the measure are allowed to intersect each other. 

The above construction of dual space of a lamination can be extended to any geodesic current $\mu$ on a compact hyperbolic surface $X$ (possibly with non-empty geodesic boundary).

\begin{maindef}[Dual space of a geodesic current]
A geodesic current $\mu$ induces a pseudo-distance on $\wt{X}$ given by
 \[
 d_\mu (\overline{x},\overline{y}) = \frac{1}{2} \left\{ \mu ( G[\overline{x},\overline{y})) + \mu ( G(\overline{x},\overline{y}] )\right\}.
 \]
 where $G[\overline{x},\overline{y})$ denotes the set of hyperbolic geodesics of $\wt{X}$ transverse to the geodesic segment $[\overline{x},\overline{y})$.
 The dual space of $\mu$, denoted by $X_{\mu}$, is defined as the metric quotient of on $\wt{X}$ under this this pseudo-distance. The set of all dual spaces of geodesic currents will be denoted by $\mathcal{D}(X)$.
 \end{maindef}

 This space was introduced by Burger-Iozzi-Parreau-Pozzetti in~\cite{BIPP21:Currents}. We note that $X_{\mu}$ depends on the choice of hyperbolic structure $X$ (see Subsection~\ref{subsec:dependence}).
 $X_{\mu}$ might not, a priori, be $0$-hyperbolic, but we show that, in fact, it is always $\delta$-hyperbolic for some $\delta \geq 0$ that can be described in terms of $\mu$.
If $B$ denotes a \emph{box} of geodesics in $\wt{X}$ given by a product of two intervals on $\partial \wt{X}$, $B^{\perp}$ denotes the \emph{opposite box} given by the complementary intervals, and $\mathcal{B}$ the family of all boxes $B$, we prove the following result.

\begin{mainthm}[Hyperbolicity]\label{thm:hyp_intro}
Let $\mu$ be a geodesic current on $X$, and let 
\[
\delta_\mu = \sup_{B \in \mathcal{B}} \min {\{ \mu(B), \mu(B^ \perp) \}}.
\]
Then the dual space $X_\mu$ is a $\delta_\mu$-hyperbolic space, and $\delta_\mu$ is the optimal hyperbolicity constant.
\end{mainthm}
It follows from this that $X_{\mu}$ is $0$-hyperbolic if and only if $\mu$ is a measured lamination (see Corollary~\ref{cor:0hyp}).
Theorem~\ref{thm:hyp_intro} is stated as Theorem~\ref{prop:dualishyp}.

For the following, compare Figure~\ref{decomposition2}.

    \begin{figure}[h!]
\centering{
\resizebox{100mm}{!}{\Huge{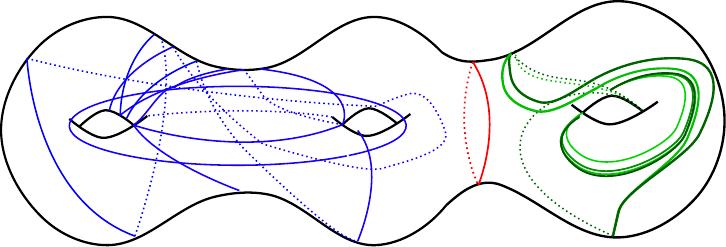}}
    \caption{  The figure shows a sketch of a geodesic current $\mu$ with two components $\mu_1$ (type 1, blue, left) and $\mu_2$ (type 2, green, right), separated by a simple multi-curve of one single component (type 3, red, middle). } \label{decomposition2}}
\end{figure}

Even though, by the above, $X_{\mu}$ is not in general isometric to an $\mathbb{R}$-tree, it has a structure that resembles that of an $\mathbb{R}$-tree, as follows. 
In~\cite{BIPP21:Currents}, Burger-Iozzi-Parreau-Pozzetti showed that there exists a simple multi-curve $m$ associated to $\mu$, called \emph{special multi-curve}, given by disjoint simple closed geodesics $(s_j)_{j=1}^k$, that decompose $X$ into sub-surfaces with geodesic boundary $X_1, \cdots, X_n$ such that the geodesic current $\mu$ on $X$ can be written as a sum  
\[
\mu=\sum_{i=1}^n \mu_i + \sum_{j=1}^k a_j s_j,
\] where $a_j \in \R$ and each $\mu_i$ is supported on $X_i$, and where the weights $a_j$ could be $0$. Moreover, for each $i$ such that $\mu_i \neq 0$, precisely one of the following holds:
\begin{itemize}
    \item $\mathrm{sys}_{X_i}(\mu_i) > 0$ (\emph{type 1});
    \item the current $\mu_i$ is a non-discrete measured lamination (\emph{type 2}).
\end{itemize}

The \emph{systole}
of a current $\mu$ on a surface with possibly non empty boundary $X$ is given by 
\[
\mathrm{sys}_X (\mu) = \inf_{c} i (\mu, c)
\]
where the infimum is taken over all closed geodesics contained in the interior of $X$. A non-discrete measured lamination is one without closed components (see Subsection~\ref{subsec:filling_boundary} for details, and Section~\ref{sec:decomposition} for more on decomposition).

Parallel to this result, we 
obtain a decomposition theorem for the dual space $X_{\mu}$ of $\mu$ as a (metric) tree graded space.
A \emph{tree graded space $X$}, in the sense of Drutu-Sapir~\cite{DS05:Treegraded}, is a geodesic metric space together with a family of distinguished subsets $\mathcal{P}$ called pieces. Intuitively, $X$ is assembled from $\mathcal{P}$ by attaching them along an $\mathbb{R}$-tree $T$ that acts as a ``central spine'', in such a way that any two pieces intersect each other at most at one point along $T$. This notion was suggested to us by A.Parreau and B. Pozzetti. 
In fact, in~\cite[Section~6]{BIPP21:PSL2xPSL2} Burger, Iozzi, Parreau and Pozzetti relate $\R^2$-tree-graded spaces to the dual of a current given as a sum of two transverse measured laminations.

Our dual spaces are not endowed with a geodesic structure in general, so we introduce the more general notion of \emph{metric} tree-graded space, and prove the following.
\begin{mainthm}[Dual structure theorem]\label{thm:intro_treegraded}
The dual space $X_\mu$ is a metric tree-graded space whose underlying tree is the dual tree of the special multi-curve $m$, and the pieces are the dual spaces $(X_i)_{\mu_i}$ of the currents $\mu_i$ on the sufaces $X_i$.
\end{mainthm}
See Theorem~\ref{thm:metric_tree_graded} and Section~\ref{sec:decomposition} for a precise statement and definitions.

Compare Figure~\ref{fig:dual_treegraded2} for a part of a sketch of a geometric realization of the dual space $X_\mu$ and a hint of its tree graded structure, where $\mu$ is the current illustrated in Figure ~\ref{decomposition2}. The Figure shows three pieces of $X_{\mu}$, two peripheral ones corresponding to the current $\mu_2$, and one, central, corresponding to $\mu_1$, as well as two edges of the tree $T$.

    \begin{figure}[h!]
\centering{
\resizebox{80mm}{!}{\includegraphics[scale=.7]{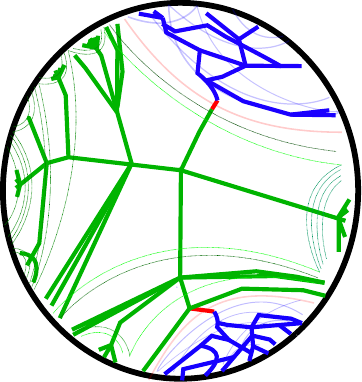}}
    \caption{  The figure shows a sketch of the dual space $X_{\mu}$ corresponding to the current in Figure~\ref{decomposition2}. The sketch is superimposed on the support of $\mu$ (the lifts of the geodesics on $X$ to the universal cover $\wt{X}$), that has been faded out so that a (geometric realization of the) dual stands out. }\label{fig:dual_treegraded2}} 
\end{figure}

Dual spaces come equipped with a natural action of $\pi_1(X)$, induced from the action of such group in the universal cover $\wt{X}$. We relate the properties of the action to the properties of the geodesic current.

\begin{mainthm}[Action]
Given any geodesic current $\mu$ on a surface $X$, the fundamental group
$\pi_1(X)$ acts by isometries on the dual space $X_{\mu}$, and it does so:
\begin{enumerate}
    \item Coboundedly. 
    \item Properly if and only if $\mu$ is filling.
    \item Freely if and only if $\mu$ has only one component in its decomposition.
\end{enumerate}
\label{thm:action_intro}
\end{mainthm}
Theorem~\ref{thm:action_intro} is stated as a series of smaller results in Section~\ref{sec:actions}.
We also study the metric completeness of the dual spaces.

\begin{mainthm}[Completeness]
Let $\mu$ be a geodesic current on $X$ with no atoms. Then the dual space
 $X_{\mu}$ is metrically complete if and only if $\mu$ has no components of type 2 in its decomposition.
 \label{thm:comple_intro}
\end{mainthm}
Theorem~\ref{thm:comple_intro} appears stated as Theorem~\ref{thm:new_completeness}, and its proof spans Section~\ref{sec:completeness}. That theorem also analyzes the atomic case, where the dual is complete if and only if it is a multi-curve. 

The dual space $X_{\mu}$ comes equipped with the natural projection map $\pi_{\mu} \colon \wt{X} \to X_{\mu}$. We study its continuity properties.

\begin{mainthm}[Continuity of projection]
Given a geodesic current $\mu$ on $X$, the projection $\pi_{\mu} \colon \wt{X} \to X_{\mu}$ satisfies:
\begin{enumerate}
\item  The projection $\pi_{\mu}$ is continuous if and only if $\mu$ has no atoms;
\item If $\mu$ has no atoms and is filling, then $\pi_{\mu}$ is closed;
\item If $\mu$ has no atoms and has full support, then $\pi_{\mu}$ is a homeomorphism.
\end{enumerate}
\label{thm:continuity_intro}
\end{mainthm}

Theorem~\ref{thm:continuity_intro} appears as a series of Propositions in Section~\ref{sec:semicontinuity}.
Theorem~\ref{thm:continuity_intro}, in conjunction with Theorem~\ref{thm:action_intro}, shows that $X_{\mu}/\pi_1(X)$ is homeomorphic to $X$
for geodesic currents coming from certain higher rank representations of $\pi_1(X)$ (precisely, for those coming from \emph{positively ratioed representations} in the sense of~\cite{MZ19:PositivelyRatioed}).
In the case a real convex projective structure $Z=\Omega/\pi_1(Z)$ on $S$, we use this result to induce an isometry between the dual space with underlying metric $Z$ and the surface $S$ equipped with the Hilbert metric coming from $Z$ (see Proposition~\ref{prop:convexproj}). 

We also study and relate two natural topologies in the space of geodesic currents $\Curr(X)$ and in the space of duals $\mathcal{D}(X)$. 
The space $\mathcal{D}(X)$ can be equipped with the equivariant Gromov-Hausdorff topology, first introduced and studied by Paulin~\cite{Pau88:Thesis}.
This
is a variation of the Gromov-Hausdorff topology that bakes in the action of a group.
On the other hand, the space of currents $\mathrm{Curr}(X)$ is naturally endowed with the weak$^*$-topology.
We prove

\begin{mainthm}[Topologies]
\label{thm:homeointroduction}
The map $\Psi \colon \Curr(X) \to \mathcal{D}(X)$, defined by $\mu \mapsto X_{\mu}$, is a continuous injection.
\end{mainthm}
Theorem~\ref{thm:homeointroduction} appears stated as Theorem~\ref{thm:homeo}.
Our theorem partially extends a result of Paulin~\cite{Pa89:Paulin1989TheGT} from the setting of $\mathbb{R}$-trees to more general $\delta$-hyperbolic spaces within the class $\mathcal{D}(X)$.

\subsection{Outline}

In Section~\ref{sec:background} we introduce geodesic currents and give some examples whose duals we will study later on in Section~\ref{sec:examples}. We also describe the weak$^*$-topology on the space of currents and give a convenient family of neighborhoods that will play a role in Section ~\ref{sec:topology}.

In Section~\ref{sec:currentdual} we introduce the dual of a geodesic current and relate it to the notion of measured wall spaces~\cite{CD17:MedianMeasuredWall}.

In Section~\ref{sec:semicontinuity} we show that the natural projection map from the universal cover of the surface to the dual space is continuous when the current has no atoms, and a homeomorphism when the current is non-atomic and has full support. We also show that if $\mu$ has atoms, the projection map is neither lower nor upper-semicontinuous.

 In Section~\ref{sec:examples} we explore other natural examples of such duals other than the $\mathbb{R}$-trees coming from measured laminations (which are discussed in detail in Section~\ref{subsec:dual_tree}). For example, we study the dual space of two intersecting measured laminations, and relate it to the concept of core of trees previously introduced by Guirardel in~\cite{Gui05:Coeur}. 
 
 We also show how for geodesic currents coming from certain Anosov representations, known as positively ratioed, such as strictly convex projective structures, the duals are homeomorphic to the surface $X$. In fact, in the case of a Hitchin representation $\rho$ in $\PSL(3,\mathbb{R})$, the associated strictly convex projective structure $\Omega_{\rho}$ is isometric to the dual $X_{\mu_{\Omega}}$, when $X$ is the Hilbert metric on $S$.

In Section~\ref{sec:dualis0hyperbolic}, we prove that duals are $\delta$-hyperbolic and moreover relate the optimal $\delta$-hyperbolicity constant to the geodesic current. Using this relation, we give inequalities between the $\delta$-hyperbolicity constants of $X_{\mu}$ and the duals of the subcurrents $\mu_i$ in its structural decomposition (in the sense of~\cite{BIPP21:Currents}).

In Section~\ref{sec:decomposition}, we prove that, using the decomposition theorem for geodesic currents proven in~\cite{BIPP21:Currents}, one can obtain a corresponding decomposition for the dual space as a tree graded space (in the sense of Drutu-Sapir~\cite{DS05:Treegraded}). Strictly speaking, we need to develop a notion of metric tree graded space, because of the lack, in general, of a geodesic structure on $X_{\mu}$.

In Section~\ref{sec:actions}, we prove that the action of the fundamental group on the dual space is cobounded, it is proper if and only if the current has no filling measured lamination components (type 2), and it is free if and only if the current has only one component in its decomposition.

In Section~\ref{sec:completeness} we prove that, for a non-atomic $\mu$, 
the dual space $X_{\mu}$ is complete if and only if $\mu$ has no components of type 2 in its decomposition. We also analyze the case where the current has atoms, and show it is complete if and only if $\mu$ is a multi-curve.

In Section~\ref{sec:topology} we relate two natural topologies in the space of duals. One the one hand, the axis topology, given in terms of translation lengths, is directly related to the weak$^*$-topology on currents. On the other hand, one can also consider the equivariant Gromov-Hausdorff topology, previously introduced by Paulin in~\cite{Pau88:Thesis}. 
We show that the weak$^*$-topology is finer than Gromov-Hausdorff topology. Explicitly, we prove that the map sending a geodesic current to its dual is a continuous injection when the space of currents is equipped with the weak$^*$-topology and the space of duals is equipped with the equivariant Gromov-Hausdorff topology. This partially generalizes work of Paulin in~\cite{Pa89:Paulin1989TheGT} for $\mathbb{R}$-trees within $\mathcal{D}(X)$. We also discuss connections with recent work of Oreg\'on-Reyes~\cite{OR22:SpaceMetric} and Jenya Sapir~\cite{Sap22:ExtensionThurston}.

Cantrell and Oreg\'on-Reyes~\cite{COR22:Manhattan} fit the notion of dual spaces of geodesic currents in the general framework of boundary metric structures. These are left-invariant hyperbolic pseudo-metrics on a non-elementary hyperbolic group satisfying the \emph{bounded backtracking property}.
Such property has also been studied independently by Kapovich and the second author in upcoming work~\cite{KMG:BBP}, where they use it to construct an extension to geodesic currents for the stable length of such actions, as well as other natural notions of length, and relate it to the concept of small action of groups on $\mathbb{R}$-trees.

\subsection{Acknowledgments}

\label{sec:acknowledgments}

We thank Marc Burger, Indira Chatterji, Michael Kapovich, Giuseppe Martone, Eduardo Oreg\'on-Reyes, Anne Parreau, Beatrice Pozzetti, Jenya Sapir and Dylan Thurston for useful conversations.
We are grateful to the referee for their invaluable feedback. Their careful reading of the previous version revealed several errors and offered insightful suggestions that enhanced the clarity of the manuscript.
The first author would like to thank Raphael Appenzeller, Benjamin Br\"uck, Matthew Cordes, Xenia Flamm and Francesco Fournier Facio for insightful conversations.
The second author acknowledges support from U.S. National Science Foundation grants DMS 1107452, 1107263, 1107367 "RNMS: Geometric Structures and Representation Varieties" (the GEAR Network), and from the Luxembourg National research Fund AFR/Bilateral-ReSurface 22/17145118, and thanks Marc Burger for his invitation to ETH Z\"urich, during which a portion of this work was done. 

\section{Background}
\label{sec:background}

Table~\ref{tab:object-notation} outlines the main notation in this paper, unless we explicitly state otherwise.

\begin{table}
    \begin{tabular}{@{}ll@{}}
      \toprule
      Notation & Meaning \\
      \midrule
         $S$ &  compact topological surface \\
         $X$ &  compact $\delta$-hyperbolic geodesic surface \\
         $\Gamma, \pi_1(X)$ & deck transformation group of $X$ \\       
         $\wt{X}$ &  universal cover of $X$ \\
         $X_{\mu}$ & dual of $\mu$ \\
         $\overline{x}, \overline{K}$ &  points/sets in $\wt{X}$ (as opposed to in $X_{\mu}$) \\ 
         $x, K$ &  points/sets in $X_{\mu}$ \\
         $\wt{X_{\lambda}}, T_{\lambda}$ & geometric realization of $X_{\lambda}$ \\
         $\G(\wt{X})$ & bi-infinite geodesics on $\wt{X}$ \\
         $\gamma$ & geodesic in $\G(\wt{X})$\\
         $\lambda$ & measured lamination\\
         $\mu$ & geodesic current\\
         $s$ & simple closed curve \\
         $c$ & closed curve \\

         $\Curves(X)$ &  (weighted) multi-curves \\
         $\ML(X)$ & measured laminations on $X$ \\
         $\Curr(X)$ & geodesic currents on $X$ \\
         $\mathbb{P}\Curr(X)$ & projective geodesic currents on $X$ \\
         $\mathcal{D}(X)$ & space of duals \\
         $G$ & general group (mostly hyperbolic) \\
      \bottomrule\addlinespace
    \end{tabular}
  \caption{Notation for the objects related to surfaces, curves,
    geodesic currents and duals.}\label{tab:object-notation}
\end{table}

In this section we introduce the basic concepts we will explore in this paper: the definition and basic properties of geodesic current, some examples of geodesic currents that will feature in this paper, and the weak$^*$-topology for geodesic currents.
This will motivate the central object of study, in Section~\ref{sec:currentdual}, the \emph{dual of a geodesic current}.

\subsection{Geodesic Currents and the Intersection Form}
\label{intro:curr}
 Let $S$ be a closed connected orientable topological surface of negative Euler characteristic and genus $g$.
 Let $p \colon \wt{S} \to S$ denote the universal cover of $S$. For any $\delta$-hyperbolic geodesic metric on $S$, we will denote by $X$ the surface $S$ equipped with said metric. $\wt{X}$ will denote $\wt{S}$ with the induced geodesic metric on $\wt{S}$.
Let $\partial \wt{X}$ denote the Gromov boundary of $\wt{X}$. This compactifies $\wt{X}$ to a closed disk, and the action of $\pi_1(X)$ on $\wt{X}$ by isometries extends to an action by homeomorphisms on $\wt{X} \cup \partial \wt{X}$~\cite[Chapter~III.H.3, 3.7~Proposition]{BH11:NonPosCurvature}.
Let $\mathcal{G}(\wt{X})$ denote the space of bi-infinite unoriented geodesics in $\wt{X}$, given by the quotient of the space of unit speed parameterized geodesics with the compact-open topology, where we furthermore forget the parameterization. Any $\gamma \in \mathcal{G}(\wt{X})$ has unique endpoints at infinity $(\gamma_+,\gamma_-)$, and the map $\partial \colon \mathcal{G}(\wt{X}) \to (\partial \wt{X} \times \partial \wt{X} - \Delta) / \mathbb{Z}_2$ sending $\gamma$ to its endpoints is a continuous, closed, $\pi_1(X)$-equivariant surjective map~\cite[Proposition~2.2]{BL18:Marked}. It need not be injective: for example, in $CAT(0)$ metrics, it can have non-degenerate intervals as fibers.
In this paper, we will restrict our attention to geodesic metrics on $\wt{X}$ that are proper, $\delta$-hyperbolic and for which the map $\partial$ is a homeomorphism. In particular, given two endpoints at infinity there is a unique element of $\mathcal{G}(\wt{X})$ with those endpoints.
We will furthermore assume that given any two points in $\wt{X}$, there is a unique (unoriented, unparameterized) geodesic segment between them.
Any negatively curved Riemannian metric satisfies these assumptions, including hyperbolic metrics. However, there are other noteworthy examples such as Hilbert metrics on $\wt{X}$ coming from endowing $S$ with the Hilbert metric obtained by quotienting a strictly convex projective domain $\Omega$ by a discrete and cocompact representation of $\pi_1(S)$ into $PGL(3,\mathbb{R})$~\cite[Page~99]{dlH93:Hilbert}. See more details in Section~\ref{sec:examples}. Unless stated otherwise, \emph{the metric $X$ will be a hyperbolic structure on $S$},  so that we can identify $X$ with the quotient $\wt{X} / \Gamma$ for some $\Gamma=\pi_1(X) \leq {\PSL(2,\R)}$ Fuchsian subgroup. 

  There are several equivalent definitions of geodesic currents. For an excellent account we refer the reader to~\cite[Chapter~3]{ES22:GeodesicCount}. We will mainly use the following one.
 \begin{definition}[Geodesic current]
 A \emph{geodesic current} on $X$ is a positive $\Gamma$-invariant locally finite Radon measure on the set $\mathcal{G}(\wt{X})$ of unoriented unparameterized bi-infinite geodesics in $\wt{X}$, which we identify by their endpoints in the boundary at infinity. Let
 \[(\del \wt{X})^{(2)}=\{ (x,y) \in (\del \wt{X})^{2} \colon x \neq y \}.\]
 We define
 \[
 \mathcal{G}(\wt{X}) \coloneqq (\del \wt{X})^{(2)}/\sim
 \]
 where $(x,y) \sim (y,x)$. We denote with $\mathrm{Curr}(X)$ the set of all geodesic currents on $X$.
 \end{definition}

We say that an interval in $\partial \wt{X}$ (or a geodesic segment $\tau$ in $\wt{X}$) is \emph{degenerate} if it is a singleton. We say it is \emph{non-degenerate} otherwise.

 We will consider two types of subsets of geodesics: boxes and transversals.

\begin{definition}[Box of geodesics]
Let $I_{a,b}$ denote a \emph{generalized ordered interval} in $\partial \wt{X}$, i.e., any of the following non-empty and possibly degenerate intervals: $(a,b),[a,b),(a,b], (a,b)$, where $a,b$ are ordered counter-clockwise in $\partial \wt{X}$. 

We define a \emph{box of geodesics} as any subset of $\mathcal{G}(\wt{X})$ of the type $B=I_{a,b} \times I_{c,d}$.
Let $\mathcal{B}$ denote the family of all boxes of geodesics.
\label{def:boxes}
\end{definition}

 \begin{definition}[Transversal of geodesics]
 Given a geodesic segment $\tau$ in $\wt{X}$ (which could be degenerate), $G(\tau)$ denotes the subset of geodesics in $\mathcal{G}(\wt{X})$ intersecting $\tau$ transversely, i.e., all $\gamma \in \mathcal{G}(\wt{X})$ so that $\gamma \cap \tau = \{ p \}$, where $p \in \tau$. Given any subset $A$ of $\wt{X}$ which contains at least one non-degenerate geodesic segment, $G(A)$ denotes the set of geodesics $\gamma \in \mathcal{G}(\wt{X})$ intersecting transversely at least one non-degenerate geodesic segment $\tau \subseteq A$. We will refer to these sets $G(\tau)$ and $G(A)$ as \emph{transversals}.
 \end{definition}

  Note that as subsets of geodesics, a transversal $G(\tau)$, for $\tau$ a geodesic segment, is not contained in one single box of geodesics, but it can be contained in a union of two boxes. Also, if $\tau$ is a non-degenerate geodesic segment, then $G(\tau)$ contains a non-degenerate box of geodesics.
 
 Geodesic currents provide a unifying generalization of notions such as simple closed curves, hyperbolic structures on $X$, and geodesic laminations, among many others. For an introduction on the subject we refer to the seminal paper \cite{Bonahon88:GeodesicCurrent} by Bonahon.

 Denote by $\mathfrak{I} \subset  \mathcal{G}(\wt{X}) \times  \mathcal{G}(\wt{X})$ the open set consisting of pairs of transversely intersecting geodesics in $\wt{X}$. 
 We endow $\mathfrak{I}$ with the subspace topology from $\mathcal{G}(\wt{X}) \times  \mathcal{G}(\wt{X})$. Notice that the diagonal $\Gamma$-action on $\mathcal{G}(\wt{X}) \times \mathcal{G}(\wt{X})$ descends to a free, properly discontinuous and cocompact action on $\mathfrak{I}$~\cite[Page~44]{ES22:GeodesicCount}, and hence the projection $\pi \colon \mathfrak{I} \to \mathfrak{I}/\Gamma$ is a topological covering. 
 
 Given two geodesic currents $\mu, \nu \in \Curr(X)$, they induce a $\Gamma$-invariant product measure $\mu \times \nu$ on $\mathcal{G}(\wt{X}) \times \mathcal{G}(\wt{X})$, and hence on $\mathfrak{I}$. This measure descends to a measure on $\mathfrak{I}/\Gamma$ via the covering $\pi \colon \mathfrak{I} \to \mathfrak{I}/\Gamma$. We will still indicate such measure  with $\mu \times \nu$, for notation simplicity. 
 
 \begin{definition}[Intersection form]
 The \emph{intersection form} evaluated on the two currents $\mu, \nu \in \Curr(X)$ is the total volume of $\mathfrak{I}/\Gamma$ in the measure $\mu \times \nu$
 \[
 i(\mu, \nu) := \left( \mu \times \nu \right) \left( \mathfrak{I}/\Gamma  \right).
 \]
 \label{def:intersection}
 \end{definition}

\subsection{Boundaries and filling currents}
\label{subsec:filling_boundary}
 In most of the paper we will assume $X$ is a closed hyperbolic surface, even though most of our results only use the cocompactness of $\pi_1(X)$ on $\wt{X}$, and so are also true for surfaces with geodesic boundary. The only time we will explicitly refer to surfaces with boundary, though, will be when working with subsurfaces.
We notice that certain notions such as Liouville currents only make sense for closed surfaces.

Given a compact hyperbolic surface $X$ with geodesic boundary, the space of \emph{internal geodesic currents} is the subspace $\Curr_0(X) \subseteq \Curr(X)$ consisting of currents not supported on lifts of boundary geodesics.
If $\mu$ is only supported on lifts of boundary parallel geodesics, we say it is a \emph{boundary geodesic current}.
By the definition of intersection number of geodesic currents (see also~\cite[Exercise~3.11]{ES22:GeodesicCount}), $i(\mu, \nu)=0$ for every non-trivial $\nu \in \Curr(X)$ if and only if $\mu$ is a boundary geodesic current.
In fact, any geodesic current can be written uniquely as a sum of a boundary current and an internal current. Moreover, the subspace of boundary currents is closed, and the subspace of internal currents is dense in $\Curr(X)$~(see the proof of all these claims in~\cite[Lemma~2.12]{EM18:ErgodicGeodesicCurrents}).
We define a \emph{filling geodesic current} as any internal current $\mu$ so that $i(\mu,\nu)>0$ for every non-trivial internal current $\nu$.

Let $p \colon \wt{X} \to X$ denote the universal covering projection, $Y$ a subsurface of $X$ with totally geodesic boundary, and suppose that $\mu$ is a geodesic current on $X$ so that $p(\supp(\mu)) \subset Y-\partial Y$. We will say $\mu$ is \emph{filling in a subsurface} $Y$ of $X$ if for every non-trivial current $\nu \in \Curr(X)$ so that $p(\supp(\mu)) \subset Y-\partial Y$, we have $i(\mu,\nu)>0$.
An example of a filling geodesic current on $X$ is a \emph{filling multi-curve}, i.e., a multi-curve $c$, so that $X-c$ is a disjoint union of topological disks and once-punctured disks.
In fact, a multi-curve $c$ is a filling multi-curve if and only if its associated geodesic current is filling in the sense of geodesic currents~\cite[Exercise~3.13]{ES22:GeodesicCount}.
Another example, if $X$ is closed, is the Liouville current (see Examples~\ref{ex:Liouville}).

\subsection{Examples of geodesic currents}
\label{subsec:examples_currents}

In this section we recall some examples of geodesic currents that will appear in the forthcoming sections.

\subsubsection{Weighted multi-curves}

Let $K \subset S$ be a smooth non-necessarily-connected
1-manifold without boundary, together with a map $c$ from $K$
to~$S$.
We say that $c$ is \emph{trivial} if the image of $c$ is contained in a disk (so it is \emph{null-homotopic}) or it is \emph{boundary parallel}, i.e. homotopic to a component of $\partial S$.
A \emph{multi-curve} is an equivalence class of maps $c$ as above where $c$ and~$c'$ are considered equivalent if they are related by 
\emph{homotopy} within the space of all maps from $K$ to~$S$ (not necessarily immersions), \emph{reparametrization} of the 1-manifold and
dropping trivial components.
If the domain $K$ of $c$ is connected, we will call $\gamma$ a \emph{curve}.

We write $\Curves(S)$ for the space of all curves on $S$.
There is a bijection between curves $c$ and conjugacy classes of $g$ and $g^{-1}$, for $g, g^{-1} \in \pi_1(X)$.
A \emph{weighted multi-curve} $c = \lambda_1 c_1 + \cdots + \lambda_n c_n$, is a multi-curve $\cup_{i} c_i$ where each component (curve) $c_i$ is equipped with a non-negative real number $\lambda_i$. The geodesic current corresponding to $c$ is the sum of weighted Dirac measures on $\mathcal{G}(\wt{X})$ supported on $\cup_i \Gamma \tilde{\gamma}_i \subseteq \mathcal{G}(\wt{X})$, where $\tilde{\gamma_i}$ is a lift of the geodesic in the class of $c_i$.

In fact, a geodesic current $\mu$ has atoms as a measure if and only if $\mu = \nu + \delta$, where $\delta$ is a non-trivial weighted multi-curve.
First, we consider the case of ``$0$-dimensional atoms'', i.e. atoms whose topological dimension in the space $\mathcal{G}(\wt{X})$ is $0$.

 \begin{lemma}[{\cite[Proposition~8.2.7]{Martelli16:IntroGeoTop}}]
 \label{martelli_discrete}
 Let $\mu \in \Curr(X)$.
 If $\mu(\{ \gamma \}) >0$ for $\gamma \in \G(\wt{X})$, then $\gamma$ is a lift of a closed geodesic.
\label{lem:discreteatoms}
\end{lemma}

In fact, $1$-dimensional atoms also come from weighted multi-curves. For $z,x,y \in \partial \wt{X}$ cyclically ordered we define a \emph{pencil} $P(z, [x,y]) \subseteq \mathcal{G}(\wt{X})$ to be the set of all geodesics with one endpoint $z$ and the other endpoint in $[x,y]$. When we just write $P(z)$ we mean the pencil consisting of all geodesics with $z$ as one of their endpoint.
We prove the following characterization.

\begin{lemma}
A geodesic current $\mu$ has an atom if and only if there exists $z \in \partial X$ so that the pencil $P(z)$ of geodesics at $z$, satisfies $\mu(P(z))>0$. Moreover, $\mu(P(z))<\infty$.
\label{lem:atompencil}
\begin{proof}
If $\mu$ has an atom $\gamma$, then $\mu(P(\gamma^+))>0$.
On the other hand, suppose that $\mu$ has no atoms, but $\mu(P(z))>0$ for some $z \in \partial X$.
Then, by~\cite[Proposition~8.2.8]{Martelli16:IntroGeoTop}, there exists a closed geodesic $c$ so that $z=\gamma_+$. Let $P(z,[x,y])$ be the set of geodesics with one endpoint at $z$ and the other within $[x,y] \subset \partial X$, where we assume that $\gamma_-\in [x,y]$. 
Note that, by the north-south dynamics of $\gamma$, $P(z)=\cup_{n \geq 0} \gamma^n P(z,[x,y]) = \gamma^n P(z,[x,y])$. Then, taking measures, by  $\pi_1(X)$-invariance, we have
$\mu(P(z)) = \mu(\gamma^n P(z,[x,y]))=\mu(P(z,[x,y]))$. Thus, by assumption it follows $\mu(P(z,[x,y]))>0$.
On the other hand, $\{(\gamma_+,\gamma_-)\} =\cap_{n \geq 0} \gamma^{-n} P(z,[x,y])$.
Thus, by continuity of measures from below (\cite[Theorem~D]{Ham50:Measure}, $\mu(\gamma)=\lim_{n} \mu(\gamma^{-n} P(z,[x,y]))$.
By assumption, $\mu(\gamma)=0$, and by $\pi_1(X)$-invariance, $\lim_{n} \mu( \gamma^{-n} P(z,[x,y]) ) = \mu(P(z,[x,y]))$. Thus, it follows $\mu(P(z,[x,y]))=0$, a contradiction.
Note that since $\mu(P(z)) = \mu(\gamma^n P(z,[x,y]))=\mu(P(z,[x,y]))$, and $P(z,[x,y])$ is a compact set of geodesics, $\mu(P(z))$ is finite.
\end{proof}
\end{lemma}

A consequence of the above lemma is the following, which will not be used in the sequel but worth recording.

\begin{lemma}
Let $\mu$ be a geodesic current, and $\gamma \in \supp \mu$.
The geodesic $\gamma$ is not an atom of $\mu$ if and only if for all $\epsilon>0$, there exists $\delta>0$ so that $\mu(N_{\delta}(\gamma))<\epsilon$, where $N_{\delta}$ is the $\delta$-neighborhood of $\gamma$ in $\G(\wt{X})$.
\label{lem:shrinkingatom}
\end{lemma}
\begin{proof}
If $\gamma$ is an atom, then the condition is obviously violated.

If $\gamma$ is not an atom, denote with $\gamma_{-}$ and $\gamma_{+}$ its endpoints, and consider the pencil $P = P \left( \gamma_-, (\gamma_+-\epsilon,\gamma_+ + \epsilon ) \right)$.
We know from Lemma~\ref{lem:atompencil} that $\mu(P) > 0$ if and only if it contains the axis of some non-zero element of $\Gamma$ which projects to a closed geodesic in $X$. Suppose that this is the case, i.e. there exists a geodesic line $l \in P$ such that $p(l)$ is a closed geodesic in $X$. Then $l$ is an atom for $\mu$, contradicting the fact that $\gamma$ is in the support of $\mu$, and is asymptotic to $l$. It follows that $\mu(P(\gamma_-, (\gamma_+-\epsilon,\gamma_+ + \epsilon ))=0$.

Now we define a sequence of boxes as follows: start with \[B_1 = N_\epsilon (\gamma) = (\gamma_{-} - \epsilon , \gamma_{-} + \epsilon) \times ( \gamma_{+} -\epsilon, \gamma_{+} + \epsilon)\] and define \[ B_n =(\gamma_{-} - \epsilon/n , \gamma_{-} + \epsilon/n ) \times ( \gamma_{+} -\epsilon, \gamma_{+} + \epsilon) \] by pinching one of the intervals of the corresponding box of geodesics so that $\cap_{n=1}^{\infty} B_n  = P(\gamma_-, (\gamma_+-\epsilon,\gamma_+ + \epsilon )$.
By continuity of measures from below, we have $\lim_n \mu(\cap_{i=1}^n B_i) \to \mu(P(\gamma_-, (\gamma_+-\epsilon,\gamma_+ + \epsilon ))=0$. For any $\epsilon>0$, if we take $\delta$ so that $N_{\delta}(\gamma) \subset B_n$, and $\mu(B_n)<\epsilon$, then we have $\mu(N_{\delta}(\gamma))<\epsilon$, as wanted.
\end{proof}

Finally, any geodesic current can be approximated by a sequence of weighted multi-curves.

\begin{proposition}[{\cite[Proposition~4.4]{Bonahon86:EndsHyperbolicManifolds}}]
The subset of geodesic currents coming from weighted multi-curves is dense with respect to the weak$^*$-topology of currents.
\end{proposition}

We discuss the weak$^*$-topology in Subsection~\ref{subsec:weakstar}.

\subsubsection{Measured laminations}

We start with the definition of measured geodesic lamination:
\begin{definition}
A \emph{geodesic lamination} $\Lambda$ is a set of disjoint simple complete geodesics in $X$, whose union is a closed subset of $X$. A \emph{transverse measure} for $\Lambda \subset X$ is a family $\lambda$ of locally finite Borel measures ${\lambda}_\alpha$ on each arc $\alpha \subset X$ transverse to $\lambda$, such that
\begin{enumerate}
\item For every $\alpha$ transverse arc, the support of $\lambda_\alpha$ is $\alpha \cap \Lambda$;
    \item If $\alpha' \subset \alpha$ is a sub-arc of $\alpha$, then the measure $\lambda_{\alpha'}$ is the restriction of $\lambda_\alpha$;
    \item For every $\alpha$ transverse arc, the measure $\lambda_\alpha$ is invariant through isotopies of transverse arcs.
\end{enumerate}

A \emph{measured geodesic lamination} is a geodesic lamination together with a transverse measure.
\label{def:lam}
\end{definition}

 It is a well-known fact (see \cite[Proposition~17]{Bonahon88:GeodesicCurrent}) that measured laminations can be embedded into the space of geodesic currents (see {\cite[Lemma~4.4]{AL17:HyperbolicStructures}}).

 In fact, the image of the above embedding is characterized as those geodesic currents $\eta$ so that $i(\alpha,\alpha)=0$.

 \begin{proposition}[{\cite[Proposition~14]{Bonahon88:GeodesicCurrent}}]
  The image of the embedding $(\Lambda,\mu) \mapsto \eta_{\mu}$ consists of geodesic currents $\eta$ so that $i(\alpha,\alpha)=0$.
  \label{prop:lamselfzero}
 \end{proposition}

When $X$ has boundary, there are multiple types of measured laminations one can consider (see~\cite[1.8]{HP92:TrainTracks}, \cite[Chapter~11]{Kap00:Kapovich2000HyperbolicMA} for several treatments). In this paper, we will only consider measured laminations whose associated geodesic currents are internal currents, so they are in $\Curr_0(X)$.
When working in a subsurface $Y \subset X$, we will only consider internal measured laminations within that subsurface, in the sense that, for $p \colon \wt{X} \to X$ the universal covering projection, we have $p(\supp(\mu)) \subset Y-\partial Y$.
We say that a measured lamination is \emph{discrete} if it is a simple multi-curve, i.e., all the leaves of the support of the lamination $\Lambda$ are simple closed geodesics. We say it is a \emph{non-discrete measured lamination} otherwise.
Non-discrete measured laminations are equivalent to type 2 subcurrents in the structural decomposition theorem for geodesic currents (see Section~\ref{sec:decomposition}). For any such measured lamination $\lambda$, there is a minimal (with respect to inclusion) subsurface $Y$ of $X$ that contains $p(\supp(\lambda))$, and so that for every internal closed curve $c$ in $Y$, we have $i(\lambda,c)>0$. Some authors choose to call these measured laminations ``filling'', but that would clash with our choice of ``filling'' for geodesic currents. Observe that a measured lamination is never filling in the sense of geodesic currents, since $i(\lambda,\lambda)=0$.

\subsubsection{Liouville current}

\begin{example}
\label{ex:Liouville}
Given a box $B= [a,b] \times [c,d] \subset \mathcal{G}(\wt{X} )$, the Liouville current can be explicitly defined as follows.
Consider the hyperbolic cross ratio on the upper half space $\mathbb{H}^2$, defined by taking, for any box of geodesics $B=[a,b] \times [c,d]$ in $\mathcal{G}(\mathbb{H}^2)$, the expression
\[
L(B) = \left| \log \frac{|a-c||b-d|}{|a-d||b-c|} \right|.
\]
Let $[(Y,\varphi)] \in \Teich(X)$, where $Y$ is some hyperbolic structure on $X$, a  priori distinct from $X$.
Since $Y$ is a hyperbolic structure, we have a $\pi_1(Y)$-invariant isometry $I \colon \mathbb{H}^2 \to \wt{Y}$.
We consider the following measure on $\mathcal{G}(\wt{Y})$, given by $\mathcal{L}_Y \coloneqq I_* (L)$.
From the definition of $\mathcal{L}_Y$ and the intersection number of geodesic currents, one can check the following property
\[
i(\mathcal{L}_Y, c)=\ell_{Y}(g)
\]
where $\ell_Y$ denotes the hyperbolic length (see~\cite[Proposition~14]{Bonahon88:GeodesicCurrent}).
In fact, $\mathcal{L}_Y$ is characterized by this property, by~\cite[Th\'eor\`eme~2]{Otal90:SpectreMarqueNegative}.
We will call this the \emph{intersection property}.

A stronger (a priori) property of 
$\mathcal{L}_Y$, which also follows from its definition, and fully characterizes $\mathcal{L}_Y$ (see \cite[Proposition 8.1.12]{Martelli16:IntroGeoTop}), is the following. 
\begin{definition}
   Let $X$ be a compact surface of genus $g \geq 2$ endowed with a length geodesic metric $Y$. A geodesic current $\mu \in \Curr (X)$ satisfies the \emph{Crofton property} if
    \[
    \mu (G [x,y]) = \ell_{\wt{Y}} ([x,y])
    \]
    for every $x, y \in \wt{Y}$, where $[x,y]$ denotes a geodesic from $x$ to $y$, and $\ell_{\wt{Y}}$ is the lift of $\ell_Y$ to the universal cover. 
\end{definition}

Such property is named after Crofton since it is a special case of the Crofton formula for integral geometry ~\cite[19]{Sa04:Integral}. 

Now, since
$\varphi \colon X \to Y$ is a quasi-conformal marking from the base hyperbolic structure $X$ to another hyperbolic structure $Y$, we can define a geodesic  current $\mathcal{L}_Y^X$ in $\Curr(X)$ as follows.
Since the marking $\varphi$ induces a $\pi_1(X)$-equivariant homeomorphism $\varphi \colon \wt{X} \to \wt{Y}$, we put
\[
\mathcal{L}_Y^X \coloneqq \varphi^{-1}_* \mathcal{L}_Y =  \varphi^{-1}_* \circ I_* (L).
\]

The geodesic current $\mathcal{L}_Y^X$  has full support and has no atoms, and it is defined as the \emph{Liouville current} associated to $[(Y,\varphi)]$ (see~\cite[Page~145]{Bonahon88:GeodesicCurrent}).

Otal, in~\cite[Page~155]{Otal90:SpectreMarqueNegative}, extended the construction of Liouville current $\mathcal{L}_Y$ to any negatively curved Riemannian metric $Z$ on $X$ (not necessarily of constant curvature $-1$). Otal's current, $\mathcal{L}_Z$, also satisfies the Crofton property
\[
\mathcal{L}_Z(G[x,y])=\ell_{Z}([x,y])
\]
for any $Z$-geodesic segment $[x,y]$ in $\wt{Z}$, where here $G[x,y]$ denotes the set of $Z$-geodesics intersecting $[x,y]$ transversely. 
\end{example}

\subsubsection{Geodesic currents coming from Anosov representations}
\label{subsubsec:positivelyratioed}
Let $G$ be a real, connected, non-compact, semisimple, linear Lie group. Let $K$ denote a maximal compact subgroup of $G$, 
so that $V=G/K$ is the Riemannian symmetric space of $G$. Let $[P]$ be the conjugacy class of
a parabolic subgroup $P \subset G$. 
Then there is a notion of \emph{$[P]$-Anosov representation} $\rho\colon \pi_1(X) \to G$; see, for example, Kassel's notes \cite[Section~4]{K18:GeometricStructures}.
When $\operatorname{rank}_{\mathbb{R}}(G)=1$ there is essentially one
class $[P]$, so we can simply refer to them as Anosov representations,
and they can be defined as those injective representations
$\rho\colon \pi_1(X) \to G$ where $\Gamma\coloneqq \rho(\pi_1(X))$
preserves and acts cocompactly on some nonempty convex
subset of~$X$. 
Examples of these are Fuchsian and quasi-Fuchsian representations.
In general rank, the conjugacy classes of parabolic subgroups of~$G$ correspond to
subsets~$\theta$ of the set of restricted simple roots $\Delta$ of $G$. 
For a given $[P]$-Anosov representation and each $\alpha \in \theta$, 
one can define a curve functional on oriented curves
\[
\ell_{\alpha}^{\rho}\colon \mathcal{C}(X) \to \mathbb{R}_{\geq 0}
\]
by considering the $\log$ of the diagonal matrix of eigenvalues of $\rho(g)$ and composing it with $\alpha + i(\alpha)$, where $\alpha \in \Delta$ is a root, and 
$\iota(\alpha)$ denotes the root obtained by acting by the negative of the largest element in the Weyl group. See~\cite[Section~2]{MZ19:PositivelyRatioed} for details.
Martone and Zhang
show that for a certain subset of Anosov representations called \emph{positively ratioed}~\cite[Definition~2.21]{MZ19:PositivelyRatioed}, there exists a geodesic current $\mu_{\rho}$ so that
\[
i(\mu_{\rho},[g])=\ell^{\rho}_{\alpha}(g)
\]
for all $g \in \pi_1(X)$. 
The construction goes through interpreting geodesic currents as generalized positive cross-ratios, an observation that was already used by~Hamenstaedt~\cite[Lemma~1.10]{Ham97:Cocycles}, \cite[Section~2]{Ham99:Cocycles}.
This class includes two types of representations of interest: Hitchin representations and maximal representations. In this paper we will only consider  \emph{Hitchin representations}, i.e., a representation $\rho \colon \pi_1(X) \to \SL(n,\mathbb{R})$ which 
may be continuously deformed to a composition of the irreducible representation of $\PSL(2,\mathbb{R})$ into $\PSL(n,\mathbb{R})$ with a discrete faithful representation of $\pi_1(X)$ into $\PSL(2,\mathbb{R})$.

Continuity of the cross-ratio is crucial in Martone-Zhang's construction of positive cross-ratios.
From the geodesic current viewpoint, it translates into the fact that their associated geodesic currents have no atoms.
In fact, the following can be extracted from~\cite[Page~17]{MZ19:PositivelyRatioed}).

\begin{lemma}
For $\rho \colon \pi_1(X) \to G$ a positively ratioed Anosov representation, the associated geodesic current $\mu_{\rho}$ is non-atomic and has full support.
\label{lem:anosovnoatomfull}
\end{lemma}

Recently, Burger-Iozzi-Parreau-Pozzeti~\cite[Proposition~4.3]{BIPP21:Crossratios} have lifted the continuity assumption in the generalized cross-ratio, 
thus extending the construction of such currents beyond positively ratioed representations: see
\cite{BP21:Hk} and~\cite{BP21:Theta}. Their associated currents can, in general, have atoms.

In Subsection~\ref{ex:positivelyratioed}, we discuss in more detail the case of Hitchin representations for $\SL(3,\mathbb{R})$, their connection to convex projective structures and their associated dual spaces.

\subsection{Weak$^*$-topology of currents}
 \label{subsec:weakstar}
 As a space of Radon measures on $\mathcal{G}(\wt{X})$, it is natural to endow the space of geodesic currents $\Curr(X)$ with the \emph{weak$^*$-topology on geodesic currents}, defined by the family of
semi-norms
\[
| \alpha |_{\xi} = \int_{\mathcal{G}(\wt{X})} \xi \alpha
\]
for $\alpha \in \Curr(X)$, as $\xi$ ranges over all continuous function $\xi \colon \mathcal{G}(\wt{X}) \to  \mathbb{R}$ with compact support.
The space $\Curr(X)$ is second countable and completely metrizable~(see~\cite[Proposition~A.9]{ES22:GeodesicCount}.
Thus, the topology can be specified via sequential convergence.

The intersection number $i \colon \Curr(X) \times \Curr(X) \to \mathbb{R}$ is continuous with respect to this topology (see~\cite[Proposition~4.5]{Bonahon86:EndsHyperbolicManifolds}).

In fact, the weak$^*$ topology coincides with the topology of intersection numbers, by~\cite[Theorem~11]{DLR10:DegenerationFlatMetrics}, which essentially follows by work of Otal in~\cite[Th\'eor\`eme~2]{Otal90:SpectreMarqueNegative}.

\begin{theorem}
A sequence of geodesic currents $(\mu_i)$ converges $\mu_i \to \mu$ in the weak$^*$-topology if and only if $i(\mu_i,c) \to i(\mu,c)$ for all closed curves $c$.
\label{thm:weakintersection}
\end{theorem}

\subsection{Systole of a geodesic current}
\label{subsec:systole}

Given a geodesic current $\mu$ on $X$, we define the \emph{systole of $\mu$} as 
\[
\sys(\mu) \coloneqq \inf \{ i(\mu,c) : c \in \Curves(X)\}
\]

We point out that, as a function on geodesic currents with the weak$^*$-topology, $\sys$ is a continuous function (see~\cite[Corollary~1.5(1)]{BIPP21:Currents}).

Given a subsurface $Y$ of $X$ with totally geodesic boundary and $\mu$ a geodesic current with $p(\supp(\mu)) \subset Y$, we define the \emph{systole of} $\mu$ \emph{relative to} $Y$, as follows,
\begin{equation*}
\mathrm{sys}_{Y} (\mu) := \inf \{ i (\mu, c) : c \in \Curves(Y - \partial Y)\}. 
\end{equation*}

 \section{Dual space of a geodesic current}
\label{sec:currentdual}

In this section we define and prove the basic properties of the dual space of a current.

We start by recalling some facts about pseudo-metric spaces.
 A \emph{pseudo-metric} on a set $Y$ is a map $d \colon Y \times Y \to \mathbb{R}$ satisfying the symmetry and triangle inequalities. Points $x \neq y$ with $d(x,y)=0$ are allowed.
 A pseudo-metric space $Y$ with pseudo-metric $d$ has a canonical metric space quotient $Y / \sim$. It is given by the equivalence classes for the equivalence relation identifying $x$ and $y$ in $Y$ if and only if $d(x,y)=0$, and endowed with the induced metric coming from $d$. We call this the \emph{metric quotient} of $Y$. 
  In most of this paper, we will let $Y=\wt{X}$ be the universal cover of the surface $X$ equipped with the pullback metric on $X$, and the pseudo-distance will be defined from a geodesic current as defined below.
  We will decorate the points in $\wt{X}$ with an overline, as in $\overline{x}$. 

\subsection{The dual space of a geodesic current}
A geodesic current $\mu \in \Curr(X)$ induces a pseudo-distance on $\wt{X}$ given by
 \[
 d_\mu (\overline{x},\overline{y}) = \frac{1}{2} \left\{ \mu ( G[\overline{x},\overline{y})) + \mu ( G(\overline{x},\overline{y}] )\right\}.
 \]
 
 Note that the pseudo-distance $d_\mu$ is \emph{straight} (see \cite[Proposition 4.1]{BIPP21:Currents}, in the sense that it is additive on hyperbolic geodesic lines.
 Precisely, let $\overline{x}, \overline{y}, \overline{z} \in \gamma$ be three points lying in the order $\overline{x} < \overline{y} < \overline{z}$ on a hyperbolic geodesic $\gamma \subseteq \wt{X}$, then we have
 \[
 d_\mu (\overline{x}, \overline{z}) = d_\mu (\overline{x}, \overline{y}) + d_\mu (\overline{y},\overline{z}).
 \]
  A non-straight version of this pseudo-distance was first considered by Glorieux in~\cite{Glo17:CriticalExponents}.
 \begin{definition}[Dual space of a geodesic current]
 \label{def:currentdual}
 For $\overline{x}, \overline{y} \in \wt{X}$,  consider the equivalence under the pseudo-metric, $\overline{x} \sim \overline{y}$ if and only if $d_\mu (\overline{x},\overline{y}) =0$. The metric quotient $X_\mu \coloneqq \wt{X} / \sim$ will be called the \emph{dual space of the geodesic current} $\mu$.
 \end{definition}

  Given $X$, we denote by $\mathcal{D}(X)$ the set of all duals of geodesic currents on $X$.
 
  When $\mu$ is a measured lamination, then it is known that $X_\mu$ is a $0$-hyperbolic space and, hence, it can be isometrically embedded in a unique $\R$-tree $\widehat{X_\mu}$ (see also \ref{subsec:dual_tree}). It follows that $X_\mu$ can be endowed with a geodesic structure via such embedding $X_\mu \hookrightarrow \widehat{X_\mu}$.

  \begin{remark}
  In the remaining of the paper, when $\mu$ is a measured lamination, we will often denote $\widehat{X_\mu}$ simply with $\mathcal{T}(\mu)$, to emphasise that it is an $\R$-tree.
  \end{remark}

 \subsection{Geodesic structure}
\label{subsec:geodesics}

  In this subsection we explore how to define a geodesic structure on $X_\mu$ when $\mu$ is not necessarily a measured lamination, i.e. when the support of $\mu$ is allowed to have intersections. In particular, we will construct an isometric embedding $X_\mu \hookrightarrow \widehat{X_\mu}$ when $\mu$ is a multi-curve with self intersections, and when $\mu$ has no atoms at all. The mixed case, i.e., when the current has both atomic and non-atomic parts, or the case when the current is purely atomic with a non-discrete set of atoms, will not be discussed in this project and will be fleshed out in a sequel to this project with Anne Parreau.

\subsubsection{Multi-curve case}

Let now $\mu$ be a weighted multi-curve. The support of $\mu$ is union of finitely many discrete orbits of lifts of closed geodesics. In this case, the dual space $X_\mu$ is in general not an $\R$-tree, but it can still be isometrically embedded in a graph, hence in a geodesic space, as follows.
 
 Let $x , y \in X_\mu$ be two point in the dual space of $\mu$. They correspond to two subsets $R_x=\pi_{\mu}^{-1}(x) , R_y=\pi_{\mu}^{-1}(y)$ of $\wt{X}$. Note that $R_x$ and $R_y$ can be a connected component of $\wt{X} \setminus \mathrm{supp}(\mu)$, a geodesic, a geodesic segment, or just a point.
 
 \begin{definition}
 \label{def:adjacent}
 We say that $x$ and $y$ are \emph{adjacent} in $X_\mu$ if there exist $\overline{x} \in\pi_{\mu}^{-1}(x)$, $\overline{y} \in \pi_{\mu}^{-1}(y)$ and a geodesic segment $[\overline{x}, \overline{y}]$ in $\wt{X}$ that doesn't intersect transversely any geodesic of $\mathrm{supp}(\mu)$ in its interior $(\overline{x}, \overline{y})$, and intersects lifts of closed geodesics only in one of its endpoints $\overline{x}$ or $\overline{y}$.
 \end{definition}

 In Figure \ref{fig:adj} we can see a `zoomed-in' of a configuration of geodesics in the support of $\mu$ in $\wt{X}$. According to definition \ref{def:adjacent} the point $\pi_\mu (x)$ is adjacent to $\pi_\mu (y)$, but not adjacent to $\pi_\mu (z)$. The point $\pi_\mu (v)$ is adjacent to $\pi_\mu (w)$ but not adjacent to $\pi_\mu (x), \pi_\mu (y)$ or $\pi_\mu (z)$.  Note that the additional condition that $[\overline{x}, \overline{y}]$ intersects in only one of its endpoint lifts of closed geodesics is motivated by the fact that we want the region corresponding to $\pi_\mu (x)$ to be adjacent to the points $\pi_\mu (y)$ and $\pi_\mu (u)$ corresponding to boundary  geodesics, but we don't want $\pi_\mu (y)$ and $\pi_\mu (u)$ to be adjacent.

 \begin{figure}[htbp]
\begin{center}
\includegraphics[scale=0.2]{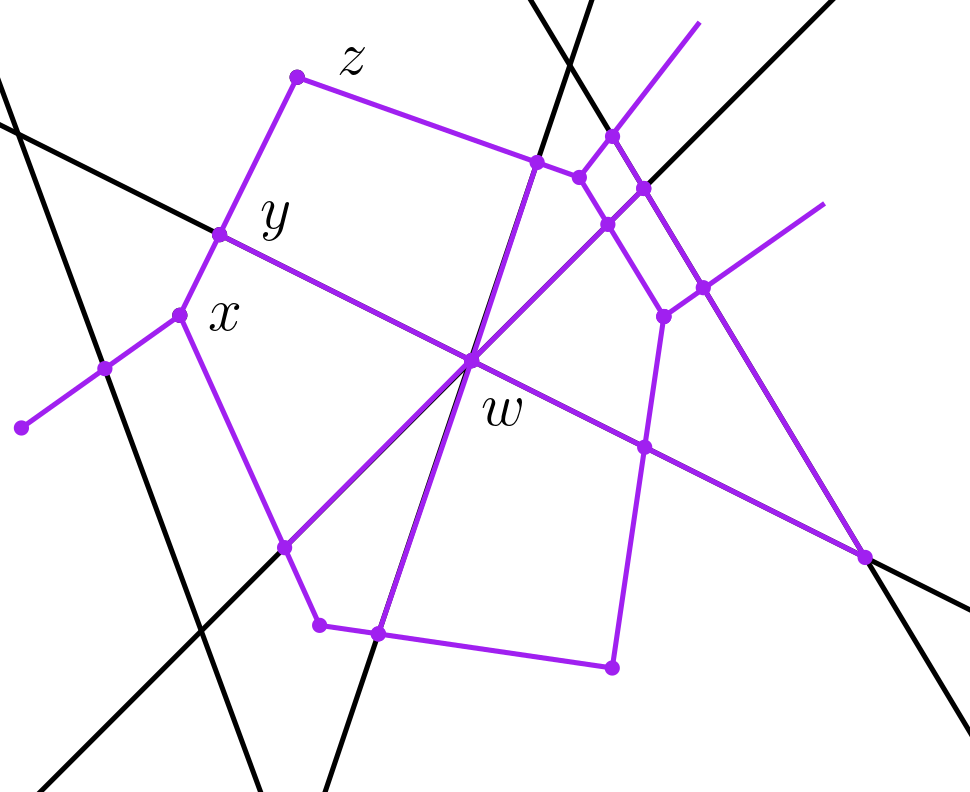}
\caption{}
\label{fig:adj}
\end{center}
\end{figure}

If $v_1$ and $v_2$ in $X_\mu$ are adjacent, we add an edge $e_{1,2}$ joining them of length $d_\mu (v_1, v_2)$. This defines a graph $\widehat{X_{\mu}}$ that has as vertices the points of $X_\mu$, and edges realising the adjacency between points. This embeds $X_{\mu}$ isometrically into a connected graph $\widehat{X_{\mu}}$, and hence a geodesic space. We stress that this construction is by no means canonical, since the notion of adjacency could be defined in many geometrically meaningful different ways.
In the next Example \ref{ex:dual_pair_intersecting_curves} we will see the geometric realisation $\widehat{X_\mu}$ of a current $\mu = \alpha + \beta$ supported on a pair of intersecting simple closed curves.

\begin{example}
\label{ex:dual_pair_intersecting_curves}

Consider the 1-punctured torus $X= X_{1,1}$ and the current $\mu = \alpha + \beta$ whose support is given by the two orthogonally intersecting simple closed geodesics $\alpha$ and $\beta$. Up to pre-composing with an isometry of $\wt{X}$, the lift in $\wt{X}$ of $\alpha$ and $\beta$ in yellow and green respectively, is as in Fig. \ref{fig:dual}.  In that figure we also see the points of $X_\mu$, i.e the equivalence classes with respect to the pseudo-distance $d_\mu$.

Notice that each point in $X_\mu$ corresponds to either a geodesic segment or to an unbounded complementary region. Finally, we can embed $X_\mu$ in its geometric realisation graph $\widehat{X_\mu}$, which is sketched in purple in Figure~\ref{fig:dual}.

\begin{figure}[h]
\begin{center}
        \includegraphics[width=0.55\textwidth]{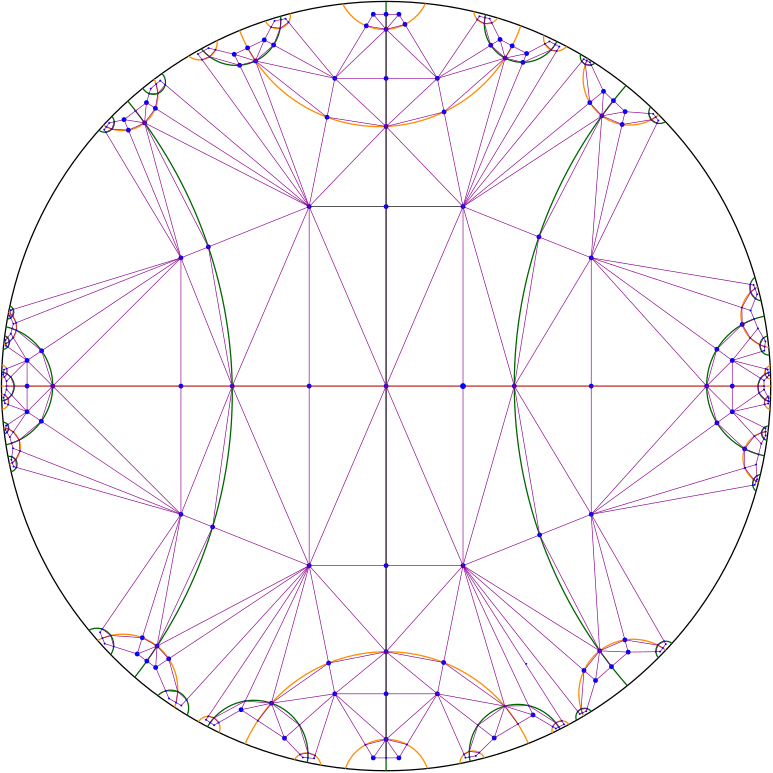} 
        \caption{A sketch of $\widehat{X_\mu}$ when $\mu$ is a union of two intersecting simple closed curves $\alpha$ and $\beta$ in a once-punctured torus. The orbits $\Gamma   \tilde{\alpha}$ and $\Gamma \tilde{\beta}$ are denoted in green and yellow, respectively. The support of $\mu$ is precisely the union of the two orbits. The blue dots represent the points of $X_\mu$.}
        \label{fig:dual}
\end{center}
\end{figure}
\end{example}

 \subsubsection{Non-atomic case}
Assuming $\mu$ has no atoms, we will show $X_{\mu}$ is a geodesic space.
 
  A metric space $(X,d)$ is called \emph{Menger convex} if, for every $x,y \in X$, there exists $z \in X$ so that $d(x,z)=d(y,z)=\frac{1}{2}d(x,y)$.
  The following Lemma can be found  in~\cite[Theorem~2.6.2]{Papadopoulos14:MetricConvexity}.
  \begin{lemma}
  Let $X$ be a proper metric space. $X$ is geodesic if and only if it is Menger convex.
  \label{lem:menger}
  \end{lemma}

  \begin{proposition}
  If $\mu$ is a geodesic current without atoms, then $X_{\mu}$ is a geodesic metric space.
  \label{prop:geodesic}
  \end{proposition}
\begin{proof}
Since $\mu$ has no atoms, by Proposition~\ref{prop:continuousproj}, $\pi_{\mu}$ is continuous.
First, we assume that $\mu$ is filling in $X$, then by Proposition~\ref{prop:properspace}, 
$X_{\mu}$ is proper. Thus, it suffices to check Menger convexity, by Lemma~\ref{lem:menger}. 
Given any two points $x,y \in X_{\mu}$, let $\overline{x} \in \pi_{\mu}^{-1}(x)$ and $\overline{y} \in \pi_{\mu}^{-1}(y)$, and let $I$ be the geodesic segment connecting $\overline{x}$ and $\overline{y}$. $\pi_{\mu}(I)$ is connected, and thus there exists $\overline{z} \in I$ so that $z=\pi_{\mu}(\overline{z})$ satisfies the condition of the statement.
The proof for the general case will use the tree graded structure discussed in Section~\ref{sec:decomposition}.
Let $X_i$, for $i=1,\cdots,k$, denote the subsurfaces in the structural decomposition of $\mu$ (in the sense of Theorem~\ref{thm:structurecurrents}).
If no components of the decomposition of $\mu$ are of type 1 (i.e., filling within $X_i$), then $X_{\mu}$ is equivariantly isometric to an $\mathbb{R}$-tree by Lemma~\ref{lem:isom_trees}. We can thus pullback the geodesics using this isometry.
If any of its structural components is of type 1, we proceed as follows.
Given any two points $x,y \in X_{\mu}$, let $\overline{x} \in \pi_{\mu}^{-1}(x)$ and $\overline{y} \in \pi_{\mu}^{-1}(y)$, and let $I$ be the geodesic segment connecting $\overline{x}$ and $\overline{y}$. We partition $I$ at points $\overline{x_1}, \overline{x_2}, \cdots, \overline{x_{n-1}} \in I$, and set $\overline{x_0} \coloneqq \overline{x}$ and $\overline{x_n} \coloneqq \overline{y}$, so that each subinterval $I_j \coloneqq [\overline{x_j}, \overline{x_{j+1}}]$ is contained in $R_j$, the closure of $\wt{X_j}$ (a lift of a subsurface $X_i$ in the universal cover $\wt{X}$).
There is a subgroup $\Gamma_j$ of $\pi_1(X)$ isomorphic to $\pi_1(X_i)$, acting on $\wt{X_j}$ cocompactly. 
Let $\mu_i$ be a component of $\mu$ which is filling in $X_i$.
Restricting $\mu_i$ to $\wt{X_j}$, we get a measure $\mu_j$ invariant under the action of $\Gamma_j$. Since $\mu_i$ is filling in $X_i$, and has no atoms, $\wt{X_j}/\{d_{\mu_j}=0\}$ is proper by Proposition~\ref{prop:properspace}, so $[\overline{x_j},\overline{x_{j+1}}]$ is geodesic by the same argument using Menger convexity.
If $\mu_j$ is of type 2, then $\wt{X_j}/\{d_{\mu_j}=0\}$ is isometric to an $\mathbb{R}$-tree.
Finally, by Theorem~\ref{thm:metric_tree_graded}, $\wt{X_j}/\{d_{\mu_j}=0\}$ correspond to the pieces of the tree graded structure of $X_{\mu}$. We have proven that $\pi_{\mu}(I)$ is a concatenation of geodesics within the pieces of the tree graded structure and segments of the transverse tree (if the subcurrent $\mu_j$ is of type 3).
Thus, by Proposition~\ref{prop:geodesictreegraded}, is a geodesic compatible with the tree graded structure. This endows $X_{\mu}$ with a geodesic structure.
\end{proof}

 It would be interesting to see which conditions to impose on $\mu$ in order to obtain sharper convexity properties. We say that a $\delta$-hyperbolic space with midpoints is \emph{convex} if, for every triple of points $p,q,r$, if $m_1$ denotes the midpoint between $p$ and $q$ and $m_2$ denotes the midpoint between $p$ and $r$, then we have $d(m_1,m_2) \leq \frac{1}{2} d(q,r)$. What are the conditions on $\mu$ ensuring that $X_{\mu}$ is convex? Convexity properties of this type (and stronger) are useful to guarantee sequential pre-compactness in the setting of the equivariant Gromov-Hausdorff topology as well as separation properties for this topology (see Section~\ref{sec:topology} and~\cite[Chapter~4]{Pau88:Thesis}).

  \subsection{Dependence on the metric structure}

\label{subsec:dependence}
To define the space of geodesic currents $\Curr(X)$, we fixed a $\delta$-hyperbolic geodesic structure $X$ on $S$.
Given a homeomorphism between two different structures $f \colon X \to X'$, we get a homeomorphism $\wt{f} \colon \partial \wt{X} \to \partial \wt{X'}$ by~\cite[Lemma~3.7]{BC88:AutomorphismsSurfaces} extending $f$ to the boundary. This induces a homeomorphism between the corresponding spaces of geodesics and, by pushforward of measures, 
 induces a homeomorphism $\wt{f}_* \colon \Curr(X) \to \Curr(X')$.
 In particular, this induces an action of the mapping class group on $\Curr(X)$.
 For any two structures $X$, we have a commutative diagram of bijections

 \[
  \begin{tikzpicture}[x=2.25cm,y=1.25cm]
    \node (simples) at (0,0) {$\Curr(X)$};
    \node (laminations) at (1,0) {$\mathcal{D}(X)$};
    \node (curves) at (0,-1) {$\Curr(X')$};
    \node (currents) at (1,-1) {$\mathcal{D}(X')$.};
    \draw[-to] (simples) to (curves);
    \draw[-to] (simples) to (laminations);
    \draw[-to] (curves) to (currents);
    \draw[-to] (laminations) to (currents);
  \end{tikzpicture}
\]
The left vertical arrow is a homeomorphism from the above paragraph.

From Theorem~\ref{thm:homeo}, it will follow that, with respect to the equivariant Gromov-Hausdorff topology on $\mathcal{D}(X)$ (see Section~\ref{sec:topology}), the two horizontal maps are continuous injections.
Observe that the isometry class of $X_{\mu}$ does depend on the choice of the underlying hyperbolic structure $X$.
However, for a fixed (unweighted) multi-curve $\mu$ the dual spaces $X_{\mu}$, for different choices of hyperbolic structure $X$, embed as 2-dimensional faces of a $\CAT(0)$-cube complex $\mathcal{S}(\mu)$ independent on the choice of hyperbolic structure: the \emph{Sageev complex} (see~\cite[3]{AG17:Dualcubecomplex} for a description in this setting, and~\cite{CN05:Walls,S01:Cube1,S02:Cube2} in bigger generality). It would be interesting to see if there is an analogous construction of such complex for an arbitrary geodesic current.

 \subsection{Measured wall spaces}

We point out that the dual space of a geodesic current is an example of a measured wall space, in the sense of~\cite{CFV04:MeasuredWall}.

 Given a set $X$, a \emph{wall} of $X$ is a partition $X = h \cup h^c$ where $h$ is any subset of $X$ and $h^c$ denotes its complement. A collection $\mathcal{H}$ is called a collection of half-spaces if for every $h \in \mathcal{H}$ the complementary subset $h^c$ is also in $\mathcal{H}$. Let $\mathcal{W}_{\mathcal{H}}$ denote the collection of pairs $w=(h,h^c)$ with $h \in \mathcal{H}$. We say that $h$ and $h^c$ are the two half-spaces bounding the wall $w$.
 We say that a wall $w=(h,h^c)$ separates two disjoint subsets $A,B$ in $X$ if $A \subset h$ and $B \subset h^c$ or vice-versa and denote by $\mathcal{W}(A|B)$ the set of walls separating $A$ and $B$. In particular, $W(A| \emptyset)$ is the set of walls $w=(h,h^c)$ such that $A \subset h$ or $A \subset h^c$, hence $\mathcal{W}(\emptyset | \emptyset)=\mathcal{W}$.
 We use $\mathcal{W}(x|y)$ to denote $\mathcal{W}(\{ x \} | \{ y \})$.
 
 \begin{definition}[Space with measured walls]
 A space with measured walls is a 4-tuple $(X,\mathcal{W},\mathcal{A},\mu)$ where $\mathcal{W}$ is a collection of walls, $\mathcal{A}$ is a $\sigma$-algebra of subsets in $\mathcal{W}$ and $\mu$ is a measure on $\mathcal{A}$ so that for every two points $x, y \in X$, the set of separating walls $\mathcal{W}(x|y)$ is in $\mathcal{A}$ and has finite measure. We let $d_{\mu}(x,y)=\mu(\mathcal{W}(x|y))$, and we call it the wall-pseudo metric.
 \end{definition}

It is easy to see that $(\wt{X},d_{\mu})$ is a measured wall space on $\wt{X}$, where $\mathcal{W}$ are oriented geodesics of $\wt{X}$, $\mathcal{A}$ is the Borel $\sigma$-algebra of oriented geodesics, and $W \subset \mathcal{W}$, the measure on walls $\mu'$ is given by $\mu' \coloneqq \frac{1}{2}\mu(W)$. This normalization factor is introduced to account for the fact that we get two copies of each geodesic in the support of the current in the space of walls.

 \section{Continuity of the projection}
\label{sec:semicontinuity}
In this section we study the semicontinuity and continuity properties of the natural metric quotient projection map $\pi_{\mu} \colon \wt{X} \to X_{\mu}$.

\subsection{Measure theory results}

We recall some measure theory results that will be used in this section.
Given a sequence of subsets $(A_n)$, $A_n \subset X$,
we have the following definitions:
\begin{equation}
\liminf_n A_n \coloneqq \cup_{n\geq 1} \cap_{j \geq n} A_j
\label{eq:liminf}
\end{equation}
and
\begin{equation}
\limsup_n A_n \coloneqq \cap_{n\geq 1} \cup_{j \geq n} A_j.
\label{eq:limsup}
\end{equation}
The Morgan laws immediately give the following identity
\begin{equation}
\liminf_n A_n = \left( \limsup_n A_n^c \right)^c.
\label{eq:complements}
\end{equation}

The following result is standard and can be found, for example, in~\cite[Theorem~D]{Ham50:Measure} and~\cite[Theorem~E]{Ham50:Measure}.

\begin{lemma}
If $(A_n)$ is a sequence of measurable sets in $X$, so that $A_{n+1} \subseteq A_n$, then for any measure $\mu$ on $X$, we have
\[
\lim_n \mu(A_n) = \mu( \cap_n A_n).
\]
This property is called \emph{continuity of measures from below}.
If $(A_n)$ is a sequence of measurable sets in $X$, so that $A_{n} \subseteq A_{n+1}$ so that, for some $N>0$, $\mu(A_n) < \infty$ for $n \geq N$, then for any measure $\mu$ on $X$, we have
\[
\lim_n \mu(A_n) = \mu( \cup_n A_n).
\]
This property is called \emph{continuity of measures from above}.
\label{lem:cont_measures}
\end{lemma}

\begin{lemma}
Let $\mu$ be a (non-necessarily finite) measure on a measurable space $X$, and let $(A_n)$ be a sequence of measurable sets.
Then
\[
\liminf_n \mu(A_n) \geq \mu(\liminf_n A_n).
\]
Moreover, if for some $n$, $\mu \left( \cup_{j \geq n} A_j \right)$ is finite, we have
\[
\limsup_n \mu(A_n) \leq \mu(\limsup_n A_n).
\]
\label{lem:fatou}
\end{lemma}
\begin{proof}
For the first inequality, let  $B_n \coloneqq \cap_{j \geq n} A_j$.
Note that $B_{n} \subset B_{n+1}$, and thus 
by continuity of measures from below~\ref{lem:cont_measures}, we get
\[
\mu(\lim_n B_n) = \lim_n \mu(B_n).
\]
Note that
\[
\lim_n B_n = \cup_{n \geq 1} B_n = \liminf_n A_n. 
\]
Since $\cap_{j \geq n} A_n \subset A_k$ for $k \geq n$, we have
\[
\mu \left( \cap_{j \geq n} A_n \right) \subset \inf \{ \mu(A_k) \colon k \geq n \}.
\]
Thus, letting $n \to \infty$, we get
\[
\mu \left( \liminf A_n \right) \leq \liminf_n \mu(A_n).
\]
The other inequality follows from continuity of measures from above~\ref{lem:cont_measures}, where the hypothesis of $\mu \left( \cup_{j \geq n} A_j \right)$ being finite for some $n$ is crucial.
\end{proof}

The following result will play a crucial role in understanding the behaviour of $d_{\mu}$ in the presence of atoms.

\begin{proposition}
Let $X=\Gamma \backslash \wt{X}$ be a hyperbolic structure. Let $G[\bar{x}]$ denote the set of geodesics through $\bar{x} \in \wt{X}$.
If $\bar{x}$ does not lie on the axis of an element $g \in \Gamma - \{ e\}$, then
\[
\mu(G[\overline{x}])=0.
\]
\label{prop:zeromeasure}
\end{proposition}
\begin{proof}
Let $\mu$ be a geodesic current, which can be thought of as flip invariant geodesic flow invariant measure on $T^1X$~(see, for example,\cite[Chapter~3]{ES22:GeodesicCount}).
In $T^1X$, the set $G[\overline{x}]$ corresponds to the fiber at $x \in X$ of the canonical projection $\pi \colon T^1 X \to X$, where $x$ is the projection of $\overline{x} \in \wt{X}$ to $X$. Let $K \coloneqq \pi^{-1}(x)$.
If $\mu(K)=0$ we are done. Hence, suppose $\mu(K)>0$. By Poincar\'e recurrence theorem (see comment after~\cite[Theorem~1.2.1]{VO16:FoundationsErgodic} for a statement for continuous flows), we have
\[
\mu(K)=\mu(K^r)
\]
where $K^r$ is the set of \emph{recurrent points} of  $K$, i.e. the set of points $v \in K$ so that there exist $t_n \in \mathbb{R}$ with $t_n \to +\infty$ and $\phi_{t_n}(v) \in K$.
Each $v \in K^r$ is tangent to a geodesic on $X$ going through $x$ more than once. Let $C$ be the countable collection of lifts of $x$ to $\wt{X}$. In particular, $\overline{x} \in C$. Let $\hat{v}$ be the corresponding
lift of $v$ based at $\overline{x}$.  Denote by $x(\hat{v})$ the  point of intersection with $C$ (after $\overline{x}$) of the geodesic through $\hat{v}$.
Since the geodesic between $\overline{x}$ and $\overline{x}(\hat{v})$ is unique, $\hat{v}$ is uniquely determined by the choice of $\overline{x}(\hat{v}) \in C$ which is countable. Hence, $K^r$ is countable.
Thus, using countable additivity of $\mu$, $\mu(K^r)$ can be written as a sum of measures of singletons, which are all 0 except (possibly) if $v$ is contained in a lift of a closed geodesic. 
\end{proof}

\begin{proposition}
If $\mu$ is a geodesic current without atoms, then the function
\[f(\cdot, \cdot) \coloneqq d_{\mu}(\pi_{\mu}(\cdot ), \pi_{\mu}(\cdot)) \colon \wt{X} \times \wt{X} \to \mathbb{R}\] is a continuous function.
\label{prop:contdistance}
\end{proposition}
\begin{proof}
Let $\left( \overline{x_n}, \overline{y_n} \right)$ be a sequence in $\wt{X} \times \wt{X}$ converging to $(\overline{x}, \overline{y})$.
We first show that
\begin{equation} \liminf_n d_{\mu}(\pi_{\mu}(\overline{x_n}), \pi_{\mu}(\overline{y_n})) \geq d_{\mu}(\pi_{\mu}(\overline{x}), \pi_{\mu}(\overline{y}))
\label{eq:lower}
\end{equation}
Note that, since $\mu$ has no atoms, we have, by Proposition~\ref{prop:zeromeasure},
\[
d_{\mu}(\pi_{\mu}(\overline{x_n}), \pi_{\mu}(\overline{y_n})) = \mu(G(\overline{x_n},\overline{y_n})).
\]
for all $n$.
\begin{claim}
If $\gamma \in G(\overline{x},\overline{y})$, then $\gamma \in \cup_{N \geq 1} \cap_{n \geq N} G(\overline{x_n},\overline{y_n})$.
\label{claim:lim_inf}
\end{claim}
\begin{proof}
Since $\gamma \in G(\overline{x},\overline{y})$, $\gamma$ intersects $(\overline{x},\overline{y})$ transversely at some $\overline{z} \in (\overline{x},\overline{y})$. Since
$(\overline{x_n},\overline{y_n})$ is converging in Hausdorff distance to $(\overline{x},\overline{y})$, we have that for $n_0$ large enough, $\gamma$ must intersect transversely all $(\overline{x_n},\overline{y_n})$ for all $n \geq n_0$.
\end{proof}
From Claim~\ref{claim:lim_inf} and Lemma~\ref{lem:fatou}, it follows that 
\[
\mu(G(\overline{x},\overline{y})) \leq \mu \left( \cup_{N \geq 1} \cap_{n \geq N} G(\overline{x_n},\overline{y_n})\right) \leq \liminf_n \mu(G(\overline{x_n},\overline{y_n}))
\]
from which Equation~\ref{eq:lower} follows.
We will now prove 
\begin{equation} \limsup_n d_{\mu}(\pi_{\mu}(\overline{x_n}), \pi_{\mu}(\overline{y_n})) \leq d_{\mu}(\pi_{\mu}(\overline{x}), \pi_{\mu}(\overline{y})).
\label{eq:upper}
\end{equation}

\begin{claim}
If $\gamma$ is a geodesic disjoint from $[\overline{x},\overline{y}]$, then $\gamma \notin \limsup_{n} G(\overline{x_n},\overline{y_n})$.
\label{claim:upper}
\end{claim}
\begin{proof}
Let $A_n=G(\overline{x_n},\overline{y_n})$.
We are assuming $\gamma$ is a geodesic disjoint from $(x,y)$, and we want to show  $\gamma \notin \limsup_{n} A_n$, which, by Equation~\ref{eq:complements} is equivalent to showing
$\gamma \in \liminf_n A_n^c$, i.e., to showing that there exists $N > 0$ so that for all $n \geq N$, $\gamma \notin A_n$.
This follows, since $\gamma$ and $(\overline{x},\overline{y})$ are a definite distance apart, and $(\overline{x_n},\overline{y_n})$ is converging to $(\overline{x},\overline{y})$ in the Hausdorff distance of $\wt{X}$. Thus, for $n_0$ large enough, $(\overline{x_n},\overline{y_n})$ is also disjoint from $\gamma$, for all $n \geq n_0$.
\end{proof}
A geodesic which is not in $G(\overline{x},\overline{y})$ is either disjoint from $(\overline{x},\overline{y})$, or it is the geodesic $\gamma_{x,y}$ determined by the points $x$ and $y$.
Claim~\ref{claim:upper} can be phrased as follows:
\[
\limsup_{n}  A_n \subset G(\overline{x},\overline{y}) - \{ \gamma_{x,y} \}.
\]
Since $\mu$ has no atoms, we have
\[
\limsup_{n}  \mu(A_n ) \leq \mu(\limsup_{n} A_n ) \leq \mu( G(\overline{x},\overline{y}) - \{ \gamma_{x,y} \} ) = \mu( G(\overline{x},\overline{y})),
\]
where we have used the second equation in Lemma~\ref{lem:fatou} in the first inequality, and the last equation plus the monotonicity of $\mu$ in the second inequality.
It remains to explain why the application of Lemma~\ref{lem:fatou} is justified. Since $\overline{x_n} \to \overline{x}$ and $\overline{y_n} \to \overline{y}$, there exists $n_0$ so that for $n \geq n_0$, $(\overline{x_n},\overline{y_n}) \subset B$ for some compact ball $B$. Then \[\cup_{n \geq n_0} A_n \subset G(B).\]
Since $B$ is compact, $\mu(G(B))$ is finite, and thus the application of the Lemma is justified.
This shows Equation~\ref{eq:upper} and finishes the proof.
\end{proof}

\begin{proposition}
Let $\mu \in \Curr(X)$. 
If $\mu$ has atoms, then $f(\cdot, \cdot)$ is neither lower semicontinuous nor upper semicontinuous.
\label{prop:discontdistance}
\end{proposition}
\begin{proof}
Suppose $\mu$ is a geodesic current with atoms.
By Lemma~\ref{lem:discreteatoms} and Lemma~\ref{lem:atompencil}, the only 0-dimensional and 1-dimensional atoms of a geodesic current are concentrated on lifts of closed geodesics and pencils containing lifts of closed geodesics, respectively.
Let then $\gamma$ be a lift of a closed geodesic in the support of $\mu$.
Let $[\overline{x},\overline{y})$ be a subsegment of $\gamma$, short enough so that no other lift of a closed geodesic in the support of $\mu$ intersects it transversely. Let $\overline{y_n}$ be a sequence of points outside of $\gamma$ converging to $\overline{y}$.
Then, if we denote $\mu(G[\overline{x},\overline{y_n})) = \epsilon_1$, we have $\epsilon_1>0$ for all $n$, while $\mu(G[\overline{x},\overline{y}))=\epsilon_2$, where $\epsilon_1>\epsilon_2 \geq 0$.
This shows that
\[
\limsup_n f(\overline{x_n},\overline{y_n}) > f(x,y)
\]
so $f$ is not upper semicontinuous.
Observe that this construction can also be adapted for an open interval $(\overline{x},\overline{y})$, by considering open segments $(\overline{x_n},\overline{y_n})$ given by sequences of points $(\overline{x_n})$ and $(\overline{y_n})$, where $\overline{x_n} \to \overline{x}$ and $\overline{y_n} \to \overline{y}$ from distinct sides of $\gamma$. This ensures that $(\overline{x_n},\overline{y_n})$ crosses $\gamma$ transversely for all $n$, while $\gamma$ does not intersect $(\overline{x},\overline{y})$ transversely.
Let's show how lower semicontinuity fails. Let $\gamma$ be again a lift of a closed geodesic in the support of $\mu$.
Suppose that $\gamma$ intersects a segment $[\overline{x},\overline{y})$ transversely at $\overline{x}$, and let $\overline{x_n}$ be a sequence converging to $\overline{x}$ so that $[\overline{x_n},\overline{y}) \subset [\overline{x_{n+1}},\overline{y})$ for all $n$.
Then, $\mu(G[\overline{x_n},\overline{y})) \leq \epsilon_1$ for all $n$ and for some $\epsilon_1 \geq 0$, while, if $\mu(G[\overline{x},\overline{y}))=\epsilon_2$, we have that $0 \leq \epsilon_1 < \epsilon_2$.
This shows that
\[
\liminf_n f(\overline{x_n},\overline{y_n}) < f(x,y),
\]
so $f$ is not lower semicontinuous.
\label{rem:notcontinuous}
\end{proof}

Combining both Propositions~\ref{prop:contdistance} and~\ref{prop:discontdistance}, we get the following.

\begin{proposition}
A geodesic current $\mu \in \Curr(X)$ has no atoms 
if and only if the map $\pi_{\mu} \colon \wt{X} \to X_{\mu}$ is continuous. 
\label{prop:continuousproj}
\end{proposition}
\begin{proof}
If $\mu$ has no atoms, then the implication follows from Propositions~\ref{prop:contdistance} and~\ref{prop:discontdistance}. Namely, for $\overline{x} \in \wt{X}$, consider the continuous function $g_{\overline{x}}(\cdot) \coloneqq f(\overline{x},\cdot)$. Since this function is continuous at $\overline{x}$, for every $\epsilon>0$, we can find $\delta_{\overline{x},\epsilon}$ so that if $d_{\wt{X}}(\overline{x},\overline{y})<\delta$, we have $g_{\overline{x}}(\overline{y})<\epsilon$, i.e.,  $d_{\mu}(\pi_{\mu}(\overline{x}), \pi_{\mu}(\overline{y}))<\epsilon$, which proves continuity of $\pi_{\mu}$ at $\overline{x}$.
If $\mu$ has atoms, then let $\bar{x} \in \gamma$, for $\gamma$ an atom of $\mu$.
Then, there exist an $\epsilon>0$ and a sequence of points $\overline{x_n} \notin \gamma$, so that $d_{\mu}(\pi_{\mu}(\bar{x_n}),\pi_{\mu}(\bar{x}))> \epsilon$ as $\overline{x_n} \to \overline{x}$, so $\pi_{\mu}$ is not continuous.
\end{proof}

Proposition~\ref{prop:continuousproj} might come as a surprise for those familiar with the classical work on $\mathbb{R}$-trees dual to measured foliations/laminations in~\cite{W98:Duals}, where the natural projection map is always continuous.
We emphasize that when $\mu$ is a measured lamination, one can eliminate the discontinuities of $\pi_{\mu}$ created by the atoms by a process of blow-up, as described in~\cite[Defintion~11.27]{Kap00:Kapovich2000HyperbolicMA}. 

We can combine the above results with the results in Section~\ref{sec:actions}, to obtain the following.

\begin{proposition}
Given a geodesic current $\mu$ on $X$, the projection $\pi_{\mu} \colon \wt{X} \to X_{\mu}$ satisfies:
\begin{enumerate}
\item If $\mu$ has no atoms and it is filling then $\pi_{\mu}$ is a $\pi_1(X)$-equivariant closed map.
\item If $\mu$ has no atoms and has full support, then $\pi_{\mu}$ is a $\pi_1(X)$-equivariant homeomorphism.
\end{enumerate}
\label{thm:homeoprojection}
\end{proposition}
\begin{proof}
Since $\mu$ has no atoms, by~Proposition~\ref{prop:contdistance}, $\pi_{\mu}$ is continuous and by Proposition~\ref{prop:properspace}, $X_{\mu}$ is a proper metric space. Therefore, by Lemma~\ref{thm:hopfrinow}, it is locally compact.
By Lemma~\ref{lem:real_cocompact} the action of $\pi_1(X)$ on $X_{\mu}$ is cocompact and by Lemma~\ref{prop:proper}, the action is proper. Hence, the quotient $X_{\mu}/\pi_1(X)$ is Hausdorff (and compact). By $\pi_1(X)$ equivariance of $\pi_{\mu}$, the descended map $\overline{\pi_{\mu}} \colon X \to X_{\mu}/\pi_1(X)$ is a well-defined continuous surjective map from the compact surface $X$ to the Hausdorff space $X_{\mu}/\pi_1(X)$, so it is a closed map. If moreover $\mu$ has full support, then $\pi_{\mu}$ is injective, and hence it is a homeomorphism. By equivariance, this implies the same results for $\pi_{\mu}$.\end{proof}
\section{Examples}
\label{sec:examples}

In this section we study examples of dual spaces associated to the geodesic currents discussed in Subsection~\ref{subsec:examples_currents}

\subsection{The dual tree of a measured geodesic lamination}
\label{subsec:dual_tree}
The first example of dual space of a geodesic current is the so-called \emph{dual tree $\mathcal{T}(\lambda)$ of a lamination $\lambda$}.
Such space has been well studied (for example in \cite{MORGAN1991143}) before the notion of dual space of a geodesic current was introduced. We recall here the original construction by Morgan--Shalen and show that the dual tree $\mathcal{T}(\lambda)$ of a lamination coincides with the dual space $X_\lambda$ of the current $\lambda$.

Let
$(\lambda, \mu)$ be a measured lamination (recall Definition~\ref{def:lam}) on a hyperbolic surface $X$ with support $\lambda = \mathrm{supp}(\mu)$ and transverse measure $\mu$. Denote with $(\wt{\lambda}, \tilde{\mu})$ the lifted measured lamination on $\wt{X} $, and define $\mathfrak{C}$ the set of connected components of $\wt{X} \setminus \supp \lambda$, each of which is called a \emph{complementary region} of $\lambda$.

Let $c_0, c_1 \in \mathfrak{C}$. We define a metric on $\mathfrak{C}$ as follows:
\[
d_{(\lambda, \mu)} (c_0, c_1) = \inf \mu (\gamma)
\]
where the $\inf$ is taken over all quasi-transverse arcs $\gamma$ such that $\gamma(0) = x_0 \in c_o$ and $\gamma(1) = x_1 \in c_1$. A \emph{quasi-transverse arc} is an arc intersecting transversely each leaf of the lamination at most once. 

Morgan and Shalen proved the following

 \begin{theorem}[\cite{MORGAN1991143} Lemma 5]
 \label{thm:morganinclusion}
Denote with $\mathfrak{C}$ the set of complementary regions of $\tilde{\lambda}$. There exists an $\R$-tree $\mathcal{T}(\lambda)$ called the \emph{dual tree of $\lambda$} and an isometric embedding $\psi: \mathfrak{C} \hookrightarrow \mathcal{T}(\lambda)$ such that:
\begin{enumerate}
\item $\psi(\mathfrak{C})$ spans $\mathcal{T}(\lambda)$.
\item Any point $x \in \mathcal{T}(\lambda) \setminus \psi(\mathfrak{C})$ is an edge point, i.e. separates $\mathcal{T}(\lambda)$ in two connected components.
\item The action $\pi_1 (X)$ on $\mathfrak{C}$ extends uniquely to an action by isometries of $\pi_1 (X)$ on $\mathcal{T}(\lambda)$.
\end{enumerate}

Moreover, if $T$ and $T'$ are two $\R$-trees satisfying the above properties, then there exists an equivariant isometry $T \to T'$ with respect to the $\pi_1 (X)$-action.

\end{theorem}

On the other hand in Section \ref{sec:dualis0hyperbolic} we show that the dual space $X_\lambda$ of the current $\lambda$ is $0$-hyperbolic as well, and hence embeds isometrically into an $\R$-tree $\widehat{X_\lambda}$~\cite[Theorem~3.38]{E08:Lectures}.

\begin{lemma}
\label{lem:isom_trees}
    The $\R$-trees $\widehat{X_\lambda}$ and $T(\lambda)$ are isometric.
    \begin{proof}
        The action of $\pi_1(X)$ on both $\mathbb{R}$-trees satisfy the hypotheses of~\cite[Theorem~3.7]{CM87:Groups} and they the same length functions, hence, there exists an equivariant isometry between them.
    \end{proof}
\end{lemma}
By Lemma~\ref{lem:isom_trees}, we will often refer with a slight abuse of notation to $X_\lambda$ as the dual tree $\mathcal{T}(\lambda)$ of $\lambda$.

It is worth noticing that there are other equivalent definitions of the dual tree of a lamination, for example the one given by Kapovich in {\cite[Section 11.12]{Kap00:Kapovich2000HyperbolicMA}} or by Wolf in \cite{W98:Duals}, which uses the notion of measured foliations. For a comparison between these definitions we refer the reader to \cite{DeRosa23Thesis}. 

\subsection{Guirardel core: a filling sum of two measured laminations on the surface}
\label{ex:guirardelcore}

Let $X$ be a compact hyperbolic surface (possibly with boundary), and consider $\alpha,\beta$ two measured laminations so that $\alpha+\beta$ is filling as a geodesic current. For concreteness, the reader might want to assume that $\alpha$ and $\beta$ are simple closed curves, so that the multi-curve $\mu=\alpha \cup \beta$ is filling (this can always be achieved, see \cite[1.3.2]{FM12:Primer} for the argument in genus 2). 

Given a finitely generated group $G$ and two $\mathbb{R}$-trees $T_1, T_2$ equipped with isometric actions, the \emph{Guirardel's core} $\mathcal{C}$ is the smallest subset of $T_1 \times T_2$ which is
\begin{enumerate}
\item $\pi_1(X)$-invariant
\item closed and connected
\item For every $x_0 \in T_1$ and every $x_1 \in T_2$, both $\mathcal{C} \cup (\{ x_0 \} \times T_{2})$ and $\mathcal{C} \cup (T_{1} \times \{ x_1 \})$ are convex.
\end{enumerate}
We commend the reader to the paper \cite[Section~2.2, Example~3]{Gui05:Coeur} for more details. 
When $G$ is the fundamental group of a closed surface, and $T_1,T_2$ are the $\mathbb{R}$-trees $\widehat{X_{\alpha}}$ and $\widehat{X_{\beta}}$,
we claim that $X_{\mu}$ embeds isometrically into the Guirardel's core $\mathcal{C}$ of the product of trees $\widehat{X_{\alpha}}$ and $\widehat{X_{\beta}}$.
In this case, the core can be described as follows.
Let $\pi_{\alpha} \colon \wt{X} \to \widehat{X_{\alpha}}$ and $\pi_{\beta} \colon \wt{X} \to \widehat{X_{\beta}}$ be the composition of the natural projection maps with the isometric embedding of the dual into their $\mathbb{R}$-tree, and define
$f \colon \wt{X} \to \widehat{X_{\alpha}} \times \widehat{X_{\beta}}$ by $(\pi_{\alpha}(\overline{x}),\pi_{\beta}(\overline{x}))$.
Guirardel's core corresponds to $\mathcal{C}=f(\wt{X})$. 
We now define a map $\overline{f}$ from $X_{\mu}$ to $\mathcal{C}$,
by picking, for every $x \in X_{\mu}$, a point $\overline{x} \in \pi_{\mu}^{-1}(x)$, and defining $\overline{f}(x)\coloneqq f(\bar{x})$.

\begin{lemma}
    The map $\overline{f} : X_\mu \to \mathcal{C}$ is a homeomorphism. 
    \begin{proof}
        We firstly show that $\overline{f}$ is well-defined.
Indeed, suppose that $x \sim y$, i.e., $d_{\mu}(x,y)=0$.
This happens in the following cases:
\begin{enumerate}
\item $x,y$ are the same point;
\item $x,y$ are in the same complementary region of $\mu$;
\item $x,y$ are in the same lift of $\alpha$ and not separated by any lift of $\beta$, or on the same lift of $\beta$ and not separated by any lift of $\alpha$.
\end{enumerate}
In any of these cases, we see that $\overline{f}(x)=\overline{f}(y)$. 
In fact, one can check that those conditions also characterize the set of pairs $(x,y)$ so that $\overline{f}(x)=\overline{f}(y)$, which proves injectivity. Surjectivity is clear from the definition. Continuity of 
$\overline{f}$ and its inverse with respect to the topology of $X_{\mu}$ follows from the fact that, as pseudo-metrics, we have
\[
d_{\mu}(\bar{x},\bar{y})=d_{\alpha}(\bar{x},\bar{y}) + d_{\beta}(\bar{x},\bar{y}).
\]
    \end{proof}
\end{lemma}

\begin{remark}
In his work, Guirardel goes on to prove that $i(\alpha,\beta)$ is equal to the volume of $\mathcal{C}/\pi_1(X)$, where volume is defined by taking the supremum, over finite trees (convex hull of finitely many points) $K_1,K_2$ of $T_1,T_2$ of the product of Lebesgue measures $\mu_{K_1} \times \mu_{K_2}$ (a finite tree is simplicial, and the Lebesgue measure of a simplicial tree is induced from the Lebesgue measure on the edges, since each edge is isometric to an interval of $\mathbb{R}$). On the other hand, $i(\alpha,\beta)=\frac{1}{2}i(\alpha + \beta, \alpha + \beta)=\frac{1}{2} i(\mu,\mu)$. It would be interesting to see if one can recover the self-intersection number of $\mu$ as some sort of volume of $X_{\mu}/\pi_1(X)$. Compare this with Example~\ref{ex:LiouvilleDual}, where we show that when $\mu$ is the Liouville current of $X$, $X/\pi_1(X)$ is isometric to the hyperbolic surface $X$, and, on the other hand, $i(\mu,\mu)=\pi^2 |\chi(X)|$ (by~\cite[Proposition~15]{Bonahon88:GeodesicCurrent}, so the volume on the dual is not quite the hyperbolic area of $X$, but rather a multiple of it.
We refer to \cite[Definition~3.4]{BIPP21:PSL2xPSL2} as well as~\cite[Section~6.5]{O19:Energy} for a construction of a space on which our space embeds isometrically (provided one makes a choice of $\ell_1$ metric on the product).
\end{remark}

\subsection{Duals for positively ratioed Anosov representations}
\label{ex:positivelyratioed}
In this example we assume $X$ is a closed hyperbolic surface and we go back to the class of geodesic currents introduced in Subsection~\ref{subsubsec:positivelyratioed} namely, positively ratioed Anosov representations.
First, we have the following immediate observation.

\begin{lemma}
Let $\rho \colon \pi_1(X) \to G$ be a positively ratioed representation and $\mu_{\rho}$ its associated geodesic current as in \ref{subsubsec:positivelyratioed}. Then $X_{\mu_{\rho}}/\pi_1(X)$ is homeomorphic to $X$.
\end{lemma}
\begin{proof}
By Lemma~\ref{lem:anosovnoatomfull}, geodesic currents associated to positively ratioed Anosov representations have no atoms and full support.
Thus, by Proposition~\ref{thm:homeoprojection}, the result follows.
\end{proof}

\begin{remark}
Note that there are points in the boundary of the Hitchin component corresponding to geodesic currents that might have atoms or might not have full support~(see~\cite{BIPP21:Currents} and~\cite{BIPP21:Crossratios}).
In fact, as we mentioned in Subsection~\ref{subsubsec:positivelyratioed}, in the paper~\cite{BIPP21:Crossratios} the notion of positively ratioed representations is generalized to representations whose cross-ratios need not be continuous (and, thus, whose associated geodesic currents might have atoms). In the sequel to this project with Anne Parreau, we will endow the dual spaces coming from these representations with a geodesic structure.
\end{remark}

In what follows we specialize to the case $G=\SL(3,\mathbb{R})$, and consider only Hitchin representations.
By work of Choi-Goldman~\cite{CG93:ConvexProjective}, to every $\SL(3,\mathbb{R})$ Hitchin representation corresponds, in a one-to-one fashion, a strictly convex real projective structure $\Omega_{\rho}$, that we define now.

A \emph{strictly convex real projective surface} is a quotient $Z = \Omega / \Gamma$ where $\Omega \subseteq \mathbb{RP}^2$ is a strictly convex domain of the real projective plane, and $\Gamma < SL(3, \R)$ is a discrete group of projective transformations acting properly on $\Omega$. Thus, $Z$ is the topological surface $S$ equipped with a Hilbert metric.
Let $\wt{Z}$ be the induced metric in the universal cover of $S$. With this choice, the compactification $\wt{Z} \cup \partial Z$ can be identified with $\Omega \cup \partial \Omega$ and thus a geodesic current on $Z$ is a $\pi_1(Z)$-invariant locally finite Borel measure on the set
$(\partial \Omega \times \partial \Omega - \Delta)/\mathbb{Z}_2$.

We define a proper complete path metric $d_\Omega$ which is invariant under $\mathrm{Aut}(\Omega)$, called the \emph{Hilbert metric} associated to a convex $\mathbb{RP}^2$-structure $Z$.
Let $p,q \in \Omega$, the projective geodesic through $p$ and $q$ in $\Omega$ defines a pair of points $a$ and $b$ on $\partial \Omega$. We define the following complete metric (see, for example~\cite[Theorem~2.1]{Bea99:HilbertDomain}) on $\Omega$
\begin{equation}
\label{hilbert metric}
d_\Omega (p,q) \coloneqq \left \{ \begin{matrix} \log{\mathrm{cr}(a,p,q,b)}& \mbox{ if } p \neq q \\ 0 & p=q \end{matrix} \right.
\end{equation}
where $a,p,q,b$ appear in this order along the projective line and
$\mathrm{cr}(a,p,q,b)$ denotes the projective cross-ratio.
This metric descends to the strictly convex projective structure $Z$.
Furthermore, the projective cross-ratio also induces a geodesic current $\mu_{\Omega}$, by setting 
 \[
 \mu_{\Omega}(B)=\log \mathrm{cr}(a,b,c,d)
 \]
where $B=I_{a,b} \times I_{c,d} \subseteq \mathcal{G}(\wt{Z})$ is a box of geodesics with vertices $a,b,c,d \in \partial Z$, see for example \cite[Subsection 4.1]{BIPP21:Crossratios} for details. 
The current $\mu_\Omega$ satisfies the relation
$\ell_{\rho}^H(c)=i(\mu_{\Omega},c)$ for every closed curve $c \in \Curves(Z)$, where $\ell_{\rho}^H$ is the Hilbert length function associated to $\rho$, which in this setting is given by
\[
\ell_\rho^H \colon \Gamma \to \R : g \mapsto \inf_{z \in \wt{Z}} d_{\Omega}(z,g z)
\]
or, equivalently, by
\[\log{\frac{\lambda_1 (\rho(g))}{\lambda_3 (\rho(g))}} 
\]
where $\lambda_i (\rho(g))$ is the $i$-th eigenvalue of $g \in \SL (3, \R)$.

We show now that $\mu_{\Omega}$ satisfies the \emph{Crofton property} in the sense introduced in Subsection~\ref{ex:Liouville}.
First, we recall some classical results in projective geometry of dimension 2.

Let $\Omega$ be an arbitrary convex domain of the projective plane $\mathbb{R}P^2$. Busemann~\cite{Bus55:Geometry} introduced an additive non-negative measure $\sigma$ on the set of projective lines on $\Omega$ satisfying
\begin{enumerate}
\item For any $C \subset \Omega$ so that $\overline{C} \cap \partial \Omega = \emptyset$, $\sigma(G(C))<\infty$
\item $\sigma(G[x])=0$ for any $x \in \Omega$
\item $\sigma(G(S))>0$, where $S \subset \Omega$ contains a non-degenerate projective line segment.
\end{enumerate}

Using this measure $\sigma$, he defined a \emph{$\sigma$-metric} on $\Omega$ given by
\[
d_{\sigma}(x,y)=\sigma(G[x,y]).
\]
It turns out that, in dimension two, any continuous metric on $\Omega$ for which the projective lines are geodesics can be realized as such a $\sigma$-metric~\cite{Pog73:HilbertFourth}. 
The Hilbert metric $d_{\Omega}$ defined above satisfies these assumptions, and thus, there exists a measure $\sigma$ on the set of projective lines on $\Omega$ satisfying 
\[
d_{\Omega}(x,y)=\sigma(G[x,y]).
\]

\begin{proposition}
For every pair $\overline{x},\overline{y} \in \Omega$, \begin{equation}
    \label{HilbertLength}
\ell_{\Omega} ([\overline{x},\overline{y}]) =\mu_{\Omega} ( G[\overline{x},\overline{y}]).
\end{equation}
\end{proposition}
\begin{proof}
    By the discussion above, we know that
    \[
    d_{\Omega} ([\overline{x},\overline{y}]) =\sigma ( G[\overline{x},\overline{y}]).
    \]
    Observe that, since $\pi_1(Z)$ acts by projective isometries, $g(G[\overline{x},\overline{y}])=G[g(\overline{x}),g(\overline{y})]$.
    Thus, by the above equation
    \[
    \sigma(g(G[\overline{x},\overline{y}]))=\sigma(G[g(\overline{x}),g(\overline{y})])=\ell_{\Omega} ([g(\overline{x}),g(\overline{y})])=\ell_{\Omega} ([\overline{x},\overline{y}]).
    \]
    This shows that
    \[
    \sigma(g(G[\overline{x},\overline{y}]))=\mu(G[\overline{x},\overline{y}]),
    \]
    for every $g \in \pi_1(Z)$, and every $\overline{x},\overline{y} \in Z$.
    Since the sets $G[\overline{x},\overline{y}]$ generate the Borel sigma algebra of geodesics (see~\cite[Lemma~A.2]{MZ19:PositivelyRatioed}, where the proof is given assuming $X$ is hyperbolic, but the same proof works verbatim in the setting of $X$ Hilbert strictly convex), this shows $\sigma$ is $\pi_1(Z)$-invariant. It follows that $\sigma$ is a geodesic current on $\mathcal{G}(\wt{Z})$.
    Let $g$ a loxodromic element with axis a projective line $L_g$. Let $\overline{x} \in L_g$. Then
    \[
    \ell_{\Omega}(g)=d_{\Omega} ([\overline{x},g \overline{x}]) =\sigma ( G[\overline{x},g \overline{x}])=i(\sigma, g).
    \]
    On the other hand, we also have
    \[
    \ell_{\Omega}(g)=d_{\Omega} ([\overline{x},g \overline{x}]) =\mu_{\Omega} ( G[\overline{x},g \overline{x}])=i(\mu_{\Omega}, g).
    \]    
    Thus, by~\cite[Th\'eor\`eme~2]{Otal90:SpectreMarqueNegative}, $\mu_{\Omega}=\sigma$.
\end{proof}

Let us now fix a Hitchin representation $\rho \colon \Gamma=\pi_1(X) \to \SL(3, \R)$. Martone-Zhang show in \cite{MZ19:PositivelyRatioed} that $\rho$ induces a geodesic current $\mu_\rho$ on $\mathcal{G}(X)$ such that
\[
i(\mu_\rho , c) = \ell_{\Omega_\rho} (c)
\]
for every $c \in \Curves(X)$.
 Moreover, the Hitchin representation $\rho$ induces a \emph{boundary map}
 \[
 \psi \colon \partial \wt{X} \to \partial \Omega.
 \]
 This boundary map is the extension to the boundary at infinity of the developing map of the strictly convex real projective structure. 
 
 \begin{proposition}
 Given $X$ a closed hyperbolic surface and a Hitchin representation $\rho \colon \pi_1(X) \to \SL(3,\mathbb{R})$, let $\dev \colon \wt{X} \to \Omega$ be the developing map of the associated convex projective structure.
 Let $\psi$ denote the extension to the boundary at infinity of this homeomorphism
 \[
 \psi \colon \partial \wt{X} \to \partial \Omega.
 \]
 If $[\gamma] \in \pi_1(X)$, then $\psi(\gamma^{\pm})=(\rho([\gamma]))^{\pm}$. 
 \label{prop:samemaps}
 \end{proposition}
 \begin{proof}
  The space $\Omega$ is $\delta$-hyperbolic if and only if it is a strictly convex divisible domain~\cite[Theorem~4.5]{Ben04:Convexes}. Thus, the extension of $\dev$ to the boundaries at infinity is well-defined.
  Let us recall that points in the boundary at infinity are equivalence classes of sequences of points, where the equivalence classes identify two sequences if they are at bounded distance. The map $\psi$ hence associates to a point $x = [(x_n)] \in \del \wt{X}$ the point $[(\mathrm{dev}(x_n))] \in \del \Omega$. Let us fix a basepoint $o \in \wt{X}$ and let $[\gamma] \in \pi_1 (X)$, then we can write $\gamma^+ \in \del \wt{X}$ as $\gamma^+ = [ {(\gamma^n \cdot o )}]$. We have that
\[
      \psi (\gamma^+) = \psi [ ( \mathrm{dev} (\gamma^n \cdot o ) )] = [ (\rho (\gamma^n) \cdot \mathrm{dev}(o) ) ] = \rho ([\gamma])^+
\]
where we have used in the second equality the $\Gamma$-equivariance of the developing map.
 \end{proof}

 In fact, we have the following.
 
 \begin{lemma}
 The current $\mu_{\Omega}$ is the push-forward of $\mu_\rho$ via $\psi$, the extension to the boundary of the developing map.
  \end{lemma}
 \begin{proof}
 Let $\gamma_-$ and $\gamma_+$ be the repelling and attractive fixed points in $\partial \wt{X}$ of the deck transformation corresponding to $[\gamma] \ni c$.
 By Proposition~\ref{prop:samemaps}, we have
 that $\gamma'_+ = \psi(\gamma_+)$ corresponds to the attracting point of $\rho(\gamma)$, and $\gamma'_- = \psi (\gamma_-)$ corresponds to the repelling point of $\rho(\gamma)$.
 Then we have, on the one hand, if $\gamma_-'\gamma_+'$ denotes the projective line determined by the points at infinity $\gamma_-'$ and $\gamma_+'$, and $x \in \gamma_-'\gamma_+'$, we have
 \begin{align*}
     \ell_\Omega (c) & = \mu_{\Omega} (G[x, \gamma x]) \\
     &= i(\mu_{\Omega}, \gamma) = \mu_C ([\gamma'_-, \gamma'_+] \times [z, \gamma z] ).
     \end{align*}

     By $\pi_1(X)$-equivariance of $\psi$, we get

 \begin{align*}
i(\psi_* \mu_\rho , \gamma)  &= \psi_* \mu_\rho ([\gamma'_-, \gamma'_+] \times [z, \gamma z] ) = \\
 &= \mu_\rho ([\psi^{-1}\gamma'_-, \psi^{-1}\gamma'_+] \times [\psi^{-1}(z), \gamma \psi^{-1}(z)] ) \\
 &=\mu_\rho ([\gamma_-, \gamma_+] \times [\psi^{-1}(z), \gamma \psi^{-1}(z)] ) =\ell_{\Omega_\rho}(\gamma).
     \end{align*}

Since $i(\mu_{\Omega_\rho} , \gamma) = \ell_{\Omega_\rho}(\gamma)$, by~\cite[Th\'eor\`eme~2]{Otal90:SpectreMarqueNegative} we have $\mu_{\Omega}=\psi_* \mu_{\rho}$.
 \end{proof}

 At this point, we observe that the results obtained in this paper for $X_{\mu}$ equipped with an underlying hyperbolic structure $X$ would also hold if the underlying metric was a Hilbert metric coming from a strictly convex projective structure on the surface.
 In view of this, one could have defined, from the get go, dual spaces $X_{\mu}$ in the setting of Hilbert metrics coming from strictly convex real projective structures $X$ on the surface $S$.
 We choose not to write things in such generality in this paper, since the only place we allow $X$ to be a Hilbert metric (as opposed to a hyperbolic metric) is in this example.

 \begin{proposition}
   The dual space $(X_{\mu_{\Omega}}, d_{\mu_{\Omega}})$ obtained by letting $X$ be the Hilbert metric induced by a convex projective structure $\Omega_{\rho}$ satisfies Theorems~\ref{thm:hyp_intro}, \ref{thm:intro_treegraded}, \ref{thm:action_intro}, \ref{thm:continuity_intro} and \ref{thm:homeointroduction}.

   \label{prop:general}
   \end{proposition}
\begin{proof}
We list the properties used about the metric $X$ used throughout the paper, justify why they are also satisfied when $X$ is a Hilbert metric coming from a strictly convex projective structure, and point precisely to where they have been used. 

\begin{enumerate}
    \item Geodesics are unique and they are the same as length minimizers. This is true for $X$ coming from a strictly convex projective structure, as shown in  (see~\cite{Bea99:HilbertDomain}). This is assumed in the choice of metric in Section~\ref{intro:curr}, in Propositions~\ref{prop:zeromeasure} and~\ref{prop:contdistance}, as well as Section~\ref{sec:dualis0hyperbolic}, Section~\ref{sec:actions} and Section~\ref{sec:completeness}.
    \item There exists a continuous geodesic flow for $X$, and geodesic currents can be seen also as geodesic flow invariant measures under this flow (see~\cite[Section~6]{KP14:Convex} as well as~\cite[Chapter~4]{Cra11:Dynamics}). This is assumed in~Proposition~\ref{prop:zeromeasure}. Continuity properties of the flow are assumed in Lemma~\ref{lem:distancecompare}. \item Ergodicity properties and the existence of an ergodic Borel geodesic flow invariant measure of full support are used in Proposition~\ref{prop:weakflow}. In this setting, one can use the Bowen-Margulis measure, see~\cite[4.2]{Cra11:Dynamics}.
    \item Loxodromic elements have dense set of points on the boundary, and act by north-south dynamics, both of which are true by $\delta$-hyperbolicity~(see~\cite[Chapter~11, Proposition~2.4]{CDR90:NotesGroupes}), since strictly convex real projective Hilbert metrics are $\delta$-hyperbolic~\cite[Theorem~4.5]{Ben04:Convexes}. This is assumed in Lemma~\ref{prop:deltafinite}. 
    \item Uniqueness of geodesic paths in a given base-point homotopy class, and uniqueness of base-point homotopy classes of loops of length at most $L$~(see~\cite[Page~9]{KP14:Convex}).
\end{enumerate}
 \end{proof}

From this, we obtain the following result.

\begin{proposition}
The map $\dev \circ \pi_{\mu}^{-1}$ defines a $\pi_1(X)$-equivariant homeomorphism from $X_{\mu_{\rho}}$ to $\Omega$. 
Furthermore, if we consider the metric on $\wt{X}$ induced by the convex projective structure $\Omega_{\rho}$ on $X$, $\pi_{\mu_{\Omega}}$ induces an isometry from $(\Omega, d_\Omega)$ to $(X_{\mu_{\Omega}}, d_{\mu_{\Omega}})$.
\label{prop:convexproj}
\end{proposition}
\begin{proof}
Since $\mu_\rho$ has no atoms and full support, by Proposition~\ref{thm:homeoprojection} $\pi_{\mu_\rho} \colon \wt{X} \to X_{\mu_\rho}$ is a $\pi_1(X)$-equivariant homeomorphism.
We can thus define the $\pi_1(X)$-equivariant homeomorphism~ $\varphi \coloneqq \dev \circ \pi_{\mu_\rho}^{-1}$.

We show that this map is, in fact, an isometry.
For any points $x,y \in X_{\mu_{\rho}}$, let
$x=\pi_{\mu}(\bar{x})$ and $y=\pi_{\mu}(\bar{y})$, for $\bar{x},\bar{y} \in \wt{X}$.
Note that
\begin{align*}
d_{\mu_{\Omega}}(x,y)&=
d_{\mu_{\Omega}}(\pi_{\mu}(\bar{x}),\pi_{\mu}(\bar{y}))=\mu_{\Omega}(G[\bar{x},\bar{y}])=d_{\Omega}(\bar{x},\bar{y}).
\end{align*}
\end{proof}

\subsection{Duals for hyperbolic/negatively curved Riemannian Liouville current}
\label{ex:LiouvilleDual}
In this subsection, we assume that $X$ is closed.
Recall that the geodesic current $\mathcal{L}_Y^X$  has full support and has no atoms, and it is has been defined as the Liouville current associated to $[(Y,\varphi)] \in \Teich(X)$~(see Subsection~\ref{ex:Liouville} for details).

\begin{lemma}
Given $[(Y,\varphi)] \in \Teich(X)$, the map $\varphi \circ \pi_{\mathcal{L}_Y^X}^{-1}$ induces a $\pi_1(X)$-equivariant homeomorphism between $\wt{Y}$ to
$\wt{X}_{\mathcal{L}_Y^X}$. Furthermore, if $Y=X$, $\pi_{\mathcal{L}_X}$ induces an isometry between $(\wt{X},d_{\wt{X}})$ and $(\wt{X}_{\mathcal{L}_X},d_{\mathcal{L}_X})$.

\label{lem:liouville}
\end{lemma}
\begin{proof}
Note that since $\mathcal{L}_Y^X$ has no atoms and full support, it follows from Proposition~\ref{thm:homeoprojection} that $\varphi \circ \pi_{\mathcal{L}_Y^X}^{-1}$ is a $\pi_1(X)$-equivariant homeomorphism. If $X=Y$, then we write $\mathcal{L}_Y^X=\mathcal{L}_X$, and we have $\pi_{\mathcal{L}_X}(\bar{x})=x$.
Let $x,y \in X_{\mathcal{L}_X}$, and set $\bar{x}=\pi_{\mathcal{L}_X}^{-1}(x)$, $\bar{y}=\pi_{\mathcal{L}_X}^{-1}(y)$.
\begin{align*}
d_{\wt{X}}(\bar{x},\bar{y})&=\mathcal{L}_{X}(G[ \pi_{\mathcal{L}_X}^{-1}(x), \pi_{\mathcal{L}_X}^{-1}
(y)])=d_{\mathcal{L}_X}(x,y).
\end{align*}
\end{proof}

We end this example by remarking that the same argument for a negatively curved Riemannian metric $Z$, and its associated geodesic current $\mathcal{L}_Z$ as defined by Otal (see Subsection~\ref{ex:Liouville}), yields an isometry $\pi_{\mathcal{L}_Z}$ between $(\wt{Z},d_{\wt{Z}})$ and $(\wt{X}_{\mathcal{L}_Z},d_{\mathcal{L}_Z})$.

For example, by~\cite[Proposition~4.2]{OT21:Blaschke}, the Blaschke metric induced by cubic differentials is a negatively curved metric, with the Liouville current associated to a negatively curved Riemannian metric whose dual space is isometric to $X$ equipped with said metric.
One can obtain similar equivalences for other geodesic currents satisfying the Crofton property associated to non-positively curved metrics, as long as the properties in Proposition~\ref{prop:general} are satisfied.

\section{Hyperbolicity}
\label{sec:dualis0hyperbolic}

In this section we prove that the dual spaces $X_{\mu}$ are $\delta$-hyperbolic metric spaces.

Recall (see~Definition~\ref{def:boxes}) that $I_{a,b}$ denotes a generalized ordered interval in $\partial \wt{X}$, and
a box of geodesics was defined as any subset of $\mathcal{G}(\wt{X})$ of the type $B=I_{a,b} \times I_{c,d}$.
Recall also $\mathcal{B}$ denotes the family of all boxes of geodesics.

\begin{definition}[opposite box]
Given a box of geodesics $B=I_{a,b} \times I_{c,d}$, its \emph{opposite box} is defined as $B^{\perp}=I_{d,a} \times I_{b,c}$, so that the intervals $I_{a,b}, I_{c,d}, I_{d,a}, I_{b,c}$ partition $\partial \wt{X}$.
See Figure~\ref{fig:box}.
\label{def:opposite}
\end{definition}

    \begin{figure}[h!]
\centering{
\resizebox{160mm}{!}{\Huge{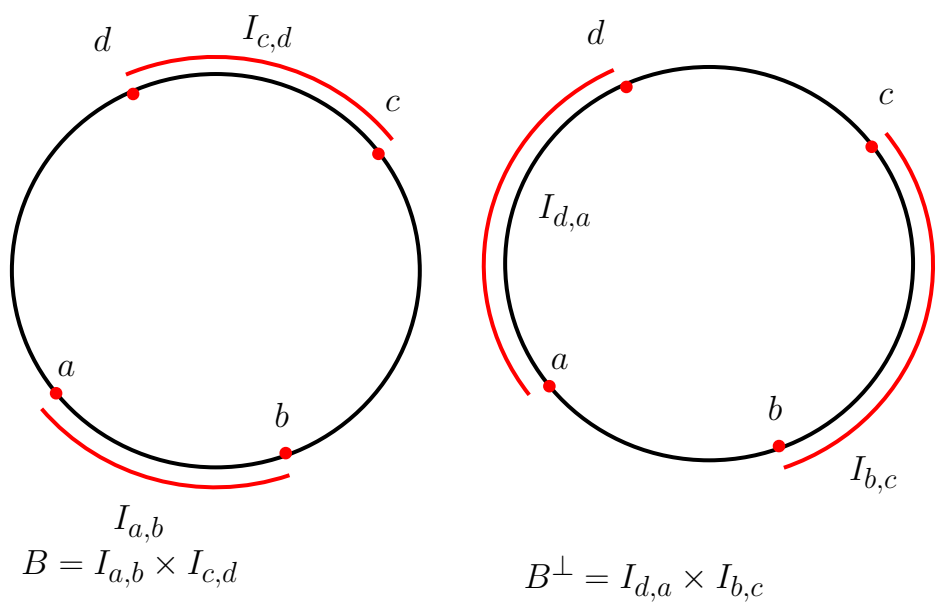}}
    \caption{Sketch of a box of geodesics and its corresponding opposite box. }\label{fig:box}
}
\end{figure}

The following lemma is straightforward and not new. The second equation in Lemma~\ref{lem:otalclaim} appears in work of Otal~\cite[Page~154]{Otal90:SpectreMarqueNegative}, without proof. We provide a proof here, for completeness.

We introduce the following notation. Given two geodesics segments $t$ and $t'$, let $G(t , t')$ be the set of geodesics intersecting both $t$ and $t'$ transversely. We will refer to sets of this type as \emph{double transversals}.
Let $\overline{x},\overline{y},\overline{z},\overline{w} \in \wt{X}$ appear counter-clockwise as vertices of an embedded geodesic quadrilateral on $\wt{X}$ with sides $[\overline{x},\overline{y})$, $[\overline{z},\overline{w})$, $[\overline{y},\overline{z})$, and $[\overline{w},\overline{x})$.

\begin{lemma}
Given the setup described above, we have 
\begin{align*}
\mu(G([\overline{x},\overline{w}), [\overline{y},\overline{z}))) + \mu(G((\overline{x},\overline{w}], (\overline{y},\overline{z}])) &= \\
&\mu(G[\overline{x},\overline{z})) + \mu(G[\overline{w},\overline{y})) \\
&- \mu(G([\overline{w},\overline{z})) - \mu(G([\overline{x},\overline{y}))
\end{align*}
and
\begin{align*}
\mu(G((\overline{x},\overline{w}], (\overline{y},\overline{z}])) + \mu(G([\overline{x},\overline{w}), [\overline{y},\overline{z}))) &= \\
&\mu(G(\overline{x},\overline{z}]) + \mu(G(\overline{w},\overline{y}]) \\
&- \mu(G((\overline{w},\overline{z}]) - \mu(G((\overline{x},\overline{y}])
\end{align*}
\label{lem:otalclaim}
\end{lemma}
\begin{proof}

We have the following partitions:
\begin{subequations}
\begin{align}
\begin{split}
G([\overline{x},\overline{z})) = G([\overline{x},\overline{z}), [\overline{x},\overline{y})) \cup G([\overline{x},\overline{z}), (\overline{z},\overline{y}])
\label{eq1}
\end{split}\\
\begin{split}
G([\overline{w},\overline{y}))=G([\overline{w},\overline{y}), [\overline{w},\overline{z})) \cup G([\overline{w},\overline{y}), (\overline{y},\overline{z}])
\label{eq2}
\end{split}\\
\begin{split}
G([\overline{w},\overline{z}))=G([\overline{w},\overline{z}), [\overline{w},\overline{y})) \cup G([\overline{w},\overline{z}), (\overline{z},\overline{y}])
\label{eq3}
\end{split}\\
\begin{split}
G([\overline{x},\overline{y}))=G([\overline{x},\overline{y}), [\overline{x},\overline{z})) \cup G([\overline{x},\overline{y}), (\overline{y},\overline{z}])
\label{eq4}
\end{split}
\end{align}
\end{subequations}
We now applying the measure $\mu$ to all the equations, add the equations resulting from Equation~\ref{eq1} and~\ref{eq2}, and subtract this from the equations resulting from Equation~\ref{eq3} and~\ref{eq4}, we get the result. The other equation follows in a similar way.
\end{proof}

\begin{definition}[double transversals and boxes for 4-tuples]
In what follows, compare Figure~\ref{fig:opposite} for illustrations.
Let $\overline{x},\overline{y},\overline{z},\overline{w}$ be four distinct points in $\wt{X}$. Up to relabeling, we assume that they appear as vertices of an embedded geodesic quadrilateral ordered counter-clockwise on $\wt{X}$. Consider the oriented hyperbolic geodesic $\gamma_1$ connecting $\overline{x}$ to $\overline{y}$, and the oriented hyperbolic geodesic $\gamma_2$ connecting $\overline{w}$ to $\overline{z}$. Let $\delta_1$ be the oriented hyperbolic geodesic connecting $\overline{z}$ to $\overline{y}$ and $\delta_2$ be the oriented geodesic connecting $\overline{w}$ to $\overline{x}$. We define three sets of geodesics associated to the tuple $(\overline{x},\overline{y},\overline{z},\overline{w})$. 
\begin{enumerate}
    \item Let $b^+_{\overline{x},\overline{y},\overline{z},\overline{w}}$ be the box of geodesics defined by $(\gamma_2^-, \gamma_1^-] \times [\gamma_1^+,\gamma_2^+)$ and 
    $b^-_{\overline{x},\overline{y},\overline{z},\overline{w}}$ the box defined by $[\gamma_2^-, \gamma_1^-) \times (\gamma_1^+,\gamma_2^+]$
    \item Let ${G^{+}}_{\overline{x},\overline{y},\overline{z},\overline{w}}$ denote a \emph{double transversal}, defined as the set of geodesics intersecting both $[\overline{x},\overline{w})$ and $[\overline{y},\overline{z})$. Let $G^{-}_{\overline{x},\overline{y},\overline{z},\overline{w}}$ denote the set of geodesics intersecting $(\overline{x},\overline{w}]$ and $(\overline{y},\overline{z}]$ transversely.
    Let, also $(G^+)^{\perp}$ denote the set of geodesics intersecting  $[\overline{x},\overline{y})$ and $[\overline{w},\overline{z})$, and $(G^-)^{\perp}$ the set of geodesics intersecting $(\overline{x},\overline{y}]$ and $(\overline{w},\overline{z}]$
    \item Let $B^+_{\overline{x},\overline{y},\overline{z},\overline{w}}$ be the box of geodesics defined by $(\delta_1^-, \delta_2^-] \times [\delta_2^+,\delta_1^+)$, and 
    $B^-_{\overline{x},\overline{y},\overline{z},\overline{w}}$ be the box of geodesics defined by $[\delta_1^-, \delta_2^-) \times (\delta_2^+,\delta_1^+]$.
\end{enumerate}
\label{def:boxes_transverse}
\end{definition}

The following result follows directly from Definition~\ref{def:boxes_transverse}.
\begin{proposition}
Given the setting as described above, dropping the subscripts, we have
\[
b^{\pm} \subseteq G^{\pm} \subseteq B^{\pm}.
\]
Moreover, we have
\[
(B^{\perp})^{\pm} \subseteq (G^{\perp})^{\pm} \subseteq (b^{\perp})^{\pm}.
\]
\label{prop:boxes_transverse}
\end{proposition}

We define the following two quantities.

\begin{definition}[$\delta_{\mu}$ with boxes]
For a given geodesic current $\mu$, define
\[
\delta^B_{\mu} =\frac{1}{2}\sup_{B  \in \B }\min \{ \mu(B^+) + \mu(B^-) , \mu((B^{\perp})^+) +  \mu((B^{\perp})^-)\}
\]
\label{lem:deltaB}
\end{definition}

We observe that $B \times B^{\perp} \subset \mathfrak{I}$ where $\mathfrak{I}$ is the subset of $\mathcal{G}(\wt{X}) \times  \mathcal{G}(\wt{X})$ consisting of transversely intersecting geodesics, used in the definition of intersection number of geodesic currents (see~Definition~\ref{def:intersection}).
Thus, $\delta_{\mu}$ is giving another measure, related to intersection number, of `how far is $\mu$ from being a measured lamination'.

\begin{definition}[$\delta_{\mu}$ with double transversals]
For a given geodesic current $\mu$, define
\[
\delta^G_{\mu} =\frac{1}{2} \sup_{G }\min \{ \mu(G^+) + \mu(G^-), \mu((G^{\perp})^+) + \mu((G^{\perp})^-) \}
\]
where $G$ ranges over all $G_{x,y,z,w}$ with $x,y,z,w$ all distinct. 
\label{lem:deltaG}
\end{definition}

\begin{lemma}
For any geodesic current $\mu$,
\[
\delta_{\mu}^B=\delta_{\mu}^G.
\]
\label{lem:equal_deltas}
\end{lemma}
\begin{proof}
By the inequalities above we have that, for fixed $x,y,z,w$,
\[
b^{\perp} \supseteq G^{\perp} \supseteq B^{\perp}.
\]
Since ranging over all $x,y,z,w$, $B$ and $b$ exhaust all possible boxes (and same for $B^{\perp}$,$b^{\perp}$), we have, taking measure $\mu$ and supremum over all distinct $x,y,z,w$,
that
\[
\delta_{\mu} ^B\geq \delta_{\mu}^G \geq\delta_{\mu}^B.
\]
as we wanted to show.
\end{proof}

    \begin{figure}[h!]
\centering{
\resizebox{130mm}{!}{\Huge{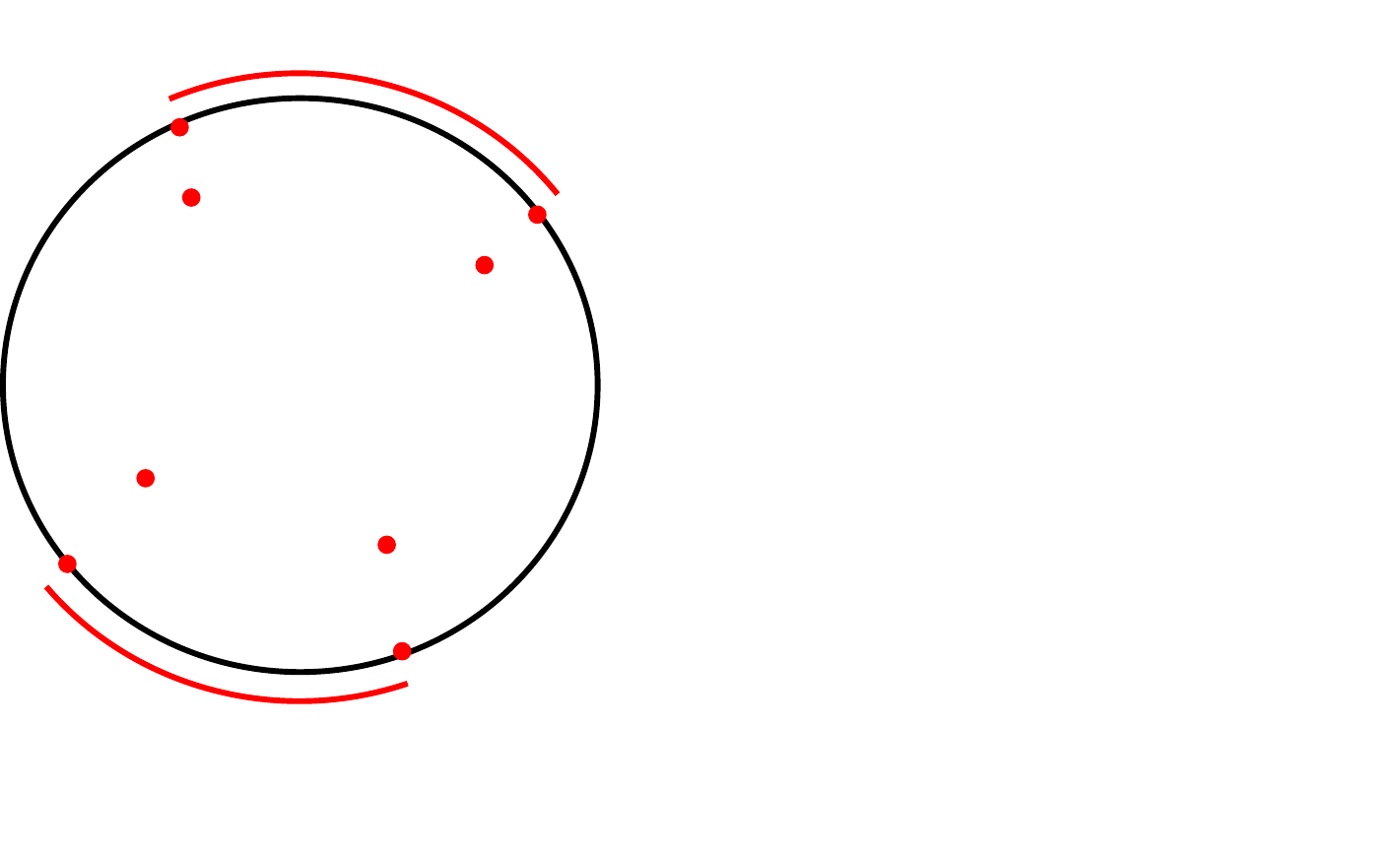}}
    \caption{ }\label{fig:opposite}
}
\end{figure}

Since $\delta^G_{\mu}$ and  $\delta^B_{\mu}$ are the same quantity, we will simply refer to it as $\delta_{\mu}$. For some proofs it will be easier to use one viewpoint or the other.

\begin{proposition}
Let $\mu$ be any geodesic current on $X$. Then $\delta_{\mu}$ is finite.
\label{prop:deltafinite}

\begin{proof}
  Assume $\delta_{\mu}$ is not finite. Hence there exists a sequence of boxes $B'_n$ such that $\mu(B'_n)$ and $\mu({B'_n}^{\perp})$ both diverge. Suppose that
  $B'_n = [a'_n, b'_n) \times [c'_n,d'_n)$ and thus $B'^{\perp}=[b'_n, c'_n) \times [d'_n, a'_n)$, with $a'_n, b'_n, c'_n, d'_n$ are ordered counter clockwise for every $n$. Consider a compact subset $K$ in $\wt{X}$ containing a fundamental region for the $\pi_1 (X)$-action on $\wt{X}$.
For every box $B'_n = [a'_n , b'_n] \times [c'_n , d'_n]$ let $m(B'_n)$ be the \emph{center} of $B'_n$, i.e. the intersection between the two `diagonal' geodesics joining $a'_n$ to $c'_n$ and $b'_n$ to $d'_n$. 
Since the $\pi_1 (X)$ action on $\wt{X}$ is cocompact,  for each box $B'_n$ in the sequence, there exists $g_n \in \pi_1 (X)$ such that $g_n  m(B'_n) \in K$. By $\pi_1 (X)$ invariance of the measure $\mu$ we have that since $\mu (B'_n)$ and $\mu ({B'_n} ^\perp)$ diverge, then also $\mu (g_n B'_n )$ and $\mu ( (g_n {B'_n} )^\perp)$ diverge. Let $B_n := g_n B'_n$.

  In order for $\mu(B_n)$ to diverge, $B_n$ has to leave every compact set of the space of geodesics. Hence either $d_{\partial \wt{X}} (d_n , a_n) \to 0$ or $d_{\partial \wt{X}} (b_n , c_n) \to 0$
  Let's assume without loss of generality that distance between $d_n$ and $a_n$ in $\partial \wt{X}$ tends to $0$ as $n$ goes to infinity.
   By compactness of $S^1$, each of the sequences $(a_n), (b_n), (c_n), (d_n)$ admits a convergent subsequence. Therefore, up to subsequences, we can assume that
   $a_n, d_n \to x$, $b_n \to b$ and $c_n \to c$.

 Now we make some case distinctions
  \begin{enumerate}
      \item $b \neq c$ and $b,c  \neq x$;
      \item  $x \neq c$ and $x \neq b$ and $b = c$;
      \item $x$ is equal to $b$ or $c$ (or both).
  \end{enumerate}
  In cases (1) and (2) we immediately reach a contradiction because we can find a compact set $C \subseteq \mathcal{G} (\wt{X})$ in which $B_n ^\perp$ is included for all $n$ large enough. This is in contradiction with the fact that both $\mu (B_n)$ and $\mu (B_n ^\perp)$ diverge.

Finally, let $B_n$ a sequence of boxes falling into case (3), i.e. three among the vertices $a_n, b_n, c_n, d_n$ converge to the same point on $\partial \wt{X}$.
In this case we also get a contradiction because $m(B_n)$ would escape all compact sets of $\wt{X}$, but the center of the boxes $B _n$  belongs to the compact set $K$ for all $n$. This completes the proof.

\end{proof}
\end{proposition}

Given a geodesic current $\mu$, we say a box of geodesics $B$ is \emph{$\mu$-generic} if $\mu(\partial B)=0$. Let $\mathcal{B}^{\mu} \subset \mathcal{B}$ denote the subset of $\mu$-generic geodesic boxes.

The following is an easy but crucial observation.

\begin{lemma} 
In the definition of $\delta_{\mu}$ we can restrict to $\mu$-generic boxes, i.e.,
\[
\delta_{\mu} =\sup_{B  \in \Bg{\mu} }\min \{ \mu(B) , \mu(B^{\perp}) \}
\]
where $\Bg{\mu} \subset \B$ consists of boxes $B$ so that $B$ and $B^{\perp}$ are $\mu$-generic.
\label{lem:deltageneric}
\end{lemma}
\begin{proof}
Suppose that $B=[a,b) \times [c,d)$, and $\mu(P(a))>0$. By Lemma~\ref{lem:discreteatoms}, there exists a lift of a closed geodesic in the support.
First, note that atoms of $\mu$ are concentrated on $C_{\mu}$, the set of lifts of closed geodesics in the support of $\mu$, and this set $C_{\mu}$ is countable. We claim that for every $\epsilon> 0$, there exists a point $a' \in \partial \wt{X}$, so that the pencil $P(a')$ has no atoms and the box $B'=[a',b) \times [c,d)$ satisfies $\mu(B')-\mu(B)<\epsilon$. By the observation that $C_{\mu}$ is countable, it follows that there are uncountably many $a' \in \partial \wt{X}$ so that $P(a')$ has no atoms. By outer regularity of the measure $\mu$, we can take an open box $B''=(a'',b) \times (c',d)$ containing $B$ so that $\mu(B'')-\mu(B)<\epsilon$. By taking the point $a'$ above in $(a'',b)$, and letting $B'=[a',b) \times [c,d)$ we have $\mu(B'')\geq \mu(B') \geq \mu(B)$, and so $\mu(B') - \mu(B)<\epsilon$, as desired.
Moreover, $\partial B'^{\perp}$ has the same atoms as $\partial B^{\perp}$. By repeating the same argument with $B^{\perp}$, we can guarantee that $B^{\perp}$ is also $\mu$-generic.
\end{proof}

The following is a restatement of the Gromov 4-point condition for $\delta$-hyperbolic spaces~\cite[Page~410]{BH11:NonPosCurvature}.

\begin{definition}[$\delta$-hyperbolicity]
A metric space $(X,d)$ is \emph{$\delta$-hyperbolic} if and only if for any 4-tuple of points $x,y,z,w \in X$, among the following three quantities
\begin{itemize}
    \item $d(x,y) + d(z,w)$
    \item $d(x,z) + d(y,w)$
    \item $d(y,z) + d(x,w)$
\end{itemize}
the two largest of them are within $2\delta$ of each other. If $\delta=0$, then it means that the maximum appears at least twice.
\label{lem:hyperbolic}
\end{definition}

\begin{theorem}
If $\mu$ is a geodesic current then $X_{\mu}$ is a $\delta_{\mu}$-hyperbolic space in the sense of Definition~\ref{lem:hyperbolic}, and $\delta_{\mu}$ is the optimal $\delta$-hyperbolicity constant.
\label{prop:dualishyp}
\end{theorem}
\begin{proof}
We prove that $X_{\mu}$ satisfies the $\delta_{\mu}$-hyperbolic 4-point condition.
Let $x,y,z,w \in X_{\mu}$ be four arbitrary points, and, up to relabeling, assume that the geodesic segments $[x,z]$ and $[y,w]$ intersect. 
Assume first that $d(x,z) + d(y,w) > d(x,y) + d(z,w)+2\delta$. We want to show that $d(x,z) + d(y,w) \leq d(y,z) + d(x,w)+2\delta$.
Notice how the first inequality is equivalent to 
\begin{align*} 
\frac{1}{2} & [\mu (G[\overline{x},\overline{z})) + \mu(G[\overline{y},\overline{w}))- \mu(G[\overline{x},\overline{y})) - \mu(G[\overline{z},\overline{w})) ] + \\
& \frac{1}{2} [\mu (G(\overline{x},\overline{z}]) + \mu(G(\overline{y},\overline{w}])- \mu(G(\overline{x},\overline{y}]) - \mu(G(\overline{z},\overline{w}]) ] > 2 \delta
\end{align*}
which is equivalent to
\[
\frac{1}{2} [\mu((G^{\perp})^+) + \mu((G^{\perp})^-) ] > \delta
\]
where $G=G_{x,w,y,z}$, by Lemma ~\ref{lem:otalclaim}.
On the other hand, $d(x,z) + d(y,w) \leq d(y,z) + d(x,w)+2\delta$, is equivalent to

\begin{align*}
\frac{1}{2} [ \mu (G[\overline{x},\overline{z})) + \mu(G[\overline{y},\overline{w}))- \mu(G[\overline{y},\overline{z})) - \mu(G[\overline{x},\overline{w})) ] + \\
\frac{1}{2} [ \mu (G(\overline{x},\overline{z})] + \mu(G(\overline{y},\overline{w}])- \mu(G(\overline{y},\overline{z}]) - \mu(G(\overline{x},\overline{w}]) ] \leq  2 \delta
\end{align*}
 i.e., $\frac{1}{2} [ \mu (G^+), \mu (G^-) ] \leq \delta$, by Lemma ~\ref{lem:otalclaim}.

Suppose, for the sake of contradiction, that  $\frac{1}{2} [ \mu (G^+), \mu (G^-) ] >  \delta$. Then, 
\begin{align*}
\min \{ \mu(G^+) + \mu(G^-), \mu((G^{\perp})^+) + \mu((G^{\perp})^-)\} & > \delta \\
= \sup_G \min \{ \mu(G^+) + \mu(G^-), \mu((G^{\perp})^+) + \mu((G^{\perp})^-)\} 
\end{align*}
which is a contradiction.
If, instead, we had $d(x,z) + d(y,w) > d(y,z) + d(x,w)+2\delta$, then, similarly as above, this would mean $\frac{1}{2} [ \mu (G^+), \mu (G^-) ] > \delta$, and if $\frac{1}{2} [\mu((G^{\perp})^+) + \mu((G^{\perp})^-) ]>\delta$, we would again contradict that $\delta$ is a supremum, so we must have $\frac{1}{2} [\mu((G^{\perp})^+) + \mu((G^{\perp})^-) ] \leq \delta$, and thus $d(x,z) + d(y,w) \leq d(x,y) + d(z,w)+2\delta$, so the 4-point condition is proven.
Since $\delta_{\mu}$ is defined in terms of a supremum, it follows it is the optimal $\delta$-hyperbolicity constant.
\end{proof}

\begin{corollary}
$X_{\mu}$ is $0$-hyperbolic if and only if $\mu$ is a measured lamination.
\label{cor:0hyp}
\end{corollary}
\begin{proof}
By~\cite[Proposition~2.1]{BIPP21:Crossratios}, $\mu$ is a measured lamination if and only if, for every box $B \subset G(X)$, we have $\mu(B)\mu(B^{\perp})=0$. This last equality is true if and only if $\min \{ \mu(B)\mu(B^{\perp})\}=0$.
If $\mu$ is a measured lamination, then $\delta_{\mu}=0$, and thus by Theorem~\ref{prop:dualishyp}, $X_{\mu}$ is $0$-hyperbolic.
If $X_{\mu}$ is $0$-hyperbolic, since $\delta_{\mu}$ is the smallest hyperbolicity constant, we must have $\delta_{\mu}=0$. Since $\delta_{\mu}$ is defined in terms of a supremum, this implies that for all boxes $B \subset \G(\wt{X})$, $\mu(B)\mu(B^{\perp})=0$, and hence $\mu$ must be a lamination.
\end{proof}

As a consequence, we recover the following well-known result.

\begin{corollary}
If $\mu$ is a Liouville current $\mathcal{L}_Y$, for $Y \in \Teich(S)$, then its optimal hyperbolicity constant is~$\delta_{\mu}=\log(2)$.
\label{cor:hypplanedelta}
\end{corollary}
\begin{proof}
By \cite[Theorem~13]{Bonahon88:GeodesicCurrent}, a geodesic current $\mu$ is a Liouville current $\mathcal{L}_X$ for some hyperbolic structure $Y$, if and only if $e^{-\mu(B)} + e^{-\mu(B^{\perp})}=1$. Maximizing the function $\delta_{\mu}$ subject to this relation, yields $\log(2)$. Thus, by Theorem~\ref{prop:dualishyp}, the result follows.
\end{proof}
\begin{remark}
Recall that when $\mu$ is the hyperbolic Liouville current of $X$, then $X_{\mu}$ is isometric to the hyperbolic plane by Example~\ref{ex:LiouvilleDual}.
Observe that $\log(2)$ is the optimal $\delta$-hyperbolicity constant which was computed for example in~\cite[Corollary~5.4]{BS16:StrongHyp}.
Corollary~\ref{cor:hypplanedelta} recovers this result.
\label{rmk:hypplanedual}
\end{remark}

\begin{proposition}
$\delta_{\mu}$ is a lower semi-continuous function on geodesic currents.
\end{proposition}
\begin{proof}
Let $\mu_i \to \mu$ in the weak$^*$-topology.
Recall that, by Lemma~\ref{lem:deltageneric}, in the definition of $\delta_{\mu}$, we can restrict to $\mu$-generic boxes $B \subset G(X)$ without affecting the supremum.
For any such $B$, let 
\[
f_B(\mu) \coloneqq \min \{ \mu(B), \mu(B^{\perp}) \}.
\]
By the Portmanteau theorem~\cite[Theorem~30.12]{Bau01:Measure}, $f_B(\mu_i) \to f_B(\mu)$, so $f_B$ is a continuous function on geodesic currents. Since $\delta_{\mu}=\sup_{B} f_B(\mu)$
is a supremum of continuous functions, by~\cite[Theorem~10.3]{vR82:RealAnalysis} it must be lower semi-continuous. 
\end{proof}

\begin{example}
We discuss an example.
 Take a sequence $\mu_i$ of scaled hyperbolic Liouville currents $\mu_i \coloneqq a_i \mathcal{L}_{X_i}$ converging to a measured lamination  $\lambda$ in the weak$^*$-topology. Then, we must have $a_i \to 0$ (see argument at the bottom of~\cite[Page~152]{Bonahon88:GeodesicCurrent}).
By Corollary~\ref{cor:hypplanedelta}, $\delta_{\mu_i}=a_i\log(2)$ for all $i \in \mathbb{N}$, and thus $\lim_i a_i \log(2)=0$.
On the other hand $\delta_{\lambda}=0$, also.
\end{example}

To finish this section, we show a few inequalities between the $\delta$-hyperbolicity constants of $X_{\mu}$ and the ones of its subcurrents according to the decomposition theorem for geodesic currents Theorem~\ref{thm:structurecurrents}. 
In Section~\ref{sec:decomposition} we will see that $X_{\mu}$ decomposes as a graph of spaces with vertices the duals of its subcurrents $\mu_i$.
The following inequalities relate the $\delta$-hyperbolicity constants of the components of the space to those of its pieces. The proof is straightforward.

\begin{lemma}
Let $\mu,\mu_1,\mu_2$ be geodesic currents so that $\mu=\mu_1+\mu_2$, and let $B \subseteq \mathcal{G}(\wt{X})$ be a box of geodesics. 
\begin{enumerate}
    \item We have \[\min \{ \mu(B), \mu(B^{\perp}) \} \geq \min \{ \mu_1(B), \mu_1(B^{\perp}) \} +  \min \{ \mu_2(B), \mu_2(B^{\perp}) \}.\]
    \item If, furthermore $\mu_1 \perp \mu_2$, then
    \[\min \{ \mu(B), \mu(B^{\perp}) \} = \min \{ \mu_1(B), \mu_1(B^{\perp}) \} +  \min \{ \mu_2(B), \mu_2(B^{\perp}) \}.
    \]
\end{enumerate}
\label{lem:ortho}
\end{lemma}

\begin{proposition}
Let $\mu$ be a geodesic current which decomposes according to the structural Theorem~\ref{thm:structurecurrents}, $\mu=\sum_{i}\nu_i + \sum_{j} \lambda_j + \sum_{k} a_k s_k$, where $s_k$ is a geodesic current supported on a simple closed curve, $\lambda_j$ is a non-discrete measured lamination, and $\nu_i$ is a geodesic current which is filling in a subsurface, and all the currents in the decomposition have orthogonal supports (in the sense that $i(\supp(\alpha),\supp(\beta))=0$, according to \cite[Proposition~3.2]{BIPP21:Currents}).
Then, we have
\begin{equation}
\min \{ \mu(B), \mu(B^{\perp})\} = \sum_{i} \min \{ \nu_i(B),\nu_i(B^{\perp})\},
\label{eq:struct_mu}
\end{equation}
and, thus, for every $i$, we have 
\begin{equation}
\delta_{\nu_i} \leq \delta_{\mu} \leq \sum_{i} \delta_{\nu_i}.
\label{eq:struct_delta}
\end{equation}

\label{prop:deltadecomp}
\end{proposition}
\begin{proof}
By the structural theorem~\cite[Proposition~3.2]{BIPP21:Currents} all the currents in the decomposition are pairwise orthogonal. 
Thus, equation~\ref{eq:struct_mu} follows by Proposition~\ref{lem:ortho}, and noting that by~\cite[Proposition~2.1]{BIPP21:Currents}, it follows that
\[\min \{ \lambda_j(B),\lambda_j(B^{\perp}) \}=0\]
and \[\min \{ s_k(B),s_k(B^{\perp}) \}=0.\]
The first inequality follows from the fact that, by Lemma~\ref{lem:ortho} (part (i)), $\min \{ \mu(B), \mu(B^{\perp})\} \geq \min \{ \nu_i(B),\nu_i(B^{\perp})\}$ for every $i$, and for any two real valued functions $f,g$ so that $f \leq g$, we have $\sup f \leq \sup g$.
The second inequality in Equation~\ref{eq:struct_delta} follows by Lemma~\ref{lem:ortho} (part (ii)) and the fact that for two real valued functions $f,g$, we have $\sup f + g \leq \sup f + \sup g$. 
\end{proof}

\section{Decomposition theorem}
\label{sec:decomposition}

We begin by defining the set of \emph{special geodesics}, as introduced in~\cite{BIPP21:Currents}. Given a geodesic current $\mu$ on a compact hyperbolic surface $X$, let
\begin{align*}
 \mathcal{E}_\mu := \{ & c \subset X : c \text{ closed geodesic such that } i(\mu , c ) =0 \\
& \text{ and for every } c' \subset X \text{ closed geodesic, }i(\mu, c') >0 \text{ whenever } i (c, c' ) >0  \}   
\end{align*}

The set $\mathcal{E}_\mu = \{ s_1 , \dots , s_n \}$ is a finite set of pairwise disjoint simple closed geodesics, called \emph{special curves}, which decomposes $X$ in subsurfaces with geodesic boundary
\begin{equation}
\label{surfacedecomposition}
X = \bigcup_i X_i.
\end{equation}

Given a current $\mu$ on $X$, recall that the systole of $\mu$ relative to $X_i$ is
\begin{equation*}
\mathrm{sys}_{X_i} (\mu) \coloneqq \inf \{ i (\mu, c) : c \in \Curves(X_i - \partial X_i)\}. 
\end{equation*}

We now state the decomposition theorem for geodesic currents, as proven in \cite[Theorem~1.2]{BIPP21:Currents}.
\begin{theorem}[Decomposition Theorem for Geodesic Currents]
\label{thm:structurecurrents}
Any current $\mu$ on $X$ decomposes as
\begin{equation}
    \label{currentdecomposition}
\mu=\sum_{i=1}^n \mu_i + \sum_{j=1}^m a_j s_j
\end{equation}
where each non-zero $\mu_i$ is a geodesic current supported on $X_i$ and $\sum_{j=1}^m a_j s_j$ is a weighted simple multi-curve, and the weights $a_j$ need not be positive.

Moreover, for each non-zero $\mu_i$ we either have
\begin{itemize}
    \item (type 1) $\mathrm{sys}_{X_i}(\mu_i) > 0$;
    \item (type 2) $\mu_i$ is a measured lamination compactly supported on the interior of ${X}_i$ and intersecting every curve in $X_i$.
\end{itemize}
\end{theorem}

\begin{remark}
For the remaining of this paper we will refer to currents which fall into the first case as \emph{subcurrents of type 1}, the ones falling into the second case will be referred as \emph{subcurrents of type 2}, and the weighted simple curves in the special simple multi-curve will be referred as \emph{subcurrents of type 3}.
\end{remark}
    \begin{figure}[h!]
\centering{
\resizebox{160mm}{!}{\Huge{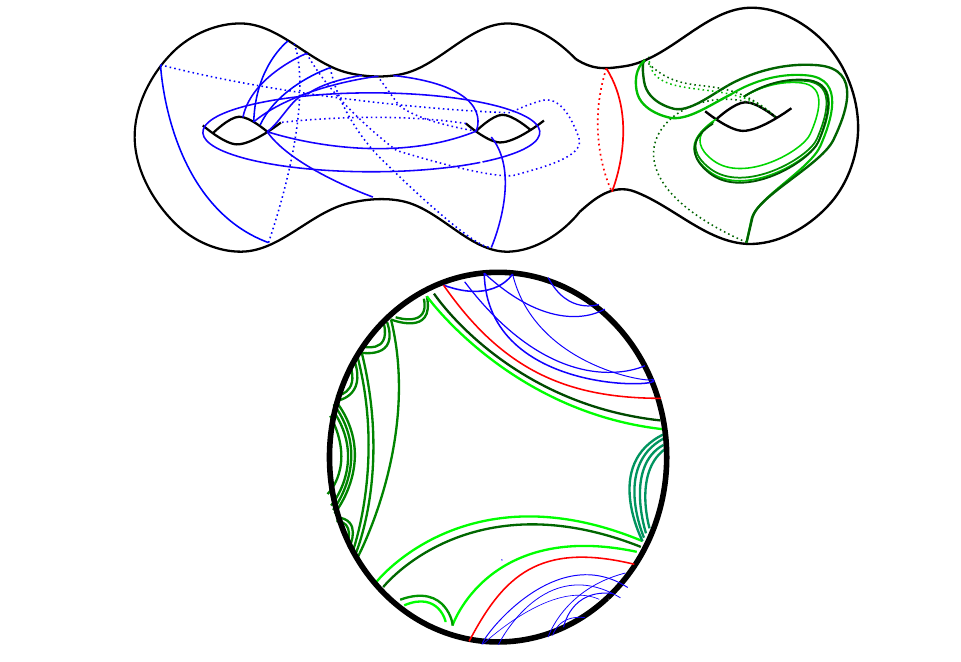}}
    \caption{The support of a current on a surface, and its corresponding lift in the universal cover.}
           \label{fig:decomposition}
}
\end{figure}

Figure \ref{fig:decomposition} shows a sketch of a genus $3$ surface with a geodesic current $\mu$ whose special multi-curve $m$ consists of one single geodesic (red curve, separating the left two handles from the right handle), and yields two subsurfaces. The left one, genus 2, supports a filling geodesic current $\mu_1$ (in blue) within that subsurface. The right one, of genus 1, supports a non-discrete measured lamination $\mu_2$ (in green). The lower figure shows a part of the support of the geodesic current $\mu$ in the universal cover. The lifts of $m$ separate $\wt{X}$ into countably many regions. In the figure, three are depicted, the central one is a region corresponding to the support of $\mu_2$, whereas the upper and lower ones correspond to the support of $\mu_1$.

\begin{definition}[subdual]
Each component $\mu_i$ of $\mu$ is a geodesic current itself, supported on the subsurface $X_i$. On $\wt{X_i} \subseteq \wt{X}$ we can define the pseudo-distance $d_{\mu_i}$ in the same way as for $d_\mu$. The \emph{sub-dual space} $X_{\mu_i}$ is the quotient space $\wt{X_i} / {\{d_{\mu_i} = 0 \}}$ endowed with the $\pi_1 (X_i)$-action.
\label{def:subdual}
\end{definition}

In order to precisely describe the dual $X_\mu$ in terms of the sub-duals $X_{\mu_i}$ we use the notion of \emph{tree-graded space}. 

 For the standard definition of tree-graded space when $(X,d)$ is a geodesic metric space we refer to the work by Drutu-Sapir~\cite{DS05:Treegraded}.

\begin{definition}[Tree-graded space]
A geodesic metric space $(X,d)$ is said to be \emph{tree-graded} with respect to a collection of geodesic subspaces $\mathcal{P}$, called \emph{pieces}, if 
\begin{enumerate}
    \item \emph{axiom pieces}. Given two distinct pieces $P_1 , P_2 \in \mathcal{P}$ the intersection $P_1 \cap P_2$ contains at most one point;
    \item \emph{axiom triangles}. Any simple geodesic triangle in $(X,d)$ is contained in a piece. 
\end{enumerate}
\label{def:treegraded_ds}
\end{definition}

The following can be thought as a local to global principle for geodesics in a tree graded space.
\begin{definition}[Piece-wise geodesic]
Let $(X,\mathcal{P})$ be a tree graded space.
Suppose that the pieces $P_k \in \mathcal{P}$ are geodesics with respect to the restricted metric.
Let $\gamma=\gamma_1 \gamma_2 \cdots \gamma_{2m}$ be a curve in the tree-graded space $(X,\mathcal{P})$ which is a composition of geodesics $\gamma_k$ in $X$. 
Suppose that all geodesics $\gamma_{2k}$ with $k \in \{1, \cdots, m-1 \}$ are non-trivial and for every $k \in \{1, \cdots, m\}$
the geodesic $\gamma_{2k}$ is contained in a piece $P_k$ while for every $k \in \{0, \cdots,m-1\}$ the geodesic $\gamma_{2k+1}$
intersects $P_k$ and $P_{k+1}$ only in its respective endpoints. In addition assume that if $\gamma_{2k+1}$ is empty then
$P_k \subset P_{k+1}$. We call this $\gamma$ a \emph{piece-wise geodesic}.
\end{definition}

The next proposition is Lemma~\cite[Lemma~2.28]{DS05:Treegraded}.

\begin{proposition}
A curve $\gamma$ in a tree graded space $(X,\mathcal{P})$ is a geodesic if and only if it is a piece-wise geodesic.
\label{prop:geodesictreegraded}
\end{proposition}

Geodesics in a tree graded space can be then thought of as concatenations of geodesics within pieces and geodesics in the \emph{transversal trees} $T_x$. Here $T_x$ denotes the set of points $y \in X$ that can be connected to $x$ by a geodesic intersecting each piece at most once (see~\cite[Lemma~2.14]{DS05:Treegraded} to see why these are trees). 

Our goal is to show that a dual space $X_{\mu}$ is a tree graded space where its pieces can be isometrically identified with the duals $X_{\mu_i}$ of the subcurrents of $\mu$.

However, the dual spaces have not been endowed with a canonical geodesic structure in general (see Subsection~\ref{subsec:geodesics}).
Because of that, we will give a more general definition of tree-graded space that is not assumed to be geodesic, but coincides with the usual definition of tree graded space when the underlying space $X$ is assumed to be geodesic.

\begin{definition}[(metric) tree-graded space]
A metric space $(X,d)$ is said to be a \emph{(metric) tree-graded} space with respect to a collection of subspaces $\mathcal{P}$, called \emph{pieces}, if 
\begin{enumerate}
    \item \emph{axiom pieces}. Given two distinct pieces $P_1, P_2 \in \mathcal{P}$ the intersection $P_1 \cap P_2$ contains at most one point;
    \item \emph{axiom transversals}. For every $x \in X$, there exists a $0$-hyperbolic metric subspace $T_x \subset X$ containing $x$ and intersecting each piece at most in one point with the following properties:  
    \begin{enumerate}
    \item for every $y \in X$, if $y \in T_x$, then $T_x=T_y$. Every intersection of a piece $P$ with a $T_x$ is called \emph{a point of contact} of $P$.
    \item All points of $T_x$ are cut points except (possibly) the points of contact.
    \end{enumerate}
    \item \emph{axiom contact chain triangle}. We say that two contact points are \emph{adjacent} if they are contained in the same $T_x$ or in the same piece.
    Let a \emph{straight contact chain} be an ordered sequence of adjacent contact points $(y_1,\cdots,y_n)$ where no two consecutive points in the sequence are equal. We say that two points $x,y \in X$ are \emph{connected by a straight contact chain} if there is a straight contact chain $(x,\cdots,y)$. An \emph{interior point} of that chain is any point of the chain different from $x$ and $y$. A \emph{contact chain triangle} $\Delta=xyz$ is a set of three points in $X$ and three straight contact chains connecting the points, one per side. The axiom contact chain triangle says that if a straight contact chain triangle is contained in more than one piece, an interior point of one of the side chains must be shared with another chain.
    \end{enumerate}
\label{def:treegraded_drmg}
\end{definition}
To motivate the role of the $0$-hyperbolic subspaces $T_x$, note that these will be the transversal trees when $X$ is geodesic.
In any case, the following Remark gives the intuitive picture the reader should probably keep in mind.

\begin{remark}[Geometric realization of the $T_x$]
Let $\mu=\sum_{i=1}^n \mu_i + \sum_{j=1}^m a_j s_j$ as in theorem ~\ref{thm:structurecurrents}, with set of special geodesics $\{s_1, \dots , s_n \}$. A \emph{region} is the closure of a connected component of the complement of the lifts of special geodesics in $\wt{X}$. A \emph{piece} $P = R / \sim$ is the restriction of the quotient $\wt{X}/{d_\mu = 0}$ to a region $R$. Two regions  $R_i$ and $R_j$ bounded by the lift of a special geodesic $c$ correspond in the dual to two pieces $P_i$ and $P_j$ joined in a single point $x := [c]$. Let $a_{i,j}$ be the weight of $c$. Hence every arc joining the piece $P_i$ to the piece $P_j$ will pass through $x$ and since there is an $a_{i,j}$ measure accumulated at the point $x_{i,j}$, every such arc will gain an $a_{i,j}$ contribution to its length as soon as it passes by through the point $x$. It follows that we may imagine, loosely speaking, that the pieces $P_i$ and $P_j$ are joined by an edge of length $a_{i,j}$, and any arc joining $P_i$ to $P_j$ passes through such arc. Let us now make this idea precise. Compare Figure~\ref{fig:dual_noded} for what follows.

Given two adjacent pieces $P_i$ and $P_j$ as above, consider their disjoint sum ~$P_i \amalg P_j$, and attach an edge $e_{i,j}$ of length $a_{i,j}$ joining the point $x_i \coloneqq [c] \in P_i$ and $x_j \coloneqq [c] \in P_j$.
At the level of the surface, this corresponds to pinching the special geodesics to points, obtaining a noded surface, and then replacing the nodes by edges of length equal to the weight of the component of the especial multi-curve.

More generally, for a metric tree graded space, one can define a projection map $p \colon X \to T$, where $T$ is a $0$-hyperbolic space obtained by collapsing each piece to a point. This space is $0$-hyperbolic because it consists of the union of the subspaces $T_x$ equipped with the restriction of the distance on $X$.  By Section~\ref{subsec:dual_tree}, this $0$-hyperbolic space can be isometrically embedded into an $\mathbb{R}$-tree. This induces a $\mathbb{R}$-tree structure on the $T_x$.

\label{rmk:geodesic}
\end{remark}

    \begin{figure}[h!]
\centering{
\resizebox{120mm}{!}{\Huge{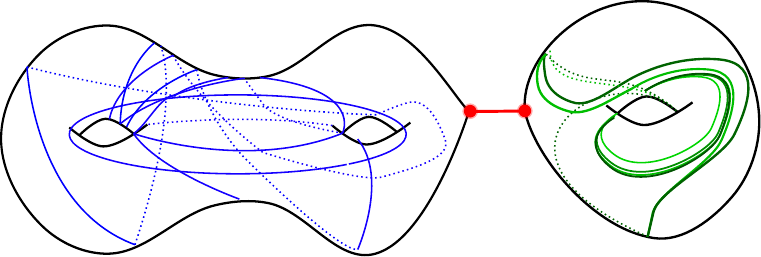}}
    \caption{ The figure shows a sketch of the result of collapsing the special multi-curve $m$ of the geodesic current of Figure~\ref{fig:decomposition}. That is how one can go about putting a geodesic structure on the transversals $T_x$: by pinching $m$ to obtain nodal surfaces, and add edges between them.}
\label{fig:dual_noded}}

\end{figure}

Under the geodesic hypothesis, both definitions above are equivalent. See~Proposition~\ref{prop:eq_defs_treegraded}.

 \begin{figure}[h!]
\centering{
\resizebox{80mm}{!}{\Huge{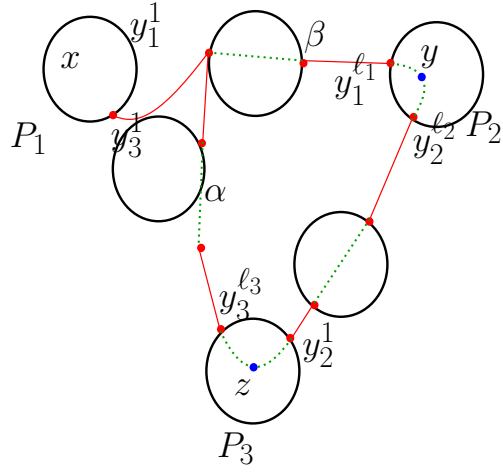}}
    \caption{This figure shows a schematic of a triangle composed of segments of type $\alpha$ (green, dashed) and $\beta$ (red, solid), and the corresponding straight chains for its sides, labelled by $y^i_j$ (red solid dots). It is useful to visualize \emph{axiom straight chain triangle}, as well as the proof of implication from Definition ~\ref{def:treegraded_ds} to Definition ~\ref{def:treegraded_drmg} (in Appendix~\ref{sec:treegraded}). }\label{fig:axiom2bis} 
}
\end{figure}

\begin{theorem}
\label{thm:decomposition}
The current dual $X_\mu$ is a (metric) tree-graded space with respect to $\mathcal{P}$.
\end{theorem}
\begin{proof}
\begin{enumerate}
\item \emph{axiom pieces}. We firstly recall that $\mathrm{supp}(m)$ is a pairwise disjoint union of geodesics in $\wt{X}$. The fact that two pieces $P_1$ and $P_2$ have at most one point in common will follow from the fact that the intersection of two (closed) complementary regions $R_1 , R_2$ is either empty, or consists of a single geodesic line. Let $P_1 = R_1 / {\{d_\mu = 0\}}$ and $P_2 = R_2 / {\{d_\mu = 0\}}$. If $P_1 \cap P_2$ is empty, we are done. Otherwise, assuming the the two pieces $P_1$ and $P_2$ have non-empty intersection implies that $R_1 \cap R_2 \neq \emptyset$. By construction, it means they are adjacent, i.e. there is a special geodesic $r \in R_1 \cap R_2$ bounding both, and hence $R_1 \cap R_2 = \{r \}$. It suffices to show that, in the quotient, the line $r$ is a single point, which amounts to show that for all $x, y \in r$ we have $d_\mu (x,y) =0$. Assume by contradiction $d_\mu (x, y) > 0$, then there exists a line $l \in \mathrm{supp} (\mu)$ that intersects transversely $r$, contradicting the fact that $r$ is special (i.e. doesn't intersect any other line in $\mathrm{supp}(\mu)$), as wanted.
\item \emph{axiom transversals}. Recall $\mathfrak{R}$ denotes the family of complementary regions determined by the lifts of the special multi-curve $m$. For every $x \in X_{\mu}$, suppose $x \in R_i$, for $R_i \in  \mathfrak{R}$, and let $\wt{c}_1,\cdots, \wt{c}_n$ denote the boundary geodesics of $R_i$, corresponding to lifts of special geodesics $c_1,\cdots,c_n$. Consider the connected component of $R_i \backslash \supp(\mu)$ containing $x$, and suppose it is bounded by geodesics $\wt{c}_{i_1},\cdots, \wt{c}_{i_k}$.  We let $T_x$ be the dual tree $T_{c_{i_j}}$, for any $j=1,\cdots,k$, which is trivially $0$-hyperbolic. Moreover, the transversal condition is trivially satisfied by construction.
\item \emph{axiom straight chain triangle}. A straight chain in this setting corresponds to a sequence $(y_1,\cdots,y_n)$, where $y_i$ are the points obtained by collapsing lifts of simple closed curves in $m$.
If a straight chain triangle is contained in more than one piece, we show that the interior point of one of the chains must intersect one of the points in another side chain. Indeed, let the straight chain triangle be composed of side chains $C_1,C_2,C_3$, and suppose that the endpoints of $C_1$, $x$ and $y$, are in distinct pieces $P_1, P_2$. Then, there is an interior point of $C_1$, say $y_i$, that is the equivalence class of a lift of a special geodesic $\wt{m_i}$, separating regions $R_1,R_2$ corresponding to the pieces $P_1,P_2$.
Since $\wt{m_i}$ separates $\wt{X}$ into two components, the straight chain $C_2 \cup C_3$, which also connects $x$ and $y$, must also contain $y_i$.

\end{enumerate}
\end{proof}

Let $P = R / {\{ d_\mu = 0 \}}$ be a piece corresponding to a region $R$ which is a lift of the subsurface $X_i \subseteq X$. The subsurface $X_i$ is a hyperbolic surface with geodesic boundary, and its universal cover is isometric to $R$. Moreover the pseudo-distance $d_{\mu_i}$ on $R = \wt{X_i}$ coincides with the restriction of $d_\mu$ on $R$. It follows that $P = (X_i)_{\mu_i}$, i.e. the piece $P$ is the dual space of the current $\mu_i \in \Curr(X)$, understood as a geodesic current in $\Curr(X_i)$ restricted to a region $R$. Precisely, there is an inclusion $\iota \colon X_i \xhookrightarrow{} X$ inducing a continuous injective pushforward of geodesic currents $\iota_* \colon \Curr(X_i) \to \Curr(X)$ (see~\cite[Section~4.2]{EM18:ErgodicGeodesicCurrents}), which sends $\mu_i$ as a current on $X_i$ to $\mu_i$ as a current on $X$. 
This allows us to restate the previous Theorem~\ref{thm:decomposition} as follows.

\begin{theorem}
The dual space $X_\mu$ is a metric tree-graded space where the underlying tree is the dual tree of the special multi-curve $m$ and the pieces are the dual spaces of the subcurrents $\mu_i$ of $\mu$ on the subsurfaces $X_i$, restricted to the regions determined by the lifts of the special multicurve. \label{thm:metric_tree_graded}
\end{theorem}

We conclude this subsection by showing that the pseudo-distance $d_\mu$ on $X_\mu$ can be computed from the pseudo-distances $d_{\mu_i}$ on the pieces $P_i$.

In this setting, we define a \emph{chain} between two points $x, y \in \wt{X}$ is a sequence of points $C = (\overline{x_0}, \overline{x_1} , \cdots ,\overline{x_{n+1}} )$ with $\overline{x_0} = \overline{x}$ and $\overline{x_{n+1}} = \overline{y}$ such that any two consecutive points $\overline{x_i}$ and $\overline{x_{i+1}}$ are in the same piece $P_i$, with $P_i \neq P_{i+1}$.

It follows from the definition that given $P_j$ and $P_{j+1}$, the point $\overline{x_{j+1}}$ is on the common geodesic boundary $c_j$. We call a chain \emph{straight} if it does not go `back and forth', i.e., if $c_j \neq c_{j+1}$ for $j = 1, \dots {n-1}$. If a chain $C$ is straight, then the corresponding ordered sequence of boundary geodesics $(c_1, \dots , c_n)$ are precisely \emph{the} geodesics separating $\overline{x}$ from $\overline{y}$. 
Each piece $P_i$ is naturally endowed with the induced pseudo-distance $d_{\mu_i}$. We define the \emph{length of a chain} to be
\[
l(C)= \inf_{C}\sum_{i=1}^{n+1} d_{\mu_i} (\overline{x_i} , \overline{x_{i+1}}),
\]
where the $\inf$ is taken over all chains $C$ joining $\overline{x}$ to $\overline{y}$.
It is enough to consider the $\inf$ over straight chains $c$. Moreover, since the curves $c_i$ are special geodesics, there is no line in the support of $\mu$ intersecting them, and hence the distance between any two points on the same special geodesic $c_i$ is zero. It follows that $d(\overline{x_j}, \overline{x_{j+1}})$ does not depend on the choice of $\overline{x_j}$ on $c_j$. This means that all straight chains from $\overline{x}$ to $\overline{y}$ have the same length, and hence we may as well define $d(\overline{x},\overline{y})$ as the length of \emph{any} straight chain.

\begin{lemma}
Let $C = (\overline{x_0}, \cdots, \overline{x_{n+1}})$ be a straight chain from $x$ to $y$ in $\wt{X}$. The pseudo-distance $d_\mu$ on $\wt{X}$ can be expressed as
\[
d_\mu (\overline{x},\overline{y}) = \sum_{i=1}^{n+1} d_{\mu_i} (\overline{x_i} , \overline{x_{i+1}}).
\]
\label{lem:gradeddistance}
\end{lemma}
\begin{proof}
Let $\gamma$ be the hyperbolic geodesic joining $\overline{x}$ to $\overline{y}$.

By the above arguments we can consider, without loss of generality, $C$ to be the chain such that $\overline{x_j} = \gamma \cap c_j$. Since the pseudo-distance $d_\mu$ is straight, it follows that
\[
d_\mu (\overline{x},\overline{y}) = \inf_{C}\sum_{i =0}^{n} d_{\mu_i} (\overline{x_i} , \overline{x_{i+1}}) = \sum_{i =0}^{n} d_{\mu_i} (\overline{x_i} , \overline{x_{i+1}}),
\]
where the last equality follows from the fact that $\overline{x_i}$ and $\overline{x_{i+1}}$ belong to the piece $P_i$.
\end{proof}

\subsection{Properties of the decomposition}

We recall the definition of graph of groups, as in \cite[10.2]{Kap00:Kapovich2000HyperbolicMA}. Let $Y$ be a finite graph where each edge is oriented. We assume that to each vertex $v$ of $Y$ is assigned a \emph{vertex} group $G_v$ and to each edge $e$ is assigned an \emph{edge} group $G_e$. Each inclusion $v \xhookrightarrow{} e$ of a vertex into an edge (as the initial of terminal vertex) corresponds to a monomorphism
$h_{ev} \colon G_e \to G_v$.
The collection
\[
(Y, \{ G_e, G_v, h_{ev} \colon \mbox{ where } e,v  \mbox{ are edges and vertices of } Y \} )
\]
is called a \emph{graph of groups} $(\mathcal{G}, Y)$, where $\mathcal{G}$ is the data of all vertex groups, edge groups and monomorphisms, and $Y$ is the underlying graph. We denote with $\mathcal{G}^0$ and $\mathcal{G}^1$ the set of vertices and edges of $Y$, respectively. When we don't need to specify the underrlying graph $Y$, we will refer to the graph of groups simply as $\mathcal{G}$.

The fundamental group $\pi_1(\mathcal{G})$ of a graph of groups $\mathcal{G}$ is defined as 
\[
\pi_1 (\mathcal{G}) = \langle G_v , t_e : v \in \mathcal{G}^0 , e \in \mathcal{G}^1 | t_e t_{\overline{e}} = 1 , {t_e}^{-1} h_e (g) t_e = h_{\overline{e}} (g) \text{ for all } g \in G_e , e \in \mathcal{G}^1 \rangle
\]

We have seen that any geodesic current $\mu$ decomposes as in \ref{thm:structurecurrents}, where each component $\mu_i$ is supported on a subsurface $X_i$. The decomposition of $X$ in the subsurfaces $X_i$ is given by the family of so-called \emph{special geodesics} $\{c_1 , \dots , c_n \}$.
Note that a special geodesic does not need to be separating.

Given $\mu \in \mathrm{Curr}(X)$ we define a graph of groups $(\mathcal{G}, Y)$ as follows.
For each subsurface $X_i$ we define a vertex $v_i$, and for each special curve $c$ we define an edge $e_{i,j}$ between the vertices $v_i$ and $v_j$ if the curve $c$ is the boundary of $X_i$ and $X_j$. Note that we may have $i = j$ when $c$ is not separating, and $e_{i,j}$ is in this case a loop based at $v_i = v_j$. 

For each edge $e \in \mathcal{G}^1$ we put $G_e \coloneqq \Z$, and for each $v \in \mathcal{G}^0$ we put $G_{v_i} \coloneqq \pi_1 (X_i)$. Let $e \in \mathcal{G}^1$ be an edge joining $v_i$ to $v_j$, and let $c$ be the special geodesic bounding $X_i$ and $X_j$. The monomorphisms $h_{v_i,e} : G_e \to G_{v_i}$ is given by $\iota_* : \Z \to \pi_1(X_i)$, i.e. the induced map from the inclusion $c \hookrightarrow X_i$ at the level of fundamental groups. 

The following result is obtained by simply invoking classical Bass-Serre theory (see~\cite{S03:Trees}).

\begin{proposition}
The fundamental group $\pi_1 (X)$ has a graph of groups decomposition. In particular, $\pi_1 (X)$ is isomorphic to $\pi_1 (\mathcal{G},Y)$.
Furthermore, the fundamental group $\pi_1 (\mathcal{G})$ acts on the simplicial tree $T = \mathcal{T} (m)$ dual tree of the special multi-curve $m = \sum_j a_j s_j$. The factor graph $T/\pi_1 (\mathcal{G})$ is isomorphic to $Y$, and for such action we have
\begin{enumerate}
    \item $\mathrm{stab}_{\pi_1 (\mathcal{G})} (v)  \cong G_v$ for all $v \in T ^0$;
    \item $\mathrm{stab}_{\pi_1 (\mathcal{G})} (e) \cong G_e$ for all $e \in T^1$.
\end{enumerate}
\label{prop:graphgroups}
\end{proposition}
\section{Actions}
\label{sec:actions}

The dual spaces $X_{\mu}$ are naturally equipped with a $\pi_1(X)$-action.
We study the properties of this action.

Recall that $d_{\mu}$ denotes both the pseudo-distance in $\wt{X}$ as well as the induced distance on $X_{\mu}$. We will distinguish them by writing points in $\wt{X}$ with the overline notation $\overline{x} \in \wt{X}$, and points in $X_{\mu}$ without it, $x \in X_{\mu}$.

\begin{definition}[Action]
Let $\pi_{\mu} \colon \wt{X} \to X_{\mu}$ be the natural quotient projection. Given $x \in X_{\mu}$, let $\overline{x} \in \pi_{\mu}^{-1}(x)$.
Let $g \in \pi_1(X)$. Define $g \cdot x = \pi_{\mu} ( g(\overline{x}))$.
\label{def:action}
\end{definition}

\begin{lemma}
The fundamental group $\pi_1(X)$ acts by isometries on $X_{\mu}$.
\end{lemma}
\begin{proof}
First, we show that the action in Definition~\ref{def:action} is well-defined. Indeed, suppose that $\overline{x},\overline{y} \in \pi_{\mu}^{-1}(x)$, i.e., $d_{\mu}(\overline{x},\overline{y})=0$,
i.e. $\frac{1}{2}\mu(G[\overline{x},\overline{y})) + \frac{1}{2}\mu(G(\overline{x},\overline{y}])=0$.
Since $\mu$ is $\pi_1(X)$-invariant,
we have
\begin{align*}0&=\frac{1}{2}\mu(g G[\overline{x},\overline{y})) + \frac{1}{2}\mu(g G(\overline{x},\overline{y}])=\\
& \frac{1}{2}\mu( G[g \cdot \overline{x},g \cdot\overline{y})) + \frac{1}{2}\mu( G(g \cdot \overline{x},g \cdot \overline{y}])=\\
& d_{\mu}(g \cdot \overline{x},g \cdot \overline{y}),
\end{align*}
which shows that 
$g \cdot \overline{y} \in \pi_{\mu}^{-1}(g \cdot y)$, and thus shows that the action is well-defined.
The same computation for two arbitrary points $x,y \in X_{\mu}$ with lifts $\overline{x},\overline{y}$ shows that the action is by isometries.
\end{proof}

\begin{definition}[Translation length]
For $g \in \pi_1(X)$, we define \[\ell_{\mu}(g) \coloneqq \inf_{x \in X_{\mu}} d_{\mu}(x,g\cdot x).\]
 \end{definition}
\begin{lemma}
For $g \in \pi_1(X)$,
\[
\ell_{X_{\mu}}(g)=i(\mu,g)
\]
and thus 
\[
\ell_{X_{\mu}}(g^n)=n \ell_{X_{\mu}}(g)
\]
\label{lem:lenvsint}
\end{lemma}
\begin{proof}
The first can be found in~\cite[Lemma~4.7]{BIPP21:Currents}.
 The second result follows from general properties about Bonahon's intersection number.
\end{proof}

\subsection{Coboundedness} 

\begin{definition}[Cobounded/cocompact]
Let $X$ be a metric space.  An action $(G,X)$ is said to be \emph{cobounded} if there exists a bounded set $B \subset X$ so that $X=G B$, i.e., $X=\cup_g g(B)$. An action $(G,X)$ is said to be \emph{cocompact} if there exists a compact set $K \subset X$ so that $X=G K$, i.e., $X=\cup_g g(K)$.
\end{definition}

Let $X,Y$ be metric spaces and $f \colon X \to Y$ be map (not necessarily continuous).
We say that $f$ is \emph{bornologous} if for every $R>0$, there is $S>0$ so that if 
$d_X(x,y) < R$ then $d_Y(f(x),f(y))<S$.
We say that $f$ is \emph{large scale Lipschitz} if there exist constants $c>0$ and $A$ so that
\[
d_Y(f(x),f(y)) \leq c \cdot d_X(x,y) + A.
\]
$f$ is a \emph{quasi-isometric embedding} if 
there exist constants $c>1$ and $A$ so that
\[
1/c\cdot d_X(x,y) - A \leq d_Y(f(x),f(y)) \leq c \cdot d_X(x,y) + A.
\]
$f$ is \emph{coarsely surjective} if there exists a constant $C$ so that for every $y \in Y$ there is $z \in g(X)$ so that $d_Y(z,y)<C$.
We say that $f$ is a \emph{quasi-isometry} if it is a coarsely surjective quasi-isometric embedding.

\begin{lemma}[{\cite[Lemma~1.10]{Joe03:Coarse}}]
Let $X$ be a length space and $Y$ a metric space.
$f$ is bornologous if and only if it is large scale Lipschitz.
\label{lem:borno}
\end{lemma}

\begin{proposition}
Let $\mu \in \Curr(X)$.
Then $\pi_{\mu} \colon \wt{X} \to X_{\mu}$ is large scale Lipschitz.
\label{prop:largelip}
\end{proposition}
\begin{proof}
Given $R>0$, we will find $S_R>0$ so that if $d_{\wt{X}}(x,y) < R$, then $d_{\mu}(\pi_{\mu}(x),\pi_{\mu}(y))<S_R$.
Let $K \subset \wt{X}$ be a compact fundamental domain of the action of $\pi_1(X)$ on $\wt{X}$, and let $D$ be the hyperbolic diameter of $K$.
If $d_{\wt{X}}(x,y) \leq D$, then there exist $g_1,g_2 \in \pi_1(X)$ (depending on $x,y$), so that $x \in g_1 K$ and $y \in g_2 g_1(K)$ (where $g_1, g_2$ could be the identity), and $g_2 g_1 (K)$ is adjacent to $g_1 K$. Indeed, otherwise $d_{\wt{X}}(x,y) > D$, contradicting the choice of $D$.
Thus, $x,y \in g_1(K) \cup g_2 g_1(K)$.
Let $G(K)$ be the set of geodesics intersecting $g_1(K) \cup g_2 g_1(K)$. We have, by subadditivity of $\mu$ and $\pi_1(X)$ invariance, that
\begin{align*}
d_{\mu} ( \pi_\mu (x) , \pi_\mu (y)) & = d_{\mu}(x,y) = \frac{1}{2} \{ \mu (G (x,y]) + \mu ( G [x,y)) \}  \leq \mu (G [x,y]) \\ & \leq \mu (G(g_1 K) \cup G (g_2 g_1 K)) \leq \mu (G(g_1 K)) + \mu (G(g_2 g_1 K)) \\ & = 2 \mu (K)
\end{align*}

is finite (and independent of $x,y$).
Furthermore, $\mu(G(K))$ is finite, since $K$ is compact.
Thus, we have proven that if $d_{\wt{X}}(x,y) \leq D$, then $d_{X_\mu} (\pi_\mu (x) , \pi_\mu (y) ) \leq 2 \mu (G(K))$.
If $D < d_{\wt{X}}(x,y) \leq 2D$, the same proof now using three consecutively adjacent fundamental regions will yield that $d_{X_\mu}(x,y) \leq 3\mu(G(K))$. An induction argument then proves the general case.
\end{proof}

\begin{lemma}
Let $\mu$ be any geodesic current.
The action of $\pi_1(X)$ on $X_{\mu}$ is cobounded.
\label{lem:cocompact}
\begin{proof}
By Proposition~\ref{prop:largelip}, there exist constants $c>0$ and $A \geq 0$ so that for all $\overline{x},\overline{y} \in \wt{X}$,
\[
d_{\mu}(\pi_{\mu}(\overline{x}) ,\pi_{\mu}(\overline{y}))\leq c \cdot d_{\wt{X}}(\overline{x}, \overline{y}) + A.
\]
Thus, for every $x \in X_{\mu}$, and $r>0$, we have
\[
B_{\wt{X}}(\overline{x}, r ) \subseteq B_{\mu}(x ,c \cdot r + A)
\]
where $\overline{x} \in \pi_{\mu}^{-1}(x)$.
Since $\pi_1(X)$ acts cocompactly on $\wt{X}$, there exists a compact subset $K \subseteq \wt{X}$ so that $\wt{X} = \cup_g g K$.
Take a hyperbolic ball $B_{\wt{X}}(\overline{x},  r )$ so that 
\[
 K \subseteq B_{\wt{X}}(\overline{x}, r ).
\]
Then, by the above,
\[
 K \subseteq B_{\mu}(x, c \cdot r + A).
\]

Thus, by equivariance of $\pi_{\mu}$, we have
\[
X_{\mu}= \bigcup_{g \in \pi_1(X)} g B_{\mu}(x,c \cdot r + A),
\]
so the action is cobounded, as we wanted to see.
\end{proof}
\end{lemma}

If $\mu$ has no atoms, we can moreover show the action is cocompact.

\begin{lemma}
If $\mu$ is a current with no atoms, then the action of $\pi_1(X)$ on $X_{\mu}$ is cocompact.
\label{lem:real_cocompact}
\end{lemma}
\begin{proof}
Since $\pi_1(X)$ acts cocompactly on $\wt{X}$, there exists a compact subset $K \subseteq \wt{X}$ so that $\wt{X} = \cup_g g K$.
Let $K'=\pi_{\mu}(K)$, which is compact by continuity of $\pi_{\mu}$.
Thus, by equivariance of $\pi_{\mu}$, we have
\[
X_{\mu}= \bigcup_{g \in \pi_1(X)} g K'
\]
as we wanted to see.
\end{proof}

\subsection{Boundedness and compactness of balls}

We show that a geodesic current $\mu$ is filling if and only if all balls in  $X_{\mu}$ (equivalently, balls in the $d_{\mu}$-pseudo metric) are bounded. 
Moreover, we show that if $\mu$ has no atoms and it is filling, then $X_{\mu}$ is a proper metric space.

The following proof is adapted from~\cite[Proposition~2.8]{Glo17:CriticalExponents}. We claim no originality, but supply details and point out a gap in that proof.
\begin{proposition}
Let $\overline{x} \in \wt{X}$ and $\mu \in \Curr (X)$.
The following are equivalent:
\begin{enumerate}
    \item The current $\mu$ is filling;
    \item for every $\overline{x} \in \wt{X}$, $B_{\mu} (\overline{x}, 0)$ is bounded;
    \item for every $\overline{x} \in \wt{X}$ and every $r \geq 0$, $B_{\mu} (\overline{x}, r)$ is bounded.
\end{enumerate}
\label{prop:bounded}
\end{proposition}

\begin{proof}
3. $\Rightarrow$ 2. is obvious.

2. $\Rightarrow$ 1. Assume $\mu$ is not filling, then there exists a geodesic line $g \in \mathcal{G}(\wt{X})$ which does not intersect any line in $\mathrm{supp}(\mu)$. It follows that $
d_{\mu} (\overline{x},\overline{y}) = 0$ for any $\overline{x},\overline{y}$ on the geodesic $\gamma$.
In particular $\gamma \subset B_{\mu}(\overline{x},0)$, and hence $B_{\mu}(\overline{x},0)$ is not bounded, contradiction.

1. $\Rightarrow$ 3. Now we assume that $\mu$ is filling.
By Lemma~\ref{lem:distancecompare}, which then implies Lemma~\ref{lem:balls}, there exists $R>0$ so that $B_\mu (\overline{x}, r) \subset B_{\wt{X}}(\overline{x},R)$. Since  $B_{\wt{X}}(\overline{x},R)$ is bounded, it follows that $B_\mu (x,r)$ is bounded, and the Proposition follows. 
\end{proof}

\begin{lemma}
Let $\mu$ be a filling geodesic current on $X$.
There exist a constant $C(\mu)>0$ and a constant $\epsilon(C,\mu)>0$, so that
if $d_{\wt{X}}(\overline{x},\overline{x}) > C+1$ then $d_{\mu}(x,y) > \epsilon$.
\label{lem:distancecompare}
\end{lemma}
\begin{proof}
The surface $X$ has been endowed with a fixed hyperbolic structure, where the boundary components (if any) are totally geodesic. Let $T^1 X_{\mathit{rec}}$ denote the compact subset of $T^1X$ of vectors for which the geodesic flow $\phi_t$ is defined for all times, both in the future and in the past.
In particular, $T^1 X_{\mathit{rec}}$ is invariant under the geodesic flow and under the geodesic flip. If $X$ is a closed surface, then $T^1 X_{\mathit{rec}} = T^1 X$, but in general it is a proper subset.

Let $\phi : T^1 X_{\mathit{rec}} \times \R \to T^1 X_{\mathit{rec}} $ denote the geodesic flow on $T^1 X_{\mathit{rec}} $, $\pi \colon T^1 X_{\mathit{rec}} \to X$ the canonical projection, and $p : \wt{X} \to X$ the universal covering. Define $r \colon T^1 X_{\mathit{rec}}  \to \mathbb{R}$ to be the \emph{first return time} of $\phi_t$ to the support of $\mu$
\[
r(v) \coloneqq \inf \{ t\in \R : \pi(\phi (v,t)) \capt p (\mathrm{supp} \mu) \neq \emptyset \}.
\]

In other words, $r(v)$ is the first time when the geodesic emanating from $v \in T^1 X_{\mathit{rec}}$ intersects the support of $\mu$ transversely. Notice that $r(v)$ is finite since $\mu$ is filling.

Since the function $r$ is upper semi-continuous~(see, for example,~\cite[Lemma~7.3]{MGT21:Smoothings}), it follows that it admits an upper bound $C >0$ on the compact set $T^1 X_{\mathit{rec}}$. By lifting $r$ to $\wt{X}$ we have the upper bound $r(v) \leq C$ for all $v \in T^1 \wt{X_{\mathit{rec}}}$.

Now let us fix $\overline{x}, \overline{y}\in \wt{X}$ such that $d_{\wt{X}} (\overline{x},\overline{y}) \geq C+1$. Then the segment $[\overline{x},\overline{y}]$ must intersect transversely some geodesic line in the support of $\mu$, and therefore
$\mu (G[\overline{x},\overline{y}]) > 0$. If for every $n$, we could find $\overline{x_n}, \overline{y_n} \in \wt{X}$ so that $d_{\wt{X}} (\overline{x_n},\overline{y_n}) \geq C+1$ and $\mu (G[\overline{x_n},\overline{y_n}]) < \frac{1}{n}$, then compactness of $X$ would yield points $\overline{x_{\infty}},\overline{y_{\infty}} \in \wt{X}$ with $d_{\wt{X}} (\overline{x_{\infty}},\overline{y_{\infty}}) \geq C+1$ but $\mu (G[\overline{x_{\infty}},\overline{y_{\infty}}])=0$, which would contradict the choice of $C$. Thus, there must exist a uniform lower bound $\epsilon>0$ so that
for every $\overline{x}, \overline{y}\in \wt{X}$ such that $d_{\wt{X}} (\overline{x},\overline{y}) > C+1$, we have $\mu (G[\overline{x},\overline{y}]) > \epsilon>0$.
This shows that
\[
B_{\mu}(\overline{x},\epsilon) \subset B_{\wt{X}}(\overline{x},C+1),
\]
and thus $B_{\mu}(\overline{x},\epsilon)$ is bounded.
\end{proof}

\begin{remark}
Note that in~\cite[Proposition~2.8]{Glo17:CriticalExponents} it is claimed that $B_{\mu}(\overline{x},\epsilon)$ is compact. We contest this: in fact, it is not true in general. Observe that since, by Lemma~\ref{lem:distancecompare}, $B_{\mu}(\overline{x},\epsilon)$ is bounded, its compactness is equivalent to $B_{\mu}(\overline{x},\epsilon)$ being closed in $\wt{X}$. At the same time, 
\[
B_{\mu}(\overline{x},\epsilon) = \{ \overline{y} \in \wt{X} \colon d_{\mu}(\overline{y},\overline{x}) \leq \epsilon \}
\]
which is closed if and only if $d_{\mu}(\overline{y},\cdot)$ is lower semicontinuous (in the topology of $\wt{X}$ induced by the hyperbolic metric). However, we know from Propositions~\ref{prop:contdistance} and~\ref{prop:discontdistance} this happens if and only if $\mu$ has no atoms.
We collect this in the next proposition.
\end{remark}

A metric space $(X, d)$ is \emph{proper} if closed balls are compact.

\begin{proposition}
Given a geodesic current $\mu$ on $X$.
If $\mu$ is filling and has no atoms, then $X_{\mu}$ is proper.
\label{prop:properspace}
\end{proposition}
\begin{proof}
Assume $\mu$ is filling and has no atoms.
Let $x \in X_{\mu}$ and $B=B_{\mu}(x,r)$ be a closed $d_{\mu}$-ball. The preimage $\pi_{\mu}^{-1}(B)$ is the closed $d_{\mu}$-ball in $\wt{X}$ (in the pseudo-metric $d_{\mu}$), which by Lemma~\ref{lem:balls} below is bounded.
If $\mu$ has no atoms, then $\pi_{\mu}$ is continuous, thus $\pi_{\mu}^{-1}(B)$ is closed. Altogether, this shows $\pi_{\mu}^{-1}(B)$ is compact, and continuity of $\pi_{\mu}$ again implies $B$ is compact, so $X_{\mu}$ is proper.
\end{proof}

\begin{lemma}
For all $r$, there exists a constant $R(r)>0$, 
so that for all $\overline{x} \in \wt{X}$,
\[
B_{\mu}(\overline{x},r) \subset B_{\wt{X}}(\overline{x},R).
\]
\label{lem:balls}
\end{lemma}
\begin{proof}
Let $C$ and $\epsilon$ be the constants given by Lemma~\ref{lem:distancecompare}. Let $\overline{x},\overline{y} \in \wt{X}$ with $d_{\wt{X}} (\overline{x},\overline{y}) \geq C +1$ and and $d_\mu (\overline{x},\overline{y}) \geq \epsilon$. In other words, we have
\[
B_\mu (\overline{x}, \epsilon) = \{ \overline{y} \in \wt{X} \colon d_\mu (\overline{x}, \overline{y}) \leq \epsilon \} \subseteq B_{\wt{X}} (\overline{x}, C+1)
\]
Now we proceed by induction on the distance $d_{\wt{X}} (\overline{x},\overline{y}) = N \cdot (C +1)$, for any $N \in \N$. The induction basis has been proven already. Assume that the following implication holds
\[
d_{\wt{X}} (\overline{x},\overline{y}) \geq (N-1) \cdot (C +1) \Rightarrow d_\mu (\overline{x},\overline{y}) \geq (N-1) \cdot \epsilon
\]
assume that $d_{\wt{X}} (\overline{x},\overline{y}) \geq N \cdot (C +1)$. We want to show that $d_\mu (\overline{x},\overline{y}) \geq N \cdot \epsilon$.

Let $\gamma \in \mathcal{G}({{\wt{X}}})$ be the geodesic line joining $\overline{x}$ to $\overline{y}$ in ${\wt{X}}$. Since $d_{\wt{X}} (\overline{x},\overline{y}) \geq (N-1) (C+1)$, then there must exist $\overline{z} \in \gamma$ such that $d_{\wt{X}} (\overline{x},\overline{z}) \geq (N-1) (C+1)$ and $d_{\wt{X}} (\overline{z}, \overline{y}) \geq C+1$, and hence, by induction hypothesis $d_\mu (\overline{x},\overline{z}) \geq (N-1)\epsilon$ and $d_\mu(\overline{z},\overline{y}) \geq \epsilon$.
Since the pseudo distance $d_\mu$ is straight, we have
\[
    d_\mu (\overline{x}, \overline{y}) = d_\mu (\overline{x}, \overline{z}) + d_\mu (\overline{z},\overline{y}) \geq (N-1)\epsilon +\epsilon = N \epsilon
\]
as wanted.
Finally, if $r < \epsilon$, then $B_{\mu}(x,r) \subset B_{\mu}(x,\epsilon) \subset B(x,R(\epsilon))$, by the previous argument.
\end{proof}

\begin{proposition}
Let $\mu \in \Curr(X)$.
If $\sys(\mu)>0$, then $\pi_{\mu} \colon \wt{X} \to X_{\mu}$ is a quasi-isometry.
\label{prop:quasi-isom}
\end{proposition}
\begin{proof}
First, note that for every $x \in X_{\mu}$, $\pi_{\mu}^{-1}(x)$ is a bounded subset of $\wt{X}$.  Indeed, its hyperbolic diameter is upper bounded by the maximum of the return time, as in the proof of Lemma~\ref{lem:distancecompare}. Making an arbitrary choice of $z_x \in \pi_{\mu}^{-1}(x)$ for each $x \in X_{\mu}$, we define a map $g \colon X_{\mu} \to \wt{X}$, given by $x \mapsto z_x$. This map is bornologous by Lemma~\ref{lem:balls}. It follows then by Lemma~\ref{lem:borno}, that there exist constants $c>0$ and $A$ so that
\[
d_{\wt{X}}(\overline{x},\overline{y})  \leq c \cdot d_{\mu}(\pi_{\mu}(\overline{x}),\pi_{\mu}(\overline{y}))   + A.
\]
for every $\overline{x}, \overline{y} \in g(X_{\mu})$. This, together with Lemma~\ref{prop:largelip} shows that $\pi_{\mu}$ is a quasi-isometric embedding.
We note that $g$ is coarsely surjective: the upper bound $C$ of the return time satisfies the property that, for every $\overline{x} \in \wt{X}$, there is $\overline{y} \in g(X_{\mu})$ so that $d_{\wt{X}}(\overline{x},\overline{y})<C$.
From this it follows $\pi_{\mu}$ is a quasi-isometry.
\end{proof}

The following definition can be found in~\cite[Definition~I.8.2]{BH11:NonPosCurvature}.
\begin{definition}[Proper action]
Let $G$ a group acting by isometries on a metric space $X$. The action is said to be proper if for each $x \in X$, there exists $r>0$, so that the set $\{ g\in G \colon g B(x,r) \cap B(x,r) \neq \emptyset \}$ is finite, where $B(x,r)$ denotes an open ball of radius $r$ centered at $x$.
\label{def:propaction}
\end{definition}

The action $\pi_1(X)$ on $X_{\mu}$ is not proper in general. In this section we characterize when it is.

Let $d^*$ denote the distance on $T^1 \wt{X}$ and $d$ the distance on $X$.
We recall the following classical result in the dynamical properties of the geodesic flow of a hyperbolic surface, known as the Closing lemma.

\begin{lemma}[{\cite[4.5.15]{Eberlein1979SurfacesON}}]
Given a compact set $C \subset T^1 \wt{X}$ and $\xi >0$, there exists $T \geq 0$, and $\delta > 0$ such that if there is $t \geq T$, $v \in C$, and $g \in \pi_1(X)$ and $d^{*}(g(v),g_t(v))<\delta$, then there is $t' \in \mathbb{R}$ with $|t' - t| < \xi$ and $v' \in T^1\wt{X}$ with $d^*(v',v) < \xi$ and $g(v')=g_{t'}(v')$. 

\label{lem:closing}
\end{lemma}

Using this result, we can prove the following.

\begin{proposition}
If $\mu$ is a geodesic current with a subcurrent of type 2 in its decomposition, then the action of $\pi_1(X)$ on $X_{\mu}$ is not proper.
\label{prop:fillingmeasnotprop}
\end{proposition}
\begin{proof}
This follows from elaborating on the proof of~\cite[Proposition~5.1]{BIPP21:Currents}.
Let $\mu_i$ be a type 2 current in the decomposition of $\mu$, and $Y$ be the subsurface $X_i$ on which $\mu_i$ is supported in the decomposition of $\mu$.
Recall that a geodesic $c$ is recurrent in $\mathring{Y}$ if there exists a sequence
$(t_n)$ in $\mathbb{R}$ with $\lim_n |t_n| = \infty$ and the sequence $( c(t_n) )$ staying in a fixed compact subset $K$ of $\mathring{Y}$. Up to reparameterization
we may assume that the sequence $(t_n)$ is monotone increasing with
$\lim t_n = +\infty$. Let $v \in T^1 \mathring{Y}_p$ be an accumulation point of the sequence
$(c'(t_n))$. Up to enlarging $K$ we may assume $p \in \mathring{K}$. Let $s \colon \mathbb{R} \to X$ be the unit speed geodesic with $s'(0) = v$, and $\gamma \subset \wt{X}$ be a lift of $s$.
Let $\overline{B}(\wt{p},\eta)$ be the closed ball centered at $\wt{p}$ of radius $\eta$ for the hyperbolic metric in $\wt{X}$.
In ~\cite[Proposition~5.1]{BIPP21:Currents} it is shown that for every $\epsilon>0$, and for all $x \in \gamma$ except at most countably many, there exists $\eta>0$ so that $\mu(G(\overline{B}(\wt{p},\eta) \setminus \{ \gamma \})< \epsilon$.
Decreasing $\eta$ if necessary, we can in addition assume $\eta$ is smaller than $d_{X}(p, \partial X)$ and the injectivity radius at $p$, so that the projection from $\wt{X}$ to $X$ sends $\overline{B}(\wt{p},\eta)$ to $\overline{B}(p,\eta) \subset \mathring{X}$
isometrically. Since $c$ is not closed, $c'(t_n)$ is never tangent to $s$. We can assume $c(t_n)$ is not in $s$ and, by passing to a subsequence, the points $c(t_n)$ are on the same side of $s$ in $\overline{B}(p, \eta)$.
Now let $0 < \xi < \eta$ and $C$ be the compact set consisting of unit tangent vectors based at a point of $\bar{B}(p,\eta)$. Let $T$ and $\delta$ be the corresponding
constants given by Lemma~\ref{lem:closing}. We may assume $\delta < \eta$, and choose $n_0 \in \mathbb{N}$ so that $d^*(c_n'(t), v)< \delta$ for all $n \geq n_0$. Then, pick $n>m>n_0$ so that $t_n - t_m > T_k$ and thus, by the hypotheses, $d^*(c_n'(t), c_m'(t))< \delta$.
Now, we construct a sequence of curves by taking a sequence of shrinking $\epsilon$'s in the above construction, yielding corresponding $\eta_k$, as follows. Letting $\epsilon=1/k$, we get a number $\eta_k>0$ satisfying the properties above, and we can construct a loop $\alpha_k$ by concatenating a geodesic segment $c([0, t_n - t_m])$ and the geodesic segment connecting $c(t_n)$ and $c(t_m)$ within $\overline{B}(p,\eta_k)$.
As in the proof of~\cite[Proposition~5.1]{BIPP21:Currents}, the closed lemma gives us a closed geodesic $c_k$ in the homotopy class of $\alpha_k$ represented by an element $g_k \in \pi_1(X)$ with the key property that $i(\mu, [g_k])< \frac{1}{k}$. We now analyze the hyperbolic geometry of any of its lifts.
A lift of $\alpha_k$ in $\wt{X}$ is a piecewise geodesic path consisting of long geodesic segments of hyperbolic length larger than $T_k$, and short geodesic segments of length $1$. 
Each endpoint $\wt{\alpha_k}^{\pm}$ of such lift is contained in intervals $I_k \subset \partial \wt{X}$ of diameter shrinking to $0$ as $k$ goes to infinity. Thus, $\alpha_k^+$ (resp. $\alpha_k^-$) converges to some $\alpha^+$ (resp. $\alpha_-$) in $\partial \wt{X}$.
Observe that the endpoints $\alpha_k^{\pm}$ are the same as the endpoints of the lift of the closed geodesic $c_k$ stabilized by $g_k$, since the periodic orbits $c_k$ is in the homotopy class of $\alpha_k$.
Thus, if we also denote $\wt{c_k}$ the geodesic lifts of the closed geodesics $c_k$, we have that $\wt{c_k}$ converges to some geodesic $\wt{c}$ in the Gromov-Hausdorff topology of $\mathcal{G}(\wt{X})$.
Observe, too, that $\ell_X(g_k) \asymp T_k$ (where $\asymp$ means asymptotically equal up to multiplicative and additive error independent of $k$), since the hyperbolic length of $\alpha_k$ is $T_k + 1$ and $\wt{\alpha_k}$ and $\wt{c_k}$ are quasi-isometric. Therefore (up to taking a subsequence), the $g_k$ are hyperbolic isometries of different translation lengths, and thus distinct elements of $\pi_1(X)$.
We also note that, by construction, $i(\mu, g_k) < \frac{1}{k}$.
By picking $\overline{x} \in \wt{c}$,  letting $x=\pi_{\mu}(\overline{x})$, and taking $B_{\mu}(x, r)$, for an arbitrary $r>0$, by the proof of~\cite[Proposition~3.14]{MZ19:PositivelyRatioed}, we have
\[
d_{\mu}(\bar{x}, g_k \bar{x})  < i(\mu, g_k) + 2d_{\mu}(\bar{x},A_{c_k}) < \frac{1}{k} + 2r < 3r
\]
for all $k$ larger than $1/r$.
Thus, we have that, for every $r$,
\[
B_{\mu}(\bar{x},3r) \cap g_k B_{\mu}(\bar{x},3r) \neq \emptyset
\]
for infinitely many distinct $g_k \in \pi_1(X)$. This shows the action of $\pi_1(X)$ is not proper on $X_{\mu}$.
\end{proof}

\begin{proposition}
If $\sys(\mu)>0$, then $\pi_1(X)$ acts properly on $X_{\mu}$.
\label{prop:proper}
\end{proposition}
\begin{proof}
Suppose not, i.e., suppose there exists $x$ so that for all $r>0$, the set $\{ g\in \pi_1 (X) \colon g B_{\mu}(x,r) \cap B_{\mu}(x,r) \neq \emptyset \}$ is infinite. Consider a sequence $(g_n)$ of elements in $\pi_1(X) \cap \{ g\in \pi_1 (X) \colon g B_{\mu}(x,r) \cap B_{\mu}(x,r) \neq \emptyset \}$.
Choose $r>\epsilon$, where $\epsilon$ is the constant in Lemma~\ref{lem:balls}, so that $B_{\mu}(\overline{x},r) \cap g_i(B_{\mu}(\overline{x},r)) \neq \emptyset$. By Lemma~\ref{lem:balls}, we have,
\[
B_{\mu}(\overline{x},r) \cap g_i(B_{\mu}(\overline{x},r)) \subset B_{\wt{X}}(\overline{x},R) \cap g_i(B_{\wt{X}}(\overline{x},R)).
\]
This contradicts that the action of $\pi_1(X)$ on $\wt{X}$ is proper.
\end{proof}

The following theorem follows immediately from the previous propositions.

\begin{theorem}
\label{thm:proper_for_filling}
The action of $\pi_1(X)$ on $X_{\mu}$ is proper if and only if $\mu$ is filling.
\end{theorem}
\begin{proof}
If $\sys(\mu)>0$, then Proposition~\ref{prop:proper} implies the action of $\pi_1(X)$ on $X_{\mu}$ is proper.
If $\sys(\mu)=0$ then either $\mu$ has a subcurrent of type 2 in its decomposition, in which case Proposition~\ref{prop:fillingmeasnotprop} implies the action is not proper, or $\mu$ fills a proper subsurface $Y$.
Then there is some geodesic $\gamma$ in $\wt{X}$ that is a lift of a closed curve $c$, so that $\overline{x} \coloneqq \pi_{\mu}(\gamma)$ is a single point in $X_{\mu}$. But then $\gamma \overline{x} = \overline{x}$, and thus $\gamma^n \overline{x} = \overline{x}$ for every $n \in \mathbb{Z}$. Since $\pi_1(X)$ is torsion-free, this implies that the set $\{g \in \pi_1(X) : g B_{\mu}(\overline{x},r) \cap B(\overline{x},r) = \emptyset \}$ is infinite for any $r>0$, and hence the action of $\pi_1(X)$ is not proper.
\end{proof}

\subsection{Freeness}

The action of $\pi_1(X)$ on $X_{\mu}$ is not always free. In this section we characterize when it is.

The following lemma relates stabilizers of points of $X_{\mu}$ and the topology of the support of a measured lamination $\mu$, and it's just an observation.

\begin{lemma}
Let $\mu$ be a non-trivial geodesic current.
Let $\mathcal{C}$ be the set of connected components of $\wt{X} \setminus \supp \mu$.
\begin{enumerate}
\item If $x \in C \in \mathcal{C}$, then the stabilizer of $\pi_\mu (x)$ is equal to the (set-wise) stabilizer of $C$.
\item If $x \in \gamma \in \supp \mu$, so that $\gamma$ doesn't intersect any geodesic in $\supp \mu$, then the stabilizer of $\pi_\mu (x)$ in $X_{\mu}$ is the stabilizer of $\gamma$ in $\wt{X}$.
\item If $\overline{x} \in \gamma \in \supp \mu$, so that $\gamma$ intersects some geodesic in $\supp \mu$, then the stabilizer of $\pi_{\mu} (\overline{x})$ in $X_{\mu}$ is trivial.
\end{enumerate}
\label{lem:dualstab}
\end{lemma}
\begin{proof}
\begin{enumerate}
\item Clear.
\item Clear.
\item Let $\gamma' \in \supp(\mu)$ so that $\gamma \cap \gamma' \neq \emptyset$.
 Then $\pi_{\mu}^{-1}(x)$ is a proper geodesic subsegment $\eta$ of $\gamma$, and the (set-wise) stabilizer of $\eta$ is trivial. 
\end{enumerate}
\end{proof}

\begin{lemma}
 Let $\mu$ be a non-trivial geodesic current.
 Then the action of $\pi_1(X)$ on $X_{\mu}$ is free if and only if $\mu$ has a unique component of type 1 (i.e., $\mu$ is filling) or a unique component of type 2.
\label{lem:free}
\end{lemma}
\begin{proof}
 Suppose that $\mu$ has more than one component in its decomposition or just one component but of type 3. Then, there is at least one simple closed curve $c$ in $X$ not intersected by any other geodesic of the support of $\mu$. By Lemma~\ref{lem:dualstab}, if $\gamma$ is a lift of $c$, the stabilizer of $\pi_{\mu}(\gamma) \in X_{\mu}$ is infinite cyclic, so the action of $\pi_1(X)$ on $X_{\mu}$ is not free.
Now, suppose $\mu$ has only one component of type 1 or type 2 in its decomposition. We will show that the action is free.
Assume first is of type 1 (i.e. $\mu$ is filling, i.e. $\sys(\mu)>0$). Then the action is of $\pi_1(X)$ is proper, hence free, by the same argument as in the proof of Theorem~\ref{thm:proper_for_filling}.
If the component is of type 2, then $\ell_{\mu}(c)=i(\mu, c)>0$ for every (non-boundary parallel) \emph{closed} curve $c$ on $X$, and hence $\ell_{\mu}(g)>0$ for every $g \in \Gamma$. In this case $X_{\mu}$ is equivariantly isometric to an $\mathbb{R}$-tree by Corollary~\ref{cor:0hyp}. Now, we recall that by~\cite[1.3]{CM87:Groups}, if $\Gamma$ is a finitely generated group acting isometrically on an $\mathbb{R}$-tree $T$, then $\ell_{T}(g)=0$ if and only if $g$ has a fixed point. Hence, applying this to $T=X_{\mu}$, no element $g \in \pi_1(X)$ acts with fixed points on $X_{\mu}$, and thus $\pi_1(X)$ acts freely.
\end{proof}

\begin{remark}
We give here a brief comment on isometry types and axes, which will not be used in the sequel.
Recall that for Gromov hyperbolic spaces, there is an analogous classification of isometries in three types: hyperbolic, parabolic and elliptic.
The classification can be described also in terms of the number of fixed points at infinity (see~\cite[Chapter~10]{CDR90:NotesGroupes}).
Observe that from the previous results, it follows that $g \in \pi_1(X)$ acts on $X_\mu$ as an elliptic or hyperbolic isometry.
Let $g \in \pi_1(X)$. If $i(\mu,[g])>0$ then $\ell_{X_{\mu}}(g)>0$, and thus $g$ is hyperbolic (by~\cite[2.2]{F15:Fujiwara}). Suppose $i(\mu,[g])=0$. If $[g]$ is a special geodesic in the decomposition of $\mu$, then it follows it has a fixed point, as in the proof of Lemma~\ref{lem:free}.
Otherwise, it's in one of the subsurfaces $X_i$ of the decomposition of $X$.
By the same argument as in Lemma~\ref{lem:free}, it must be in a subsurface that does not contain any projections of leaves in the support of $\mu$, and so $g$ has a fixed point.
Similarly, one can define a notion of axis of a hyperbolic element $g$ as
\[
T_g = \{ x \in X_{\mu} \colon d_{\mu}(x,g(x))=\ell_{X_{\mu}}(g)\}.
\]
By~\cite[Proposition~4.4]{MZ19:PositivelyRatioed}, it follows that the hyperbolic axis $A_g$ of $g$ is contained in $\pi_{\mu}^{-1}(T_g)$. In general, $T_g$ is geodesic (whenever $X_{\mu}$ is equipped with a geodesic structure) and corresponds to the sequence of complementary regions $R_i$ in $\wt{X}$ traversed by $A_g$, as in the setup of the decomposition theorem of Section~\ref{sec:decomposition}.
Finally, if two hyperbolic elements $g, h$ in $X_{\mu}$ have intersecting hyperbolic axes $A_g,A_h$ in $X$, then their axes $T_g$ and $T_h$ in the dual also intersect.
\end{remark}

\begin{theorem}
If $\sys(\mu)>0$, then $\pi_1(X)$ acts properly, coboundedly and freely on $X_{\mu}$.
\label{thm:geoaction}
\end{theorem}
\begin{proof}
It follows by Lemma~\ref{lem:cocompact}, Lemma~\ref{prop:proper} and Lemma~\ref{lem:free}.
\end{proof}

\section{Completeness}
\label{sec:completeness}

Recall that a metric space $(X,d)$ is \emph{complete} if every Cauchy sequence converges.
In this section we characterize when a dual $X_{\mu}$ is complete in terms of the structural decomposition of $\mu$ (see Theorem~\ref{thm:structurecurrents}). The main theorem of this section, Theorem~\ref{thm:new_completeness}, shows that, in the absence of atoms for $\mu$, $X_{\mu}$ is complete if and only it has no type 2 components in its structural decomposition, i.e., components which are non-discrete measured laminations.
In the presence of atoms, $X_{\mu}$ is complete if and only it is a multicurve. When $\mu$ is not a multi-curve, we show one can always find non-convergent Cauchy sequences, due to, essentially, the discontinuities of the projection map $\pi_{\mu}$.
We split the proof into two subsections. In the first one, we analyze the case of multi-curves and currents with type 2 components, and in the second, we prove the general result.

\subsection{Multi-curves and currents with type 2}
\label{subsec:measuredlam}

In Definition~\ref{def:currentdual} we introduced the dual of geodesic currents. The first observation is that, by the definition of $X_{\mu}$, if $\mu$ is a multi-curve, any the distance between two distinct points in $X_{\mu}$ is uniformly bounded from below. 
Therefore, it follows vacuously that

\begin{lemma}
If $\mu$ is a multi-curve, then the dual $X_{\mu}$ is complete.
\end{lemma}

We will now show that if $\mu$ is a non-discrete measured lamination, i.e. a type 2 current, then $X_{\mu}$ is not complete. 

We will make use of the notion of \emph{exotic ray} of a lamination as defined by T. Torkaman and Y. Zhang in \cite{TuZh21:halo}. A \emph{geodesic ray} $r$ on a complete hyperbolic surface of finite area $X = \wt{X} / \Gamma$ is a geodesic isometric immersion $r : [0, \infty ) \to X$.

Given a measured lamination $\Lambda$ on $X$ which is not a multi-curve, the intersection $i (r, \Lambda)$ between the ray and the lamination is generically infinite. There are two obvious cases when a ray has finite intersection with $\Lambda$:
\begin{enumerate}
    \item The ray $r$ is asymptotic to a leaf of $\Lambda$;
    \item The ray $r$ is eventually disjoint from $\Lambda$.
\end{enumerate}

Nevertheless, these are not the only two possibilities. Namely, there exist geodesic rays such that $i (r, \Lambda) < \infty$ but are neither asympotic to a leaf of $\Lambda$, nor eventually disjoint. Such rays are called \emph{exotic rays}.

\begin{theorem}[{\cite[Theorem~1.1]{TuZh21:halo}}]
Let $\Lambda$ be a non-multi-curve lamination. Then there exist exotic rays for $\Lambda$.
\end{theorem}

We will make use of the above result to prove the non-completeness for geodesic currents $\mu$ with type 2 components in their decomposition.

\begin{lemma}
Let $\nu$ be a geodesic current with a type 2 component $\mu$ in its decomposition. Then the current dual space $X_{\nu}$ is incomplete.
\begin{proof}
Let $Y$ be the subsurface induced by the decomposition of $\nu$ where $\mu$ is supported on.
Let $r \colon [0, \infty) \to X$ be an exotic ray for $\mu$ in $Y$, and let $T = \{ t \colon r(t) \in \supp(\mu) \}$, i.e., the times of intersection of $r$ with the support of $\mu$. Since $r$ is exotic, $T$ is infinite. We extract a countable collection of these times $(t_n)$, increasing and going to infinity, and consider the sequence $\bar{x_n} \coloneqq r(t_n) \in \wt{X}$. Also, let $r_n$ denote the subray of $r$ from $t_n$ onwards. Since $r$ is contained in $Y$, considering intersection with $\nu$ is the same as considering intersection with $\mu$. Note that $i(\mu, r_n) \to 0$ as $n$ goes to infinity, since $i(\mu, r)$ is finite.
For every $\epsilon>0$, take $N>0$ large enough so that $i(\mu,r_N) < \epsilon$. 
Then, for every $n> m \geq N$, we have $i(\mu,r_m)-i(\mu,r_n) = \mu(G[\bar{x_n},\bar{x_m}]) < \epsilon$, so the sequence $(\bar{x_n})$ is Cauchy. 
We are left to show that $(\overline{x_n})$ does not converge. Assume by contradiction $x_n \to x \in X_\mu$. Then pick any $\overline{x} \in \pi_{\mu}^{-1} (x) \subseteq \wt{X}$ and a small ball $B_\epsilon (\overline{x})$ around it. It follows that there exists $N > 0$, so that $\overline{x_n} \in B_\epsilon (\overline{x}) $ for all $n > N$, which is absurd as $(\overline{x_n})$ exits all compact sets in $\wt{X}$.
\end{proof}
\label{prop:incomplete}
\end{lemma}

\subsection{General Case}
\label{subsec:general}

The following result is elementary, see~\cite[Chapter I.3, Corollary~3.8]{BH11:NonPosCurvature}.
\begin{lemma}
If $(X,d)$ is a proper metric space, then it is complete and locally compact.
\label{thm:hopfrinow}
\end{lemma}
We can now prove the main result of this section:
\begin{theorem} Let $\mu$ be a geodesic current.
\label{thm:new_completeness}
\begin{enumerate}
    \item  If $\mu$ has no atoms, then
    $X_\mu$ is complete if and only if $\mu$ has no components of type 2 in its structural decomposition
    \item If $\mu$ has atoms, then $X_{\mu}$ is complete if and only if $\mu$ is a multi-curve.
\end{enumerate}
\end{theorem}
\begin{proof}
\begin{enumerate}
\item Suppose first $\mu$ has no atoms. If $\sys(\mu)>0$, we have proven that $X_{\mu}$ is proper in Proposition~\ref{prop:properspace}, and therefore its complete by Lemma~\ref{thm:hopfrinow}. If $\sys(\mu)=0$ and $\mu$ has a component of type 2 in its decomposition, then, by Lemma~\ref{prop:incomplete}, $X_{\mu}$ is not complete. If it has no components of type 2, then given a Cauchy sequence, we can extract a subsequence that is contained in a piece of the tree graded space structure (by Theorem~\ref{thm:metric_tree_graded}), using that the distance between non-adjacent pieces is uniformly bounded from below, that, in this case, transverse trees are singletons, and pieces intersect at a point. On the other hand, the subdual space corresponding to the piece is proper, and thus the subsequence is convergent. Hence, the original sequence is convergent.
    \item Suppose now $\mu$ has atoms. If it is a multi-curve, then it is obviously complete.
    If it is not, then it has non-atomic components or it has an infinite set of atoms (non-discrete) in $\G(\wt{X})$ (both cases can happen simultaneously, too).
    In the first case, we have a compact set of geodesics $B$ containing an infinite uncountable collection of lifts of non-closed geodesics $\mathcal{A}$ and an atom $\{ \gamma \}$. 
    Consider  $G=\mathcal{A} \cup\{ \gamma \}$ as a subset of $\wt{X}$, and consider an infinite sequence of points $(\overline{x_n})$ in $B-G \subset \wt{X}$, and $\overline{x} \in \gamma$, so that $\mu(G(\overline{x_n},\overline{x})) \to 0$. This is possible since $\mu(B)$ is finite, and we can pick a sequence of points approaching $\gamma$ along a geodesic ray that intersects $B$ and $\gamma$. Since we have $d_{\mu}(x_n, x)>\frac{1}{2}\mu(\{ \gamma \})$, the sequence is Cauchy but has no limit.
    If $\mu$ it is purely atomic but non-discrete, a similar argument as above can be performed, where now $\mathcal{A}$ is now replaced by a (countably) infinite set of atoms inside of a compact $B$.
\end{enumerate}
\end{proof}

\section{Topology}
\label{sec:topology}

In this section we prove that the map sending a geodesic current to its dual space is a continuous injection  when the space of geodesic currents is equipped with the natural weak$^*$-topology and the space of duals is equipped with the also natural equivariant Gromov-Hausdorff topology introduced by~\cite{Pau88:Thesis}.

\subsection{Topologies in the space of duals}

We start by describing a topology on the space of duals.
Let $\mathcal{Z}$ denote the space of Gromov hyperbolic spaces $Z$ with a cobounded action of $\pi_1(X)$ by isometries.
A topology we can equip it with is the \emph{equivariant Gromov-Hausdorff topology}, defined by the following family of neighbhorhoods.

\begin{definition}[$\epsilon$-relation]
\label{def:equivGH}
Given $Z \in \mathcal{Z}$, $K \subset Z$ a compact subset and  $P \subset \pi_1(X)$ a finite subset, we say that $(Z,K)$ is $\epsilon$-related to $(Z',K')$ if
\begin{enumerate} \item There exists a compact $K' \subset Z'$, and a relation $\mathcal{R}$ between $K$ and $K'$ so that for all $x,y \in K$ and $x',y' \in K'$, if $x \mathcal{R} x'$ and $y \mathcal{R} y'$, then $|d(x,y)-d(x',y')|<\epsilon$; and
\item For every $x \in K$, $x' \in K'$, and $\gamma \in P$, if $\gamma(x) \in K$, and $x \mathcal{R} x'$, then $\gamma(x') \in K'$ and $\gamma(x) \mathcal{R} \gamma(x')$.
\end{enumerate}
We will think of the relation $\mathcal{R}$ as a bijection $\varphi \colon K \to K'$, and write $(Z',K') \sim_{\varphi, \epsilon} (Z,K)$ to denote that  $(Z',K')$ is $\epsilon$-related to $(Z,K)$ via the relation $\varphi$. We will also write $(Z',K') \sim_{\epsilon} (Z,K)$ when we don't need to be explicit about the relation.
Now, given $Z \in \mathcal{Z}$, $K \subset Z$ a compact subset and $\epsilon>0$, we define the subset $V(Z, K, P, \epsilon)$ to be
\[
\{ Z' \in \mathcal{Z} : (Z',K') \sim_{\epsilon} (Z,K)
\}
\]
Using coboundedness of the action of $\pi_1(X)$, the same proof as in \cite[Proposition~4.1]{Pa89:Paulin1989TheGT} shows that if one takes the sets $K$ to be finite in the above family of neighborhoods $V(Z,K,P,\epsilon)$, the induced topology is equivalent.
\end{definition}

We recall that the space of geodesic currents is equipped with the weak$^*$ topology~(see Section~\ref{subsec:weakstar}).

We will use the following family of basis of neighborhoods which induces a topology coarser than the weak$^*$-topology of geodesic currents.

\begin{proposition}
Let \[
A = \{ \overline{x} \in \wt{X} \colon \mbox{ for all } \nu \in \Curr(X), \nu(G[\overline{x}])=0 \}.
\]
For $\mu \in \Curr(X)$, let $\mathcal{W}_{\wt{X}}^{\mu}$ denote the family of sets 
\[
W(\mu, C, \epsilon) = \{ \nu \in \Curr(X) \colon |\nu(G[\overline{x},\overline{y}]) - \mu(G[\overline{x},\overline{y}])| < \epsilon \colon \overline{x},\overline{y} \in C \}
\] where $\epsilon>0$,
and $C \subset \wt{X}$ ranges over all finite subsets of the subset $A \subset \wt{X}$. 
Then subset $A \subset \wt{X}$ is dense and
$\mathcal{W}_{\wt{X}}^{\mu}$ is a subbasis of neighborhoods at $\mu$ generating a topology coarser than the weak$^*$-topology in $\mathcal{G}(\wt{X})$.
\label{prop:weakflow}
\end{proposition}
\begin{proof}
We relate this family of neighborhoods to the topology generated by flow boxes for geodesic currents as in measures on the projective tangent bundle $PT(X)$ of $X$, as described, for example, in \cite[Lemma~3.4.4]{AL17:HyperbolicStructures} or \cite{Bonahon86:EndsHyperbolicManifolds}.
We note that, in the flow box topology of $PT(X)$, an $H$-shape $(\tau_L,\gamma,\tau_R)$ is determined by a pair of arcs $\tau_L$, $\tau_R$ on $X$, that we will furthermore assume to be geodesic, as well as another geodesic segment $\gamma$ on $X$, transverse to both $\tau_L$ and $\tau_R$, with one endpoint on the first and the other in the second. An $H$-shape then consists of all geodesic arcs on $X$ homotopic to $\gamma$ and transverse to $\tau_L$ and $\tau_R$.
An $H$-shape $H$ defines a subset of $PT(X)$ by considering $B_H$, the set of lifts to $PT(X)$ of  geodesic segments on $H$. A lift of $B_H$ to $PT(\wt{X})$ is then given by the set of lifts to $PT(\wt{X})$ of $B_H$, i.e., a set of geodesic segments on $\wt{X}$ with endpoints on lifts $\wt{\tau_L}$ and $\wt{\tau_R}$. These lifts can be uniquely extended to bi-infinite geodesics, obtaining $G(\tau_L,\tau_R)$, the set of geodesics intersecting both $\tau_L$ and $\tau_R$.
This is the set of bi-infinite geodesics
$G_H=G(\tau_L) \cap G(\tau_R) \subset \mathcal{G}(\wt{X})$.
By \cite[Lemma~4.2]{Bonahon86:EndsHyperbolicManifolds}, it follows that the family of sets
\[
W(\mu,G_H,\epsilon)=\{ \nu \in \Curr(X) : |\nu(G_H)-\mu(G_H)| < \epsilon \colon G_H \subseteq \mathcal{G}(\wt{X}) \text{ and } \nu(\partial G_H) = 0 \}
\]
as $\epsilon>0$ ranges over all positive values, and the sets $G_H$ range over all the sets of geodesics determined by flow boxes $B_H$, is a subbasis at $\mu$ inducing a topology which, a priori, is coarser than that of the weak$^*$-topology. Indeed, the condition $\nu(\partial B_H)=0$ in Bonahon's result is equivalent to the condition $\nu(\partial G_H) = 0$.
We finally note that $W(\mu, C, \epsilon)$ is a finite intersection of sets of type $W(\mu,G_H,\epsilon)$, and thus also a neighborhood. We now address the density of $A \subset \wt{X}$. We note that, by Proposition~\ref{prop:zeromeasure}, $\nu(\partial G[\overline{x},\overline{y}]) \neq 0$ if and only if there is a lift of a closed geodesic in the support of $\nu$ passing through $\overline{x}$ or $\overline{y}$. Let $B$ denote the union of the lifts of all closed geodesics in $X$. We will show that the complement of $B$ is dense. Firstly, we notice that $B$ is the projection to the surface of a countable union of disjoint Liouville measure zero sets in the unit tangent bundle. Indeed, these correspond to the periodic orbits of the geodesic flow, $\tilde{B}$. Since $\tilde{B}$ is geodesic flow invariant, by ergodicity of the geodesic flow its complement must have full Liouville measure, and thus must be dense (since the Liouville measure has full support). This argument also works with any geodesic current of full support in place of the Liouville current.
Finally, we observe that $A$ contains the complement of $B$, so $A$ is a dense subset of points of $\wt{X}$.
\end{proof}

Let $\mathcal{D}(X) \subset \mathcal{Z}$ be subset of dual spaces of geodesic currents of $X$ (it is a subset of $\mathcal{Z}$ by Proposition~\ref{prop:dualishyp}), equipped with the subspace topology inherited from the equivariant Gromov-Hausdorff topology.
Let $\Curr(X)$ denote the space of geodesic currents, equipped with the subspace weak$^*$-topology.

\begin{theorem}
\label{thm:homeo}
    The map $\Psi \colon \Curr(X) \to \mathcal{D}(X)$ given by $\mu \mapsto X_{\mu}$ is a continuous injection.
\begin{proof}
\begin{enumerate}
\item Continuity of $\Psi$. Let $X=X_{\mu}$, and $K = \{ x_1,\cdots,x_n \} \subset X$, $P=\{ g \}$. Given a standard neighborhood of the equivariant Gromov-Hausdorff topology  $U(X,K,P, \epsilon)$, we find a neighborhood $W$ of the subbasis introduced in Proposition~\ref{prop:weakflow}. We recall that, a priori, the topology it generates is coarser than the weak$^*$-topology, but this will be enough to prove continuity. We shall now construct such a $W$ so that $\Psi(W) \subset U$.
First of all, by density of the set $A$ in Proposition~\ref{prop:weakflow}, we can, if necessary, replace the points $x_i$ in $K$ by points $x_i^*$ in $K^* \subset A$ so that $|\mu(G(x_i,x_j))-\mu(G(x_i^*,x_j^*))| < \epsilon/2$.
Let $\alpha_1,\cdots,\alpha_n$ denote all the geodesic arcs of the type $[\overline{x},\overline{y}] \subset \wt{X}$, where $x, y \in K^*$, $x \neq y$, and $\pi_{\mu}(\overline{x})=x$ and $\pi_{\mu}(\overline{y})=y$.
Note that, if needed, we can also add more points to $K$ (and thus in $K^*$) so that we get at least two distinct $\alpha_i$'s (we need at least 3 distinct points in $K^*$).
Let $W = \cap_{i=1}^n W_{\mu}(\alpha_i,\epsilon/2)$. 
We can then write
\[
\cap_i G(\alpha_i)=\cap_{i \neq j} G(\alpha_i,\alpha_j)
\]
where $G(\alpha_i,\alpha_j)$ denotes the set of geodesics intersecting $\alpha_i$ and $\alpha_j$ transversely, as in~Proposition~\ref{prop:weakflow}. 
Thus, we can write
$W = \cap_{i \neq j} W_{\mu}(\alpha_i,\alpha_j,\epsilon/2)$, where
\[
W_{\mu}(\alpha_i, \alpha_j,\epsilon) = \{ \mu' \in \Curr(X) : |\mu(G(\alpha_i,\alpha_j))-\mu'(G(\alpha_i,\alpha_j))|<\epsilon/2 \}.
\]
Let $\mu' \in W$.
For every $x \in K \subset X_{\mu}$, pick one $\overline{x} \in \pi_{\mu}^{-1}(x) \subset \wt{X}$, and let $\overline{K}$ be the finite set consisting of one $\overline{x} \in \pi^{-1}(x)$ for each $x \in  K$.
Finally, let $K'=\pi_{\mu'}(\overline{K})$, where we denote $\pi_{\mu'}(\overline{x})=x'$.
We claim this defines an $\epsilon$-relation between $x \in K$ and $x' \in K'$.
 We start by proving the first condition of Definition~\ref{def:equivGH}.

Note that by
\[
|\mu'(G(\overline{x}^*,\overline{y}^*))-\mu(G(\overline{x}^*,\overline{y}^*))| < \epsilon/2
\]
and the choice of $x_i^* \in K^*$ in relation to $x_i \in K$, it follows that
\[
|d_{\mu'}(x',y')-d_{\mu}(x,y)| < \epsilon.
\]
Thus, the first condition of the $\epsilon$-relation is satisfied.
As for the second condition, $g(x) \in K$ means, by equivariance of $\pi_{\mu}$, that $g(\pi_{\mu}(\overline{x}))=\pi_{\mu}(g \overline{x}) \in K$.
Thus,
\[
g \overline{x} \in \pi_{\mu}^{-1}(K),
\]
and hence, by equivariance of $\pi_{\mu'}$, we have
\[
g(x')=\pi_{\mu'}(g \overline{x}) \in K'
\]
which shows that $g(x) \mathcal{R} g(x')$ and $g(x') \in K'$, as we wanted.
Thus, $X_{\mu'} \in U(X,K,P,\epsilon)$.
This shows continuity of $\Psi$ with respect to the topology in $\Curr(X)$ induced by the subbasis of Proposition~\ref{prop:weakflow}. Since that topology is coarser than the weak$^*$-topology, this proves continuity of $\Psi$.

\item Injectivity of $\Psi$. If $X_{\mu}=X_{\mu'}$, then in particular  $\ell_{\mu'}([g])=\ell_{\mu'}([g])$ for all $g \in \pi_1(X)$, i.e., by Lemma~\ref{lem:lenvsint} $i(\mu, g) = i(\mu',g)$ for all $g \in \pi_1(X)$, and thus $\mu=\mu'$.
\end{enumerate}
\end{proof}
\end{theorem}

It seems plausible that the weak$^*$-topology is, in general, strictly finer than the equivariant Gromov-Hausdorff topology.
On the other hand, the subspace topology in the of duals of measured laminations yields an embedding (see the next subsection~\ref{subsec:paulin}). 

\subsection{Relation to Paulin's work on $\mathbb{R}$-trees}
\label{subsec:paulin}
Theorem~\ref{thm:homeo} can be seen as a generalization of Paulin's result~\cite[Main~Theorem]{Pa89:Paulin1989TheGT}.

Let $G$ be a finitely generated group. The action of $G$ on an $\mathbb{R}$-tree $T$ is said to be \emph{minimal} if the only
invariant subtrees are $\emptyset$ and $T$.
An \emph{end} of an $\mathbb{R}$-tree $T$ is an equivalence class of rays in $T$, with two rays identified
if their intersection is a ray. The action of $G$ on $T$ is said to be \emph{irreducible}
if there is no end of $T$ fixed by every element of $G$.
The action is said to be \emph{reducible} when it is not irreducible. 
Let $\mathcal{T}(G)$
 the set of equivalence classes of $\mathbb{R}$-trees, containing more than one point, endowed with a minimal irreducible action of $G$, where two $\mathbb{R}$-trees are identified whenever there is an isometry from one onto the other
commuting with the actions.

Paulin defines the \emph{axial topology} (or \emph{translation length topology}) on $\mathcal{T}(G)$ by the family of neighborhoods
\[
V(P,T,\epsilon)=\{ T' \in \mathcal{T} \colon | \ell_T(g)-\ell_{T'}(g)| \mbox{ for all } g \in P \},
\]
as $P$ ranges over all finite subsets of $G$, and $\epsilon>0$. One can define a similar topology for any collection of metric spaces $\mathcal{D}$ acted on by a fixed group of isometries $G$.
One can also endow $\mathcal{T}(G)$ with the equivariant Gromov-Hausdorff topology.
When $G=\pi_1(X)$ is a surface group, and one restricts furthermore to small actions of $G$, Skora showed~\cite[Theorem~3.3]{S90:GeometricAction} (see also \cite[13.6]{Hub22:Vol3}) that
 $\mathcal{T}(G)$ corresponds to $\mathcal{D}_{\ML(X)}$, the subset of dual spaces $X_{\lambda}$ where $\lambda \in \ML(X)$. 
Therefore, by Proposition~\ref{lem:lenvsint} and \cite[Main~Theorem]{Pa89:Paulin1989TheGT}, it follows that the restriction of our map $\Psi$ to $\ML(X)$,
\[
\Psi \colon \ML(X) \to \mathcal{D}_{\ML}(X),
\]
is a homeomorphism onto its image, where $\ML(X)$ is equipped with the subspace weak$^*$-topology, and $\mathcal{D}_{\ML}(X)$ is equipped with the equivariant Gromov-Hausdorff topology.
Our result is thus a partial generalization of Paulin's work.
On the other hand, Paulin's result is more general than ours, since it applies to all irreducible and minimal $G$-actions on $\mathbb{R}$-trees, for any finitely generated group $G$.

\subsection{Relation to Cantrell--Oreg\'on-Reyes and Sapir's work}

Our map $\Psi$ descends to a continuous injection
\[
\mathbb{P}\Psi \colon \mathbb{P}\Curr(X) \to \mathbb{P}\mathcal{D}(X),
\]
where $\mathcal{D}(X)$ is equipped with the equivariant Gromov Hausdorff topology. Here $\mathbb{P}\mathcal{D}(X)$ denotes the projectivization of $\mathcal{D}(X)$, where $X_{\mu}$ and $X_{\nu}$ are equivalent if and only if there exists a constant $C>0$, so that 
\[d_{\mu}(\pi_{\mu}(x), \pi_{\mu}(y))=C \cdot d_{\mu'}(\pi_{\mu'}(x), \pi_{\mu'}(y))\] for all $x,y \in \wt{X}$.
We note that even though $\mathbb{P}\Curr(X)$ is a compact metrizable space, $\mathbb{P}\mathcal{D}(X)$ with the equivariant Gromov Hausdorff is not Hausdorff in general (see~\cite[Chapter~4]{Pau88:Thesis}), so one cannot conclude the induced continuous injection is an embedding.

The subspace $\mathbb{P}\Curr_{\mathit{fill}}(X) \subseteq \mathbb{P}\Curr(X)$ consisting of filling geodesic currents can be equipped with a distance, by recent work of Sapir~\cite{Sap22:ExtensionThurston}, defined as follows
\[
d_{\mathit{fill}}([\mu],[\nu]) \coloneqq \sup_{c \in \Curves(X)} \log \frac{i(\mu,c)}{i(\nu,c)} + \sup_{c \in \Curves(X)} \log \frac{i(\nu,c)}{i(\mu,c)}.
\]
This distance coincides with the symmetrization of the Thurston distance on Teichmuller space, when restricted to $\Curr_{\Teich(X)}(X)$, i.e., the embedded image of Teichm\"uller space in $\mathbb{P}\Curr(X)$ by~\cite[Theorem~8.5]{Th98:ThurstonMetric}.
On the other hand, there is also a natural distance 
$d_{\mathcal{D}}$ on $\mathbb{P}\mathcal{D}_{\mathit{fill}}(X)$, defined in recent work of Oreg\'on-Reyes~\cite[Definition~1.2]{OR22:SpaceMetric}, which can also be expressed, by~\cite[Lemma~3.5]{OR22:SpaceMetric}, as
\[
d_{\mathcal{D}}([\mu],[\nu]) \coloneqq \sup_{g \in \pi_1(X)} \log \frac{\ell_{X_\mu}(g)}{\ell_{X_{\mu'}}(g)} + \sup_{g \in \pi_1(X)} \log \frac{\ell_{X_{\mu'}}(g)}{\ell_{X_{\mu}}(g)}.
\]
Indeed, by Proposition~\ref{prop:quasi-isom}, $\mathbb{P}\mathcal{D}_{\mathit{fill}}(X)$ corresponds to a subspace of the space $\mathcal{D}(\pi_1(X))$ (this also follows from~\cite[Theorem~1.11]{COR22:Manhattan}), and one can endow it with the restriction of that metric.
By~\cite[Lemma~3.5]{OR22:SpaceMetric}, and Proposition~\ref{lem:lenvsint}, it follows that the restriction of $\Psi$ to $\mathbb{P}\Curr_{\mathit{fill}}$,
\[
\mathbb{P}\Psi \colon \mathbb{P}\Curr_{\mathit{fill}} \to \mathbb{P}\mathcal{D}_{\mathit{fill}}(X)
\]
is an isometry with respect to the metrics $d_{\mathit{fill}}$ and $d_{\mathcal{D}}$.
Furthermore, the topology induced by $d_{\mathit{fill}}$ coincides with the subspace weak$^*$-topology, by~Theorem~\ref{thm:weakintersection}. The topology induced by $d_{\mathcal{D}}$ coincides with the topology of translation lengths or axial topology on $\mathbb{P}\mathcal{D}_{\mathit{fill}}(X)$ 
by~\cite[Lemma~3.5]{OR22:SpaceMetric}.
Since by~\cite[Lemma~3.5]{EM18:ErgodicGeodesicCurrents} and~\cite[Corollary~3.8]{EM18:ErgodicGeodesicCurrents} $\mathbb{P}\Curr_{\mathit{fill}}(X)$ is dense and open in $\mathbb{P}\Curr(X)$, then the extension of $\mathbb{P}\Psi|_{\mathbb{P}\Curr_{\mathit{fill}}(X)}$ to the closure is the same as $\Psi$.

\appendix
\section{Equivalence of notions of tree graded spaces}
\label{sec:treegraded}
In this section we show that, under the assumption that $X$ is geodesic, the Drutu-Sapir definition of tree graded space as in Definition~\ref{def:treegraded_ds} and our Definition~\ref{def:treegraded_drmg} are equivalent.

\begin{proposition}
If $(X,d)$ is a geodesic metric space, then Definition~\ref{def:treegraded_drmg} is equivalent to Definition~\ref{def:treegraded_ds}.
\label{prop:eq_defs_treegraded}
\end{proposition}
\begin{proof}
Assume throughout that $(X,d)$ is geodesic metric space and also that it has a collection of pieces that intersect in at most one point, i.e., \emph{axiom pieces}.
Let $\Delta=xyz$ be a geodesic triangle with geodesic sides $\gamma^1=\gamma_{xy}$, $\gamma^2=\gamma_{yz}$ and $\gamma^3=\gamma_{xz}$.
Suppose $\Delta$ is contained in more than one piece.
We want to show $\Delta$ is non-simple.
We can decompose each $\gamma^i$, according to Proposition~\ref{prop:geodesictreegraded}, into a piece-wise geodesic consisting of geodesic segments $\alpha^i_j$ contained in pieces (we will call them \emph{$\alpha$-segments}) and geodesic segments $\beta^i_k$ each contained in a transversal $T_x$ (we will call them \emph{$\beta$-segments}), for $i=1,2,3$.
The endpoints of the $\beta$ for each $i$, give corresponding straight contact chains $(y^i_1,\cdots,y^i_{\ell_i})$.
Since $\Delta$ is contained in more than one piece, at least one of the the $\beta$ segments are non-degenerate.
  By \emph{axiom straight chain triangle}, one of the straight chains shares an interior point with another side chain say $y^1_i=y^3_j$. This means that their corresponding subsegments $\beta^1_i$ and $\beta^3_j$ share an endpoint, and thus $\Delta$ is non-simple.

\begin{figure}[htp]
     \centering
          \begin{subfigure}[b]{0.4\textwidth}
\centering{
\resizebox{80mm}{!}{\Huge{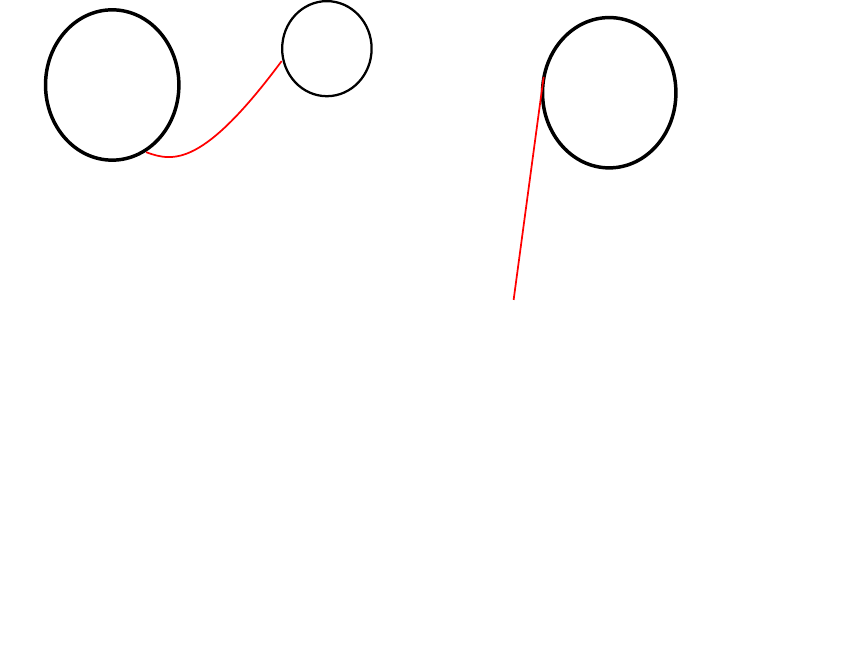}}
\label{fig:triangle_case3}
}
     \end{subfigure}
     \begin{subfigure}[b]{0.4\textwidth}

\centering{
\resizebox{80mm}{!}{\Huge{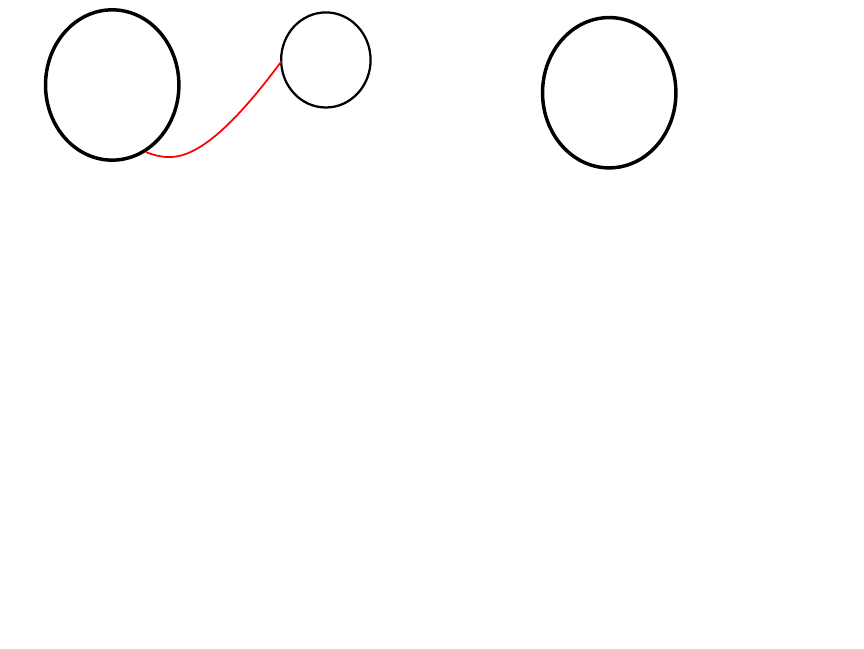}}
\label{fig:triangle_case3bis}
}
     \end{subfigure}

     \caption{From triangles to straight chain triangles: Proof of implication from Definition~\ref{def:treegraded_ds} to Definition~\ref{def:treegraded_drmg}. From a straight chain we generate a triangle with simple subsegments. The axiom triangles forces self-intersection of the triangles, giving two cases depending on which types of subsegments intersect. The left figure exhibits the cases of two $\beta$-segments intersecting, and the right one the case of two $\alpha$-segments intersecting }
     \label{fig:equivalence}
\end{figure}

Now we show that Definition~\ref{def:treegraded_ds} implies Definition~\ref{def:treegraded_drmg}.
Given \emph{axiom triangles}, the (axiom transversals) is~\cite[Lemma~2.12]{DS05:Treegraded} and~\cite[Lemma~2.13]{DS05:Treegraded}.

We will see now that the axiom contact chain triangle also follows from \emph{axiom triangles}.

Let $\Delta=xyz$ be a contact chain triangle, where $x,y,z$ are three points in $X$ contained in at least two distinct pieces. Suppose, say, that $x \in P_1$, and let $y \in P_2$ and $z \in P_3$, where $P_2$ and $P_3$ could potentially be the same, but both distinct from $P_1$. Denote by $C_i \coloneqq (y^i_1,\cdots,y^i_{\ell_i})$ the straight chains connecting the points $x,y,z$ pairwise, for $i=1,2,3$.
We can connect the contact points in these chains by simple geodesic segments that we will call, as before, $\alpha^i_j$ and $\beta^i_k$, depending on whether their interiors are in pieces or in $T_x$, respectively. Let $\hat{C_i}$ denote the corresponding piece-wise geodesic realizing the straight chain $C_i$. They together form a corresponding geodesic triangle $\hat{\Delta}$ which is contained in more than one piece, so it must be non-simple. Thus, there must be a point of intersection $z$ between, say, the interior of $\hat{C_1}$ and a point in $\hat{C_3}$.
The geodesics $[z, y^1_1]$ and $[z, y^3_1]$ both intersect the piece $P_1$ only at their endpoints $y^1_1$ and $y^3_1$, and thus, by~\cite[Lemma~2.4]{DS05:Treegraded}, we have $y^1_1=y^3_1$.
Similarly, the geodesics $[z, y^1_{\ell_1}]$ and $[z, y^2_{\ell_3}]$ both intersect, say $P_2$, at $y^1_{\ell_1}$ and $y^2_{\ell_3}$, so $y^1_{\ell_1}=y^2_{\ell_3}$. Then, a similar argument shows that $y^3_{\ell_1}= y^2_{1}$. Thus, since $\alpha^i_j$ are all simple geodesics, it follows that $\alpha^1_1$, $\alpha^2_1$ and $\alpha^1_{\ell_1}$ must be singletons and equal to the points of contact $y_1, y_2, y_3$ at $P_1, P_2$ and $P_3$, respectively.
Then, the $\beta$ segments adjacent to each $y_i$, since they share $y_i$, must be contained in the same $T_x$.
If $z$ is in a $\beta$ segment, 
a similar argument using geodesics emanating from $z$ and Lemma~\cite[Lemma~2.4]{DS05:Treegraded}, as before, shows that all other $\beta$ segments are contained in the same $T_x$, and thus $\Delta=\Delta_1 \cup A$, where $\Delta_1$ is a triangle made of $\beta$ segments and contained in $T_x$ and $A$ is a disjoint union of $\alpha$-segments.
If $z$ is in an $\alpha$ segment, a similar argument shows that $\Delta=\Delta_1 \cup A \cup B$, where $\Delta_1$ is a triangle made of $\beta$ segments contained in a single $T_x$, $A$ is a union of $\alpha$ segments, and $B$ is a union of $\beta$ segments. 
In any of the two cases, since $T_x$ is $0$-hyperbolic and $\Delta_1$ is contained in $T_x$, the triangle $\Delta_1$ must be non-simple, and thus there must be two $\beta$ subsegments corresponding to different sides of $\Delta$ sharing an endpoint, which shows that one of the corresponding straight contact chains shares an interior point with another, thus contradicting the axiom contact chain triangle.
\end{proof}

\bibliographystyle{hamsalpha}
\bibliography{main}

\end{document}